\documentclass[a4paper,10pt]{article}
\usepackage[utf8]{inputenc}

\usepackage[T1]{fontenc}
\usepackage[english]{babel}

\usepackage[svgnames]{xcolor}
\usepackage[margin=1.2in]{geometry}

\usepackage{graphicx}
\usepackage{wrapfig}

\usepackage{authblk}

\usepackage{subcaption}

\usepackage{enumitem}

\usepackage[colorlinks=true,linkcolor=blue,citecolor=FireBrick,urlcolor=red,pagebackref=true]{hyperref}

\usepackage{amsmath}
\usepackage{amssymb}
\usepackage{mathrsfs}
\usepackage{amsthm}
 
\usepackage{bm}
\usepackage{bbm}    
\usepackage[scr=boondox]{mathalpha}

\usepackage[ruled]{algorithm2e}
\usepackage{tikz}       

\usepackage{macros}

\title{Cutoff for mixtures of permuted Markov chains: \\ reversible case}
\author{Bastien Dubail \thanks{Correspondence to be sent to: \href{mailto:bastien.dubail@bristol.ac.uk}{bastien.dubail@bristol.ac.uk}} \\ Département d'Informatique de l'ENS, École Normale Supérieure, CNRS, \\ PSL Research University \\
INRIA Paris \\
CNRS, I2M, Aix-Marseille Université}
\date{}

\begin{document}

\maketitle

\begin{abstract}
	We investigate the mixing properties of a model of reversible Markov chains in random environment, which notably contains the simple random walk on the superposition of a deterministic graph and a second graph whose vertex set has been permuted uniformly at random. It generalizes in particular a result of Hermon, Sly and Sousi, who proved the cutoff phenomenon at entropic time for the simple random walk on a graph with an added uniform matching. Under mild assumptions on the base Markov chains, we prove that with high probability the resulting chain exhibits the cutoff phenomenon at entropic time $\log n / h$, $h$ being some constant related to the entropy of the chain. We note that the results presented here are the consequence of a work conducted for a more general model that does not assume reversibility, which will be the object of a companion paper. Thus an important contribution of this paper is in the arguments we propose, as most of which do not require reversibility. Among these, we establish a novel concentration result for "low-degree" functions on the symmetric group, established specifically for our purpose but which could be of independent interest. 
\end{abstract}

\section{Models and main results}

\subsection{Cutoff for mixtures of reversible permuted Markov chains}

This paper establishes a cutoff phenomenon at entropic time for a model of Markov chains in random environment. In the simplest case, think of two multi-graphs, allowed to contain self-loops and multi-edges, which are superposed one on top of the other. What can be said about the simple random walk on the resulting graph ? Does it mix faster ? In general nothing can be said, as the two graphs could be equal and give the same resulting random walk. On the opposite, if the two graphs are not perfectly aligned but rather superposed in a complementary way the random walk can be expected to behave very differently. In this paper we consider a random version of this process, where the vertices of the second graph are permuted uniformly at random. This model is inspired by the work \cite{hermon2020universality} of Hermon, Sousi and Sly, who proved the cutoff phenomenon at entropic time for the simple random walk on a sequence of deterministic graphs to which one adds a random uniform matching of the vertices. In this paper we consider a more general model which goes beyond the case of the simple random walk on the superposition of graphs and considers mixtures of reversible Markov chains. To define this model, we use the well-known theory of representing reversible Markov chains by electrical networks. 

First recall that a chain on state space $S$ with transition kernel $P$ is reversible if there exists a measure $\pi$ on $S$ such that $\pi(x) P(x,y) = \pi(y) P(y,x)$ for all $x,y \in S$. An electrical network is a pair $(G,c)$ consisting in a weighted non-directed graph $G=(V,E)$ equipped with non-negative weights $c=(c(e))_{e \in E}$, called conductances, on the edges. Any reversible Markov chain can be represented as a random walk on an electrical network, defining conductances as 
\begin{equation}\label{eq:conductances}
	c(x,y) := \pi(x) P(x,y)
\end{equation}
 which are symmetric by assumption. Conversely, any electrical network gives rise to a reversible Markov chain whose transition probabilities are proportional to conductances. We refer to \cite{lyons2016probability} for a detailed account of this theory. The particular case of the simple random walk on a multi-graph is obtained by taking all conductances equal to $1$. For more general reversible chains, the electrical network theory provides a natural and generic way to mix together two reversible chains by superpositioning the corresponding electrical networks, that is taking a linear positive combination of the conductances. 
 
Let us introduce some notation to state our main result.
Given two measures $\mu, \nu$ on a countable set $S$, their total variation distance is defined as 
\begin{equation*}
	\TV{\mu - \nu} := \sup_{A \subseteq S} \abs{\mu(A) - \nu (A)} = \frac{1}{2} \sum_{x \in S} \abs{\mu(x) - \nu(x)}.
\end{equation*}
If $P$ is the transition kernel of a positive recurrent, irreducible and aperiodic Markov chain on $S$, it admits a unique invariant measure $\pi$. In that case, given a starting vertex $x \in S$ and $\e \in (0,1)$, the mixing time is defined as 
\begin{equation*}
	\tmix(x, \e) := \inf \{ t \geq 0: \TV{P^{t}(x, \cdot) - \pi} < \e \}. 
\end{equation*}
If the chain is not irreducible or aperiodic, we consider the mixing time to be infinite. An event $A=A(n)$ is said to occur with high probability, if it has limiting probability $1$ as $n \rightarrow \infty$. Given an integer $n \geq 1$, we write $[n] := \{1, \ldots n \}$.

\begin{theorem}\label{thm:reversible}
Let $n \geq 1$ be an integer, $\sigma$ a permutation of $n$ elements chosen uniformly at random, $(G_1,c_1), (G_2,c_2)$ two electrical networks with common vertex set $[n]$ and $\alpha, \beta > 0$. Then consider the Markov chain on $[n]$ defined by the electrical network $(G^{\ast}, c^{\ast})$ with conductances
\begin{equation}\label{eq:reversible_model}
	\forall x,y \in [n]: c^{\ast}(x,y) := \alpha \, c_1(x,y) + \beta \, c_2(\sigma(x), \sigma(y)).
\end{equation}
Suppose 
	\begin{enumerate}[label=(H\arabic*)]
		\item \label{hyp:bdd_delta}The degrees and conductances of $(G_1,c_1)$ and $(G_2,c_2)$ are all bounded uniformly in $n$, from above and away from $0$. 
		\item \label{hyp:bdd_p} $\alpha, \beta$ are constants in $n$.
		\item \label{hyp:cc3} The connected components of $G_1$ have size at least $3$ and that of $G_2$ have size at least $2$.
	\end{enumerate}
	Then there exists $h=h(n)$ bounded from above and away from $0$ for which the following holds. For all $\e \in (0,1)$, there exists a constant $C(\e)$ such that with high probability,
	\begin{align*}
		\min_{x \in [n]} \tmix(x,1-\e) &\geq \frac{\log n}{h} - C(\e) \sqrt{\log n}, \\
		\max_{x \in [n]} \tmix(x,\e) &\leq \frac{\log n}{h} + C(\e) \sqrt{\log n}.
	\end{align*}
	In particular, with high probability the chain is irreducible, aperiodic and exhibits a uniform cutoff phenomenon: for all $\e \in (0,1)$  
	\begin{equation*}
		\lim_{n \rightarrow \infty} \frac{\max_{x \in [n]} \tmix(x, \e)}{\min_{x \in [x]} \tmix(x,1- \e)} = 1
	\end{equation*}
	in probability.
\end{theorem}

\begin{remark}
	Since the chain in the theorem is defined via conductances, it is automatically reversible. Given two multi-graphs $G_1, G_2$, taking unit conductances ($c_1, c_2 \equiv 1$) and setting $\alpha = \beta = 1$ gives the particular case of the simple random walk on the superposition of $G_1$ and $G_2$. The case of a graph with an added random matching studied in \cite{hermon2020universality} can then be obtained by taking for $G_2$ a sequence of edges if $n$ is even. Secondly we remark that we consider multi-graphs, so edges of $G_1, G_2$ that align under the permutation $\sigma$ would result in a transition with a higher probability. We could obtain rigorously the case of simple graphs by adjusting $\alpha, \beta$ to make them $1/2$ when edges of the two networks are aligned. However a close inspection at the proof shows this makes no difference (see Remark \ref{rk:simple_graphs}), so our result is also true when superposing simple graphs and discarding multi-edges in the resulting graph.
\end{remark}

\paragraph*{About reversibility}\label{rk:general_model}
	The previous theorem was obtained as the consequence of a work conducted for more general chains than those considered here, that does not assume reversibility. The general model is the following: let $P_1 , P_2$ be two  $n \times n$ stochastic matrices, $p_1, p_2 \in M_{n}([0,1])$ two $n \times n$ matrices with entries in $[0,1]$ satisfying $p_1 + p_2 \equiv 1$ and consider the stochastic matrix
	\begin{equation}\label{eq:model0}
		\scP(x,y) := p_1(x, \sigma(x)) \, P_1(x,y) + p_2(x, \sigma(x)) \, P_2(\sigma(x), \sigma(y)).
	\end{equation}
	The reversible model considered in this paper is a particular case of \eqref{eq:model0}: supposing $P_1, P_2$ are reversible and correspond to electrical networks $(G_1,c_1), (G_2,c_2)$ respectively, $P_i(x,y) = c_i(x,y) / c_i(x)$ where $c_i(x) := \sum_{z \in V} c_i(x,z)$ for $i=1,2$. Hence the reversible model is realized as \eqref{eq:model0} taking
	\begin{equation*}
		p_1(x,y) := \frac{\alpha c_1(x)}{\alpha c_1(x) + \beta c_2(y)}.
	\end{equation*}
	The general model \eqref{eq:model0} is studied in the companion paper \cite{cutoff_general}. We chose here to focus first on the reversible case as it simplifies several aspects of the proof. Our proof adapts the strategy used for non-backtracking chains \cite{benhamou2017cutoff,bordenave2018random,bordenave2019cutoff} to handle backtracking chains as well, including reversible ones, but reversibility is by far not essential. In particular, we do not use spectral arguments like Poincaré's inequality, traditionally used when working with reversible Markov chains. Thus an important goal of the present paper is to provide a unified approach to prove cutoff for Markov chains in random environment using the "entropic method", with or without reversibility. Nevertheless, the absence of reversibility brings real additional difficulties and incidentally makes the conclusion Theorem \ref{thm:reversible} false: for some choices of $P, Q, p_1, p_2$ the worst-case mixing time is of order $(\log n)^{a}$ with $a > 1$. What remains true however is that cutoff occurs at entropic time if the chain is started from a typical state. This phenomenon is similar for instance to the case of random walks on the giant component of Erd\H{o}s-R\'enyi graphs: with high probability these contain segments of length of order $\log n$, resulting in the worst-case mixing time being of order $(\log n)^{2}$ \cite{benjamini2014mixing,fountoulakis2008evolution}, while the typical mixing time is $O(\log n)$ \cite{berestycki2018random}. This difference between worst-case and typical mixing time in the general case creates many additional technical difficulties but does not affect the overall arguments. Focusing on the reversible case in this paper will hence allow us to be free of these technical considerations and focus on the core ideas. We refer the reader to the proof outline of Section \ref{subsec:outline} for a more detailed account of how reversibility is used.

\subsection{A concentration inequality for low-degree functions on the symmetric group}

A core argument in the proof is a concentration inequality for the uniform measure on the symmetric group, which as far as we know is new and might be of independent interest. It generalizes an inequality of Chatterjee \cite[Prop. 1.1]{chatterjee2007stein}, which was already used in previous works about cutoff for non-backtracking walks \cite{benhamou2017cutoff, bordenave2018random}. The original inequality takes the form of a Bernstein-like bound for random variables of the form
\begin{equation}\label{eq:deg1_rv}
	Z = \sum_{i} A_{i \sigma(i)}
\end{equation}
where $A \in M_{n}(\bR_+)$ and $\sigma$ is a uniform permutation of $n$ elements. Namely, for all $t \geq 0$
\begin{equation}\label{eq:concentration_1d}
	\bP \sbra{\abs{Z - \bE \sbra{Z}} \geq t} \leq 2 \exp \left( \frac{- t^2}{2 \norm{A} (t + 2 \bE \sbra{X})}\right),
\end{equation}
where $\norm{A} = \max_{i,j \in [n]} A_{i,j}$. The random variable $Z$ can be seen as arising from a linear function on the symmetric group, in the sense that it is a linear combination of indicators $\II_{\sigma(i) = j}$, which are the entries of the matrix representation of $\sigma$. With this point of view, it seems natural to inquire about more general polynomial functions. Of course any function on the symmetric group can be represented as a polynomial of degree $n$ and in fact $n-1$. Furthermore there already exist concentration results for generic functions on the symmetric group, regardless of the degree: Maurey's inequality \cite{maurey1979construction} (see also Thm. 2.14 in \cite{chatterjee2005thesis}), Talagrand's inequality \cite{talagrand1995concentration}[Thm 5.1], and Proposition 4.8 in \cite{chatterjee2005thesis} are such examples. Our motivation is thus mainly to investigate whether an additional assumption of "low degree" can yield better concentration inequalities. 

Some notations are necessary to state our result. Let $\frS_n$ denote the symmetric group on $n$ elements. Permutations $\sigma \in \frS_n$ are identified with permutation matrices $S$ defined by $S_{ij} := \II_{\sigma(i) = j}$. Let us remark first that there is no unique representation of a function on $\frS_n$ as a polynomial. In the following result, particular representations are considered, but some more intrinsic notion of degrees will be discussed in Remark \ref{rk:proper_degree}. We can without loss of generality suppose that the constant term is zero. Furthermore, since we are restricting to permutation matrices it is enough to consider the set of functions
\begin{equation}
	\frF := \left\{ \phi: M_{n}(\bR) \rightarrow \bR \ | \ \forall i, j, k \in [n]:  \partial_{ij} \partial_{ik} \, \phi \equiv \partial_{ji} \partial_{ki} \, \phi \equiv 0 \right\},
\end{equation}
where $\partial_{ij}$ denotes the partial derivative with respect to the entry $(i,j)$. Functions of $\frF$ are called multilinear, as the degree is at most one in each entry. In particular, these are polynomial functions on $M_{n}(\bR)$, which we identify with polynomials in indeterminates $X_{ij}, i,j \in [n]$. By restriction to the set of permutations, each function $\phi \in \frF$ induces a map on $\frS_n$. Conversely, any map on $\frS_n$ can be written as $\phi_{| \frS_n}$ for some $\phi \in \frF$ of degree at most $n-1$, however this representation is in general not unique. 

The vector space $\frF$ can be decomposed as 
\begin{equation*}
	\frF = \bigoplus_{d \geq 0} \frF_{= d}
\end{equation*}
where for each $d \geq 0$, $\frF_{= d}$ is the vector space of homogeneous polynomials of degree $d$. Let $\frF_{d} = \bigoplus_{k=0}^{d} \frF_{= k}$ be the vector space of polynomials of degree at most $d$. If $d \geq 1$, a convenient way to write $\phi \in \frF_{=d}$ is given by Euler's theorem:
\begin{align*}
	\phi(X) &= \frac{1}{d} \sum_{i,j \in [n]} X_{ij} \partial_{ij} \phi(X).
\end{align*}
In particular, evaluating $\phi$ at $\sigma \in \frS_n$ yields
\begin{equation}\label{eq:phi_homogeneous}
	\phi(\sigma) = \frac{1}{d} \sum_{i \in [n]} \partial_{i, \sigma(i)} \phi(\sigma),
\end{equation}
which can be seen as a generalization of \eqref{eq:deg1_rv}.
To state our result, we need the following linear operators on $\frF$. Given a homogeneous function $\phi \in \frF_{=d}$, let
\begin{equation*}
	D \phi:= \frac{1}{d n }\sum_{i,j \in [n]} \partial_{ij} \phi \qquad U \phi(X) := \frac{1}{d n} \sum_{i,j,k,l \in [n]} X_{il} X_{kj} \partial_{ij} \partial_{kl} \phi(X).
\end{equation*}
Finally, for a function $\phi \in \frF$, let 
\begin{equation*}
	\norm{\phi}_{\infty}:= \max_{\sigma \in \frS_n} \abs{\phi(\sigma)}, \qquad \norm{\nabla \phi}_{\infty} := \max_{\sigma \in \frS_n} \max_{i,j \in [n]} \abs{\partial_{ij} \phi(\sigma)}.
\end{equation*}
In this context, we write $\bE \sbra{\phi} := \bE \sbra{\phi(\sigma)}$ for the expectation with respect to a uniformly distributed $\sigma \in \frS_n$.

\begin{theorem}\label{thm:tensor_concentration}
	Let $n \geq 1$, $d \in [n]$, $\sigma$ a uniform permutation of $n$ elements and $\phi \in \frF_d$ a polynomial map of degree $d$ with non-negative coefficients. Suppose there exist $C_D, C'_D, C_U \geq 0$ such that for all $k \in [0, d]$,
	\begin{equation}\label{eq:hyp_norms_DU}
		\norm{ D^{k} \phi}_{\infty} \leq C_D, \quad \norm{\nabla D^{k} \phi}_{\infty} \leq C'_D \quad \text{and} \quad \norm{ D^{k} U \phi}_{\infty} \leq C_U.
	\end{equation}
	Then for all $t \geq 0$,
	\begin{equation*}
		\bP \sbra{\phi(\sigma) - \bE \sbra{\phi} \geq t} \leq \exp \left( \frac{- t^2}{2 (\gamma_{\phi} + \beta_{\phi} t)} \right), \quad \text{and} \quad \bP \sbra{ \phi(\sigma) - \bE \sbra{\phi} \leq -t} \leq \exp \left( \frac{- t^2}{2 \gamma_{\phi}} \right)
	\end{equation*}
	where
	\begin{equation*}
		\begin{split}
		&\beta_{\phi} := 9 d C'_D \left( \log \left( \frac{4 C_D n}{C'_D} \right)^{+} + \frac{(2/n)(2-e^{-2/n})}{1-e^{-2/n}} \right)\\
		&\gamma_{\phi} := \frac{2 \beta_{\phi}}{5} \left( 2 \bE \sbra{\phi} + C_U \right).
		\end{split}
	\end{equation*}
\end{theorem}

\begin{remark}
	If $\phi(X) = \sum_{i,j} A_{i,j} X_{i,j}$, then $\nabla \phi \equiv A^{\top}$. Therefore, the condition required on $\norm{\nabla \phi}_{\infty}$ generalizes the control needed on $\norm{A}$ in \eqref{eq:concentration_1d}. Up to the logarithmic term and multiplicative constants, Theorem \ref{thm:tensor_concentration} thus really aims to generalize Chatterjee's inequality to higher degrees.
\end{remark}

\begin{remark}
	As was mentioned earlier, the finiteness of $\frS_n$ implies that any function can in fact be seen as a polynomial function of degree at most $n-1$, so Theorem \ref{thm:tensor_concentration} provides in theory a concentration inequality for any non-negative random variable defined by a uniform permutation. It is likely however that this bound becomes irrelevant when $d / n \nrightarrow 0$. In this paper, Theorem \ref{thm:tensor_concentration} will be applied with $d = O(\sqrt{\log n})$. 
\end{remark}

\begin{remark}\label{rk:proper_degree}
	There are intrinsic notions of degree for functions on $\frS_n$. A natural notion of degree is as follows: given $k \geq 1$, a $k$-coset of $\frS_n$ is a subset of the form
		\begin{equation*}
			E_{\substack{i_1 \cdots i_k \\ j_1 \cdots j_k}} := \{\sigma \in \frS_n \ | \ \forall m=1, \ldots, k \quad \sigma(i_m) = j_m \}.
		\end{equation*}
		where  $\bfi = (i_1, \dots, i_k)$, $\bfj= (j_1, \ldots, j_k) \in [n]^{k}$ are multi-indices of length $k$. If $k = 0$, we consider the whole set $\frS_n$ to be a $0$-coset. Given a function $f: \frS_n \rightarrow \bR$, the degree of $f$ can be defined as the least integer $d \geq 0$ such that $f$ writes as a linear combination of $k$-coset indicator functions with $k \leq d$. 
		With this definition the degree of $f$ is the minimal degree of a polynomial representation of $\phi$.
		
		The previous notion of degree is also intimately related to Fourier analysis: from \cite{ellis2011intersecting}[Thm. 7], it coincides with the \emph{Fourier degree} of $f$, which is the least integer $d \geq 0$ such that all Fourier coefficients of $f$ corresponding to irreducible representations of dimension $k > d$ are zero. We do not know whether the considerations of these notions of degree could lead to better concentration inequalities. The use of Fourier analysis, in particular character theory could definitely be of help at some point in the proof but other arguments seem to require more than the sole use of the characters. See Remark \ref{rk:representation_theory}.
	\end{remark}

The presence of the operators $D$ and $U$ is a consequence of the proof. It follows the method of exchangeable pairs used by Chatterjee in the one-dimensional case to bound the log-Laplace transform, which can be done in terms of the functions $D^{k} \phi, U D^{k} \phi$. In the degree one case, these operators are trivial and thus need not be considered, but in the more general case an induction argument on the degree seems necessary, which is the reason why one needs to bound $\norm{D^{k} \phi}_{\infty}, \norm{\nabla D^{k} \phi}_{\infty}, \norm{ D^{k} U \phi}_{\infty}$ for all $k \leq d$. These quantities seem to lack good monotonicity properties that could simplify the bounds. The following proposition provides some rough control on the constants $C_D, C'_D, C_U$, which will be sufficient for our purpose.

\begin{proposition}\label{prop:concentration_constants}
	Let $\phi$ be a polynomial in the indeterminates $(X_{ij})_{i,j = 1}^{n}$ of degree $d \geq 1$. Let $M(\phi)$ be the maximal coefficient of $\phi$, and $N(\phi)$ the number of monomials in $\phi$. Then for all $k \geq 0$ we have
	\begin{equation*}
		\norm{D^{k} \phi}_{\infty} \leq 2^{k} M(\phi) N(\phi), \qquad \norm{U \phi}_{\infty} \leq \frac{d-1}{n} \, M(\phi) N(\phi).
	\end{equation*}
\end{proposition}

Plugging the previous estimates into Theorem \ref{thm:tensor_concentration} yields the following corollary.

\begin{corollary}\label{coroll:concentration}
	Let $\phi \in \frF_d$ be a polynomial function of degree at most $d \geq 1$. Using the notations of the previous proposition, let $A_{\phi} := M(\phi) N(\phi)$ and $A_{\nabla \phi} := \max_{i,j \in [n]} A_{\partial_{ij} \phi}$. 
	Then for all $t \geq 0$ we have
	\begin{equation*}
		\bP \sbra{\abs{\phi(\sigma) - \bE \sbra{\phi}} \geq t} \leq 2 \exp \left( \frac{- t^2}{2 \alpha_{\phi} (\frac{8}{9} \bE \sbra{\phi} + \frac{2^{d+2}(d-1)}{9n} A_{\phi}  + t)} \right),
	\end{equation*}
	with
	\begin{equation*}
		\alpha_{\phi} := 9 \, d \, 2^{d} A_{\nabla \phi} \left( \log \left( \frac{4 A_{\phi} \, n}{A_{\nabla \phi}} \right)^{+} + \frac{(2/n)(2-e^{-2/n})}{1-e^{-2/n}}\right).
	\end{equation*}
\end{corollary}

\subsection{Related works}

Discovered by Diaconis, Shahshahani and Aldous \cite{diaconis1981generating,aldous1983random,aldous1986shuffling}, the cutoff phenomenon is a famous feature observed in a large number of Markov chains. Not all Markov chains exhibit cutoff and it remains an open question to determine a characterization of this intriguing phenomenon. A sufficient condition was given by Salez in \cite{salez2023cutoff} which expresses an entropic concentration phenomenon at the heart of many important achievements on cutoff over the last decade. While the initial focus was on specific, explicit Markov chains with high degrees of symmetry, like random walks on groups (see \cite{diaconis1996cutoff,saloff2004random} for global references on the subject), the seminal work of Lubetzky and Sly \cite{lubetzky2010cutoff} on random walks on random regular graphs initiated a series of papers studying instead generic Markov chains. In that regard, a lot of attention was drawn on Markov chains in random environment, which showed that cutoff is actually quite common and often the result of an entropic concentration phenomenon, leading to an \emph{entropic mixing time} $(\log n) / h$, where $n$ is the size of the state space and $h$ can be interpreted as an entropy rate. 

After Lubetzky and Sly proved the cutoff for simple and non-backtracking random walks on random regular graphs, Ben-Hamou and Salez proved the cutoff for the non-backtracking walk on random graphs with a given degree sequence, i.e. the configuration model \cite{benhamou2017cutoff}. This case is also considered in \cite{berestycki2018random}, in which Berestycki, Lyons, Peres and Sly prove the cutoff, for both the simple and non-backtracking walks on the giant component of an  Erd\H{o}s-R\'enyi random graph as well as for the configuration model. In the case of the simple random walk, the starting point is not uniform but needs to be typical. For  Erd\H{o}s-R\'enyi random graphs, the worst-case mixing time had been established to be $O((\log n)^{2})$ by Fountoulakis and Reed \cite{fountoulakis2008evolution} and Benjamini, Kozma and Wormald \cite{benjamini2014mixing}, while in the configuration model cutoff was obtained subsequently by Ben-Hamou, Lubetzky and Peres in \cite{benhamou2019comparing}. Models of random non-backtracking chains have also been considered by Bordenave, Caputo and Salez: \cite{bordenave2018random} considers the non-backtracking walk on directed configuration models, while \cite{bordenave2019cutoff} considers the case of a stochastic matrix in which the entries of each row are permuted uniformly at random. In \cite{conchon2022cutoff}, Conchon-Kerjan considers random walks on random lifts of weighted graphs which are not reversible. Let us also mention the work of Hermon and Olesker-Taylor \cite{hermon2021cutoff,hermon2021cutofftriangular} on random walks on random Cayley graphs of Abelian groups. A model similar to ours is the $PS$ model, introduced by Chatterjee and Diaconis \cite{chatterjee2020speeding} and shown to exhibit cutoff at entropic time by Ben-Hamou and Peres \cite{benhamou2023cutoff}. 

Some recent works investigate the case of random graphs with community structures: in \cite{benhamou2020threshold}, Ben-Hamou proves a phase transition for the cutoff of the non-backtracking random walk on a random graph with two communities. A extension of this result for the simple random walk was obtained by Hermon, Šarković and Sousi \cite{hermon2022cutoff} who also consider a second model of random graphs allowing more communities. 

While the previously cited works consider the case of essentially totally random Markov chains, cutoff was also shown to occur when randomizing a given chain, where the final environment still keeps a lot of the structure of the initial chain. We already cited the work of Hermon, Sly and Sousi \cite{hermon2020universality} as the main inspiration of this work, where cutoff is proved for the simple random walk on a graph to which a uniform matching is added. This model was extended by Baran, Hermon, Šarković and Sousi \cite{baran2023phase} who prove a phase transition when weights are added on the random matching. 

Another direction of recent works is the competition between different mechanisms such as having a dynamic environment: see the papers by Avena, Gülda{\c{s}}, van der Hofstad and den Hollander \cite{avena2019trichotomy} and Caputo and Quattropani \cite{caputo2021trichotomy}. A related work is \cite{caputo2021pagerank} by Caputo and Quattropani on PageRank random walks on random digraphs. 

Let us mention that cutoff at entropic phenomenon is not bound to random Markov chains: it can also arise in deterministic settings, such as Ramanujan graphs \cite{lubetzky2016ramanujan,ozawa2020entropic} or environments with a Ramanujan property  \cite{bordenave2021cutoff}. In fact from the work of Friedman \cite{friedman2008proof} (see also \cite{bordenave2020new}) it is known that that random regular graphs are Ramanujan with high probability, so the above results can be interpreted as cutoff phenomena in pseudo-deterministic settings. The paper by Eberhard and Varj\'u \cite{eberhard2021mixing} falls in that category, which can also be interpreted as a deterministic realization of the $PS$ model studied in \cite{benhamou2023cutoff}. Finally, the sufficient condition of Salez \cite{salez2023cutoff} allows him to deduce cutoff for a large family of chains satisfying a non-negative curvature criterion, not necessarily random. 

\medskip

When it comes to the concentration result, random variables of the form \eqref{eq:deg1_rv} were introduced in the "combinatorial central limit theorem" by Hoeffding \cite{hoeffding1951combinatorial}. Since then, they have been one of the main motivations for the developments of Stein's method, which is at the basis of the concentration inequality \eqref{eq:concentration_1d} and our result. Early applications of Stein's method for such random variables are found in the works of Bolthausen and Goetze \cite{bolthausen1984estimate,bolthausen1993rate} who obtained error bounds on the normal approximation. 

Introduced by Stein in \cite{stein1972bound} to give a new proof of the classical CLT, Stein's method of exchangeable pairs rapidly became an important tool to prove limit theorems that go way beyond the setting of the CLT, we refer to Stein's paper \cite{stein1986approximate} and to the survey \cite{chatterjee2014short} of Chatterjee and \cite{chatterjee2005exchangeable} of Chatterjee, Diaconis and Meckes which focuses more closely on the subject of Poisson approximation. The first concentration inequalities using Stein's method are due to Chatterjee in his PhD Thesis \cite{chatterjee2005thesis}, see also \cite{chatterjee2007stein}. The inequality \eqref{eq:concentration_1d} can also be interpreted as the concentration of the uniform measure on the symmetric group. Generally, the subject of concentration for the Haar measure is considered by Chatterjee in \cite{chatterjee2007concentration}, where he establishes a connection between concentration and the rate of convergence of some random walks on groups.

\subsection{Proof outline}\label{subsec:outline}

We now give a brief outline of the proof, emphasizing on arguments that required special care compared to previous works and in particular where reversibility is used. There are three occasions where reversibility is used, however there is only one place where this is crucial, the \emph{analysis on the quasi-tree} explained below.

\paragraph*{Knowledge of the invariant measure} As a first but non-essential consequence of reversibility, the electrical network analogy provides an expression for the invariant measure. Indeed we can infer from \eqref{eq:conductances} that 
\begin{equation*}
	\pi(x) = \frac{c^{\ast}(x)}{\sum_{y,z \in [n]} c^{\ast}(y,z)} = \frac{\sum_{z \in [n]} c^{\ast}(x,z)}{\sum_{y,z \in [n]} c^{\ast}(y,z)}
\end{equation*}
is an invariant probability measure. From the boundedness Assumptions \ref{hyp:bdd_delta}, \ref{hyp:bdd_p} we immediately see this implies $\pi(x) = \Theta(1/n)$ for all $x \in [n]$. In general, this knowledge about the invariant measure is sufficient to make a lot of arguments much simpler. We will for instance use it in the proof of the lower bound of Theorem \ref{thm:reversible}. However as mentioned after Theorem \ref{thm:reversible} most of our arguments, in particular those for the upper bound, do not require reversibility and thus will not use this knowledge about the invariant measure. Instead, we follow the strategy used in \cite{bordenave2018random,bordenave2019cutoff} of proving convergence towards an approximate invariant measure $\hat{\pi}$. Letting $\scP$ be the transition matrix of the chain studied, if $\TV{\scP^{t}(x, \cdot) - \hat{\pi}} \leq \e$ holds uniformly in $x$ for a given time $t$, then the invariance property implies that a true invariant measure $\pi$ automatically satisfies $\TV{\pi - \hat{\pi}} \leq \e$, establishing in particular the uniqueness of the invariant measure. 

\paragraph*{The entropic method} The entropic method has become the standard approach to prove cutoff for Markov chains in random environment. It is essentially made of two arguments: first an entropic concentration is shown to occur at the entropic time $(\log n) / h$, which is shown to imply cutoff in the second part of the proof. These arguments come with different variations, depending on whether or not reversibility is assumed. The reversible approach would typically involve using Poincaré's inequality and spectral gap estimates, as done for instance in \cite{hermon2020universality}. Since this does not extend well to non-reversible chains, we follow the approach originally designed for non-backtracking random walks \cite{benhamou2017cutoff,bordenave2018random,bordenave2019cutoff}, which can be thought of as the complete opposite of reversiblity. A main contribution of this paper is thus to extend this approach to backtracking chains as well, including reversible chains. First we prove a quenched entropic concentration property for trajectories, namely we prove a statement of the form $\scP(X_0 \cdots X_t) \simeq e^{-t h + O(\sqrt{t})}$, where we write $\scP(X_0 \cdots X_t) = \prod_{i=0}^{t-1} \scP(X_i, X_{i+1})$. This statement is given here as a rough heuristic, which would be correct in the non-backtracking case. In general, what we prove in practice is rather concentration of the probability to follow the \emph{loop-erased trace} of $X$, so we are back at studying non-backtracking trajectories. This concentration phenomenon is proved by a coupling with another Markov chain lying on an infinite random state space. We call this environment a quasi-tree in reference to \cite{hermon2020universality}, as this object is similar to theirs. Basically, it is designed to be a stationary approximation of the universal cover of the finite chain, in the same way  Erd\H{o}s-R\'enyi random graphs can be approximated locally by infinite Galton-Watson trees. 
 
\paragraph*{Analysis on the quasi-tree} This part is essentially the only one where reversibility comes into play. In the reversible case, a comparison argument will show that conditional on the environment, from any vertex, the probability to escape to infinity along the neighbouring "branch" of the quasi-tree is lower bounded by a constant. This may not be true in the general model \eqref{eq:model0}, which is ultimately why cutoff may occur only for typical starting states and why the proof becomes more technical. The comparison argument we use is based on Rayleigh's monotonicity principle, a well-known result in Markov chain theory often used to deduce recurrence or transience of a chain given another. We show here that a more quantitative use of this result can be used to provide a lower bound on escape probabilities. As far as we know, this argument is new and its simplicity seems particularly appealing, especially in settings where a momentum estimate is not available (i.e. the expected change in the distance with the starting state can be negative, the chain does not uniformly tend to drift away to infinity). On the other hand, with a uniform lower bound on "escape probabilities", the rest of analysis can be conducted pretty much as done in \cite{hermon2020universality,hermon2022cutoff, baran2023phase}.  

The proof of the entropic concentration notably relies on the use of \emph{regeneration times}, which are times where the chain in the quasi-tree makes a transition for the first and last time. In fact, the use of these times goes beyond the entropic concentration phenomenon, so we will spend some time studying the regeneration process. In particular, as in \cite{hermon2022cutoff}, our model requires the consideration of an underlying Markov chain. We found it useful to introduce the setting of Markov renewal processes to formalize these arguments in a generic way. An important requirement for the proof is the fast mixing of the Markov chain underlying the regeneration process. In our case, the law of the environment presents enough independence so that the regeneration chain satisfies Doeblin's condition, which implies that the mixing time is $O(1)$. This is the third occasion we use reversibility, in a non-essential way.

\paragraph*{Argument for the lower bound} The lower bound in the mixing time is based on a simple coverage argument. From the concentration of entropy, at time $t = O( \log n)$ the chain is necessarily confined in a set of size at most $e^{t h + O(\sqrt{t})}$. This becomes $o(n)$ for  $t \leq \log n / h - C \sqrt{\log n}$ and a large constant $C > 0$, which is sufficient to conclude using that the invariant measure satisfies $\pi(x) = \Theta(1/n)$. We make explicit use of $\pi$ here but this is not essential, see Remark \ref{rk:pihat_flat}.

\paragraph*{Concentration of nice trajectories} The second part of the argument uses the entropic concentration to show a second concentration phenomenon. Specifically, for all $x,y \in V$, we define a set $\frN^{t}(x,y)$ of "nice" trajectories of length $t$ between $x,y$, which will have the property that: 1. they are typical, that is the trajectory of the chain is likely to be nice, 2. the probability $\scP^{t}_{\frN}(x,y)$ of following a nice trajectory concentrates around its mean, which is shown to be independent of the starting point. From these properties we deduce that $\scP^{t}(x, \cdot)$ mixes towards a proxy stationary distribution $\hat{\pi}$.

Generally speaking, nice paths are the trajectories that satisfy the entropic concentration. This criterion will ensure the probability $\scP^{t}_{\frN}(x,y)$ is a sufficiently "smooth" quantity as a function of the permutation $\sigma$ to apply Theorem \ref{thm:tensor_concentration}. When modifying the permutation $\sigma$ locally, only a limited number of trajectories between $x$ and $y$ will be affected, each of which contributing a change of at most $e^{-th + O(\sqrt{t})}$ to $\scP^{t}_{\frN}(x,y)$. This will give us the bound on the gradient norm we need to apply Theorem \ref{thm:tensor_concentration} and Corollary \ref{coroll:concentration}. However as mentioned above the entropic concentration only really holds for loop-erased trajectories with a quasi-tree-like neighbourhood, so we need to adapt the definition to accomodate this constraint. This is eventually the main difference with the case of purely non-backtracking walks studied in \cite{benhamou2017cutoff,bordenave2018random,bordenave2019cutoff}, and reversibility is here more of an issue than help.

To give a more detailed overview of the argument, nice paths are defined by splitting a path in three components $\frp_1, \frp_2, \frp_3$ and restricting $\frp_1, \frp_3$ to be contained in quasi-tree-like neighbourhoods around $x$ and $y$ respectively. The concentration result of Corollary \ref{coroll:concentration} will be used conditional on the two neighbourhoods, so only the intermediate path $\frp_2$ is random. In the case of non-backtracking chains, it merely consisted of one edge so only the concentration for degree $1$ functions was needed. When backtracking is allowed, trajectories could bounce back and forth between the two neighbourhoods, which turns out to be a big issue for estimating the expectation. Our strategy is here to reduce the setting as much as possible to the non-backtracking case to avoid this problem by letting the intermediate path $\frp_2$ have larger length $\abs{\frp_2} > 1$, hence the need of a more general concentration result. Using the lower bound on escape probabilities, we can argue that the chain is unlikely to backtrack between the two neighbourhoods on a time scale $t = O(\log n)$ if $\abs{\frp_2} > L = \Theta(\log \log n)$ by union bound. However this length also determines the degree of $\scP^{t}_{\frN}(x,y)$ when realized as a multilinear function, so it cannot be taken arbitrarily large. Finally, we need to make sure that the paths considered remain typical. All in all, this means the length $\abs{\frp_2}$ needs to be taken very carefully. To that end, we can use concentration of the speed (or drift), which is proved in the same time as for the entropy: the distance traveled by the chain at time $t$ is shown to be $\mathscr{d} t  + O(\sqrt{t})$. We will thus take $\abs{\frp_2}$ exactly of size $\Theta(\sqrt{t})$, as this is larger than $L$, will cover the set of typical paths, while being small enough to control the factors of Corollary \ref{coroll:concentration} that are exponential in the degree. In the end, we establish exponential concentration conditional on a list of parameters such as $\abs{\frp_2}$, the estimates on these parameters allowing us then to extend the results by union bound.  

\paragraph*{Computation of $\hat{\pi}$} The last difficulty is to compute the expectation precisely enough to get the probability measure $\hat{\pi}$. To that end, we use regeneration times for the finite chain, essentially defined to be times at which the chain makes a transition and does not backtrack before it has traveled distance $L = \Theta(\log \log n)$. From the non-backtracking property mentioned above, this essentially amounts to considering transitions which are made only once on the time scale $t= O(\log n)$. Furthermore, such regeneration times can of course be coupled with those defined in the quasi-tree setting. From this, we can use the mixing results proved for the regeneration chain in the quasi-tree to deduce similar mixing properties in the finite setting, which yields the expression of the limiting measure $\hat{\pi}$ (Proposition \ref{prop:nice_approx}). 

\subsection{Organisation of the paper} 
In Section \ref{section:first} we give material that will be used throughout the paper, including the definition of quasi-trees. In Section \ref{section:main_arguments}, we give in detail the main technical arguments of the paper that sum up the arguments of the entropic method, from which we will be able to deduce Theorem \ref{thm:reversible}. The next three sections deal with the analysis of the quasi-tree: Section \ref{section:QT1} establishes the lower bound on escape probabilities, Section \ref{section:QT2} proves the main results about the regeneration structure while Section \ref{section:QT3} basically establishes "nice properties" in the quasi-tree setting, including in particular the concentration of the speed and entropy. These nice properties are then transferred back to the finite setting in Section \ref{section:nice_first}, to be used in Section \ref{section:nice} where nice paths are properly defined and shown to have their probability concentrate around the mean. Finally Section \ref{section:concentration} proves the concentration result, Theorem \ref{thm:tensor_concentration}, and is independent of the rest of the paper.

\section{First moment argument, quasi-trees}\label{section:first}

\paragraph*{Basic notations} Throughout the paper, all quantities involved may have an implicit dependency in $n$ and the term constant will refer to quantities that are independent of $n$ and of the randomness. Throughout the paper, we will often consider integers of the form $t = \lfloor C u_n \rfloor$ for a certain sequence $(u_n)_n$. To ease notation, we will often omit the integer part and write only $t = C u_n$. Given a set $S$, $\abs{S}$ denotes its cardinality, $\II_S$ is the indicator of $S$. Given $x,y \in \bR$, we write $x \wedge y := \min(x,y)$, $x \vee y := \max(x,y)$. We use the standard Landau notations $o, O$ for deterministic sequences: $f(n)= O(g(n))$ if there exists a constant $C > 0$ such that $\abs{f(n)} \leq C g(n)$ for all $n$, $f=o(g)$ if $f(n) / g(n) \xrightarrow[n \rightarrow \infty]{} 0$.  We may write $O_{\e}(\cdot)$ to emphasize that the implicit constant depends on $\e$. We also write $f = \Theta (g)$ if $f =O(g)$ and $g = O(f)$, and $f \gg g$ if $g = o(f)$. 
When randomness is involved, all the random variables considered in this paper are defined on an implicit probability space with measure $\bP$. If $Y_n, Z_n, Z$ are random variables, we write $Z_n \xrightarrow[]{\bP} Z$ for convergence in probability and $Z_n = o_{\bP}(Y_n)$ if $Z_n / Y_n \xrightarrow[]{\bP} 0$. In particular $Z_n = o_{\bP}(1)$ if $Z_n \xrightarrow[]{\bP} 0$. An event $A=A(n)$ is said to occur with high probability, if its probability is $1 - o(1)$.

\subsection{Quenched vs annealed probability, first moment argument}\label{subsec:first_moment}

Consider $(X_t)_{t \geq 0}$ a Markov chain in random environment, i.e. with random transition probabilities, such as the one defined by \eqref{eq:reversible_model}. There are two laws naturally associated with the process $(X_t)_{t \geq 0}$. The probability $\bP$, which averages over both the random walk and the environment, gives rise to the \emph{annealed} law of the process $(X_t)_{t \geq 0}$ under which it is not a Markov chain. It is however a Markov chain under the \emph{quenched} law, which conditions on the environment. To emphasize the distinction, it is written using a different font, namely $\bfP$ will denote the quenched distribution. For all state $x$, we write $\bP_{x} := \bP \cond{\cdot}{X_0 = x}, \bfP_x := \bfP \cond{\cdot}{X_0 := x}$. Of course taking the expectation of the quenched law gives back the annealed law. This is the basis of the following first moment argument that will be used used throughout the article, whose phrasing is taken from \cite{bordenave2019cutoff}. To prove a trajectorial event $A$ has quenched probability vanishing to $0$ in probability as $n \rightarrow \infty$, it suffices to prove $A$ has annealed vanishing probability, as Markov's inequality implies for all $\e > 0$, 
\begin{equation}\label{eq:markov_quenched}
	\bP \sbra{ \bfP \sbra{A} \geq \e} \leq \frac{\bP \sbra{A}}{\e}.
\end{equation}
In most of this paper we will first prove statements valid for fixed starting states, chosen independently of the environment, such as $\bP_{x} \sbra{A} = o(1)$, to obtain $\bfP_{x} \sbra{A} = o_{\bP}(1)$. These can be interpreted as conditional statements for the case where the starting state is also random, with a law which is independent of the environment. Results for fixed starting states thus extend automatically to \emph{typical states}, for instance taken uniformly at random in the case the state space is $[n]$. To obtain stronger results valid for all states simultaneously, union bound shows it suffices to improve the annealed error bounds to $\bP_x \sbra{A} = o(1/n)$. This strategy will be used in Section \ref{section:nice_first} to extend results from typical to all vertices.

\subsection{The two-lift chain and half-integer time steps}

We start by rewriting the model \eqref{eq:reversible_model} further, which lacks symmetry in the roles of the two electrical networks $(G_1, c_1), (G_2, c_2)$. To obtain a model that is more symmetric, we can construct the chain as a projection of its two-lift: this consists in viewing $G_1, G_2$ as two disjoint components of one electrical network on a twice larger state space, which can be projected back to obtain the original model. This is essentially a matter of unifying notation and recovering a setting which is somewhat similar to that of \cite{hermon2020universality}, but it can also prove useful on its own. All in all, the model we will work with in this paper is the following. 

Let $V := [2n], V_1:= [n], V_2 := [n+1, 2n]$. Let $\sigma$ be now a uniform bijection from $V_1$ to $V_2$, which can be extended as a matching, or involution, $\eta$ on $V$ by 
\begin{equation*}
\eta_{| V_1} := \sigma, \qquad \eta_{| V_2} := \sigma^{-1}.
\end{equation*}
This matching defines an equivalence relation, namely $x \sim \eta(x)$. 
For all $x \in V$, define $V(x) = V_1 \ \II_{x \in V_1} + V_2 \ \II_{x \in V_2}$. Consider $(G,c)$ to be an electrical network with two disjoint components $V_1, V_2$, which can themselves have several connected components. We can consider conductances to be defined on all pairs of $V \times V$, but equal to $0$ on edges not present in $G$, and we recall $c(x) := \sum_{y} c(x,y)$. Let $\alpha, \beta > 0$ be two constants and $\gamma: V^2 \rightarrow (0,\infty)$ be a map such that $\gamma_{| V_1 \times V_2} \equiv \alpha, \gamma_{| V_2 \times V_1} \equiv \beta$ and $\gamma$ is zero elsewhere. Since we will rarely resort to the reversibility of the chains considered, most of the time we will use the notations of the general model \eqref{eq:model0}: for all $x,y \in V$ let $P(x,y) := c(x,y) / c(x)$ be the transition matrix of the chain corresponding to the conductances $c$, and
\begin{equation*}
	p(x,u) := \frac{\gamma(x,u) c(x)}{\gamma(x,u) c(x) + \gamma(u,x) c(u)}, \quad q(x,u) := 1 - p(x,u)
\end{equation*}
for all $x,u \in V$. Notice that by construction $q(x,u) = p(u,x)$ for all $x,u \in V$. Then consider the Markov chain on $V$ defined by the transition probabilities
\begin{equation*}
	\forall x,y \in V: \scP(x,y) := \left\{ \begin{array}{l l}
		p(x, \eta(x)) \, P(x,y) & \text{if $V(x) = V(y)$} \\
		q(x, \eta(x)) \, P(\eta(x),y) & \text{if $V(x) \neq V(y)$}.\\
	\end{array}\right.
\end{equation*}
We will work with this Markov chain most of the time, which may not be reversible, however the main object of interest of this paper is rather its projection to the quotient $V / \sim \ = [n]$. As can be checked, the fact that $q(x,u) = p(u,x)$ implies that $\scP(x,y) + \scP(x,\eta(y)) = \scP(\eta(x),y) + \scP(\eta(x),\eta(y))$. This condition ensures the projection is itself a Markov chain, with transition matrix given by
\begin{align*}\label{eq:main_model}
	\forall x,y \in [n], \bar{\scP}(x,y) &:= \scP(x,y) + \scP(x, \eta(y)) \\
	&= p(x,\eta(x)) P(x,y) + q(x, \eta(x)) P(\eta(x), \eta(y)).
\end{align*}
By construction $\bar{\scP}$ is reversible, corresponding to the conductances
\begin{equation*}
	\bar{c}(x,y) := \alpha c(x,y) + \beta c(\eta(x),\eta(y))
\end{equation*}
The original chain \eqref{eq:reversible_model} is recovered if one sets $c := c_1 \II_{V_1 \times V_1} + c_2 \II_{V_2 \times V_2}$, identifying $V_2$ with $[n]$ through $x \mapsto x - n$.

Given a probability measure $\mu$ on $V$, its projection $\bar{\mu}$ on $[n]$ is defined by $\bar{\mu}(x) = \mu(x) + \mu(\eta(x))$ for all $x \in [n]$. Notice that if $\pi$ is invariant for $\scP$, then $\bar{\pi}$ is invariant for $\bar{\scP}$. Furthermore, total variation distance is non-increasing under projections, so
\begin{equation}\label{eq:upper_bound_2lift}
	\TV{\bar{\scP}^{t}(x,\cdot) - \bar{\pi}} \leq \TV{\scP^{t}(x, \cdot) - \pi}.
\end{equation}
Therefore to obtain the upper bound in Theorem \ref{thm:reversible}, which is the most difficult part of the argument, it suffices to prove an upper bound for the two-lift. In fact, our arguments show both chains exhibit cutoff at entropic time.

Now let us introduce another characteristic of the two-lift. The two-lift is by construction made of two disjoint subspaces $V_1, V_2$: when at $x$ the Markov chain stays in the same subspace with probability $p(x,\eta(x))$, and changes with probability $q(x,\eta(x))$, after which it takes a step independent of $\eta$. Because of this, it may be convenient to actually distinguish between these two steps, adding transitions at half integer times, defining transition probabilities, for all $t = 0 \mod \bZ$:
\begin{equation}\label{eq:half_time_steps}
	\begin{aligned}
	&\bfP \sbra{X_{t+1/2} = y \ | \ X_{t} = x} = \left\{ \begin{array}{l l} 
		p(x,\eta(x)) & \text{if $y = x$} \\
		q(x,\eta(x)) & \text{if $y = \eta(x)$} \\
		0 & \text{otherwise} \end{array} \right.
	\\
	&\bfP \sbra{X_{t+1} = y \ | \ X_{t+1/2} = x} = P(x,y).
	\end{aligned}
\end{equation}
As is easily checked, the Markov chain $(X_t)_{t \in \bN}$ evaluated at integer times exactly has the transition matrix $\scP$.

\paragraph*{Small-range vs long-range}
 From now on let $G$ denote the deterministic graph underlying the electrical network of the chain $P$ and let $G^{\ast}$ be the graph obtained after adding the edges $(x, \eta(x))$ of the random matching. Using the terminology introduced in \cite{hermon2020universality}, call the latter \emph{long-range} edges. Conversely, the deterministic edges of $G$ are called \emph{small-range}. If one takes the half-integer time steps above into account, the random graph $G^{\ast}$ exactly supports the Markov chain $X$, whereas if one considers only integer time steps the chain can move along a long-range and a small-range edge at once, or along a small-range edge only. We will consider different metrics on these graphs defined in terms of Markov kernels. To start with, we consider the $\scP$-distance $d$ defined by $\scP$: given $x,y \in V$, $d(x,y) := \inf \{k \geq 0: \scP^{k}(x,y) > 0 \}$, which is symmetric by the reversibility of $\scP$. The corresponding (closed) ball of radius $r \in [0,\infty)$ is written $B_{\scP}(x,r) := \{y \in V \ | \ d(x,y) \leq r \}$ to highlight the role of $\scP$. We extend this definition to $r = \infty$, for which open balls are considered, so $B_{\scP}(x, \infty)$ is simply the communicating class of $x$ by reversibility. We define a path as a sequence of vertices or edges which form a connected graph. If $e$ is an oriented edge, $e^{-}, e^{+}$ denote the initial and terminal endpoint of $e$, respectively. Given a path $\frp$, we write $\scP(\frp)$ for the product of transition probabilities along the edges of this path. These notations will extend for other transition kernels as well. In particular, we can also consider the ball $B_{P}(x,r)$ which discards long-range edges. For this reason it will also be written $\BSR(x,r)$ and called the \emph{small-range} ball around $x$.
	
\subsection{Quasi-trees}\label{section:QT0}

We now define quasi-trees, which are designed to be an infinite approximation of the graph $G^{\ast}$. The same terminology and notation is used as in the finite setting to emphasize analogies. This abuse is justified by the fact that later, the two settings will be coupled so that quantities with matching terminology or notation will be equal. If necessary, we will introduce distinct notation.

\begin{definition}\label{def:quasitree}
	A rooted quasi-tree is a $4$-tuple $(\cG, O, \eta, \iota)$ where $\cG = (\cV, \cE)$ is an non-oriented graph, the root $O \in \cV$ is a distinguished vertex and $\iota$ is a map $\iota: \cV \rightarrow V$ that labels vertices of $\cG$ with states in $V$. We call $\iota(x)$ the type of the vertex $x \in \cV$. Finally $\eta$ is a map $\eta= \cV \rightarrow \cV$, which satisfies: 
	\begin{enumerate}[label=(\roman*)]
		\item $\eta$ is an involution either of $\cV$ or of $\cV \smallsetminus \{ O \}$ with no fixed point. The quasi-tree is called respectively \emph{two-sided} or \emph{one-sided} in these cases.
		\item For all $x \in \cV$, the edge $(x,\eta(x))$ is present in $\cE$. Such edges are called long-range edges, others are called small-range edges. 
		\item The long-range edges give a tree-structure to $\cG$: no cycle without repeated edges contains a long-range edge. 
	\end{enumerate}
	
	The two types of edges lead to two types of paths and distances.	
	\begin{enumerate}
		\item A \emph{small-range path} is a path made exclusively of small-range edges. Given $x,y \in \cV$, the \emph{small-range distance} $d_{\SR}(x,y)$ is the minimal number of edges in a small-range path from $x$ to $y$. The corresponding \emph{small-range balls} are written $\BSR(x,r)$. 
		\item A \emph{long-range path} is a sequence $e=(e_i)_{i=1}^{k}$ of long-range edges such that for all $i \leq k-1$ $d_{\SR}(e_i,e_{i+1}) < \infty$ (so $e$ could be completed with small-range paths between long-range edges to obtain a genuine path in $\cG$). The length $\abs{e}$ of $e$ is the number of edges it contains. It joins two vertices $x ,y \in \cV$ if $d_{\SR}(x, e_1) < \infty$ and $d_{\SR}(e_{k}, y) < \infty$.  Given $x,y \in \cV$, the \emph{long-range distance} $d_{\LR}(x,y)$ is the minimal length of a long-range path between $x$ and $y$. We write $\BLR(x,r)$ for the corresponding \emph{long-range balls}. 
	\end{enumerate}

	If $x$ belongs to the connected component of the root, condition (iii) implies that there exists a unique long-range path of minimal length joining $O$ to $x$. If this path is non-empty let $x^{\circ}$ be the furthest endpoint of these long-range edges. We call such a vertex a \emph{center}. If this path is empty, let $x^{\circ} := O$, however the root is not considered a center. 

	From the definition, for any $x \in \cV$, $\BSR(x, \infty) = \BLR(x,0)$ is the set of vertices which can be joined from $x$ by a small-range path. We call the subgraph spanned by this set the \emph{small-range component} of $x$. We can now state a fourth property we require for quasi-trees: 
	\begin{enumerate}[label=(\roman*)]
		\addtocounter{enumi}{3}
		\item \label{enum:SR_component} for all $x \in \cV$, $\BLR(x, 0)$ is isomorphic to $B_{P}(\iota(x^{\circ}), \infty)$,
	\end{enumerate}
	that is, the small-range component of $x$ is the communicating class of $\iota(x)$ in the graph $G$.
\end{definition}

In the sequel we will refer to quasi-trees by their graph component only while keeping the other parameters implicit. 

\paragraph*{One-sided quasi-trees and subquasi-trees}

We introduced one-sided quasi-trees to consider subquasi-trees. If $x$ is a center, then the subquasi-tree $\cG_x$ of $x$ is the graph spanned by the vertices $y$ for which all paths between $O$ and $y$ pass through $x$. If $x$ is not a center, then $\eta(x)$ is and we set $\cG_x := \cG_{\eta(x)}$, so $\cG_x$ does not in fact contain $x$ in that case. Finally the complement subgraph $\cG \smallsetminus \cG_{x}$ is the graph spanned by vertices which can be reached from $O$ by a long-range path that does not use the long-range edge $(x, \eta(x))$. 

\paragraph*{Non-backtracking paths, loop-erased paths, deviation, regeneration}
One main interest of having a genuine tree structure is the consideration of loop-erased trajectories. We thus introduce the following definitions.

\begin{definition}\label{def:backtracking_deviation}
	Let $\cG$ be a quasi-tree. A long-range path $e = (e_1, \ldots, e_k)$ backtracks if it contains two successive identical steps: there exists $i \leq k-1$ such that $e_i = e_{i+1}$. Conversely a long-range path is non-backtracking if it does not backtrack. The loop-erased path $\xi(e)$ is the non-backtracking path obtained after deleting all backtracking steps inductively. For instance, a long-range path $e=(e_1,e_2,e_3,e_3,e_2)$ has loop-erasure $\xi(e) = (e_1)$. Finally we say that $e$ backtracks over a distance $l$ if it contains two subpaths $e'' \subset e' \subset e$ made of successive edges such that $\xi(e') = \varnothing$ is the empty path and $\abs{\xi(e'')} \geq l$. 

	The previous definition is extended to a general path by extracting the long-range path: given a general path $\frp$, let $\xi(\frp)$ denote the loop-erased path formed by the long-range edges crossed by $\frp$. This is a non-backtracking path, which we call the loop-erased path or loop-erased trace of $\frp$. The long-range distance crossed by $\frp$ is the length of $\xi(\frp)$, or equivalently the long-range distance between its endpoints.
	
	The last two definitions require integer parameters. Given $R \geq 1$, let $\cG^{(R)}$ be the connected component of $O$ in the subgraph of $\cG$ spanned by the set
	\begin{equation}\label{eq:restricted_QT}
		\cV^{(R)} := \{ x \in \cV \ | \ d_{\SR}(x^{\circ}, x) < R \}.
	\end{equation}
	A path $\frp$ is said to \emph{deviate} from a small-range distance $R$ if it is not included in $\cG^{(R)}$. 

	Finally, the following notion of \emph{regeneration edges} will be used in several places: let $L \geq 1$ be an integer and consider a path $\frp$. A long-range edge $e$ crossed by $\frp$ is said to be a regeneration edge for $\frp$ with horizon $L$ if after the first time going through $e$ the path crosses a long-range distance $L$ or ends before going back to the endpoint of $e$ which was first visited by $\frp$. 
\end{definition}

\paragraph*{Markov chains on quasi-trees}

Given a quasi-tree $\cG$, it is naturally the underlying graph of the Markov chain $(\cX_t)_{t \geq 0}$ on $\cG$ which has transition probabilities: 
\begin{equation}\label{eq:QT_MC}
	\begin{split}
	&\bfP \cond{\cX_{t+1/2} = y }{\cX_t = x} = \left\{ \begin{array}{l l} 
        p(\iota(x), \iota(\eta(x))) & \text{if $y = x$} \\
        q(\iota(x), \iota(\eta(x))) & \text{if $y = \eta(x)$} \end{array} \right.
	\\
	&\mathbf{P} \cond{\cX_{t+1} = y }{\cX_{t+1/2} = x} =  \left\{ \begin{array}{l l} 
        P(\iota(x), \iota(y)) & \text{if $y^{\circ} = x^{\circ}$} \\
        0 & \text{otherwise} \end{array} \right. .
	\end{split}
\end{equation}

The kernel of this Markov chain will be written as $\cP$. As for the chain $X$ in finite environment, this chain may not be reversible but can be projected to the quotient $\cG / \sim$, which identifies each $x \sim \eta(x)$. The projection is then a reversible chain.

\subsection{The covering quasi-tree}

Among all quasi-trees, one is very natural to consider: this is the quasi-tree obtained by using the random matching $\eta: V \rightarrow V$ to define the matching on the quasi-tree. Given $x \in V$, this is the quasi-tree $(\cG^{\ast}(x), O, \iota, \tilde{\eta})$ defined by $\iota(O) = x$ and
\begin{equation*}
	\forall y \in \cV: \iota(\tilde{\eta}(y)) = \eta(\iota(y)).
\end{equation*}
The fact that this defines a unique quasi-tree is the consequence from property (iv) in Definition \ref{def:quasitree}, which implies that the quasi-tree can be built iteratively. Starting from the small-range component of $O$, which is necessarily $B_{P}(x, \infty)$, the previous equation uniquely determines the long-range edges at long-range distance $0$ from $O$ and thus the whole ball $\BLR(O,1)$. The process iterates to infinity to yield an infinite quasi-tree $\cG^{\ast}$. 

This quasi-tree has the property that it covers exactly $G^{\ast}$: the map $\iota$ is a surjective graph morphism of $\cG^{\ast}$ onto $G^{\ast}$ which preserves the transition probabilities, so the Markov chain $\cX$ defined above projects exactly onto the chain $X$: for all $t \geq 0$ $\iota(\cX_t) = X_t$  in distribution, conditional on $X_0 = x, \cX_0 = O$. 

The chain $\cX$ on $\cG^{\ast}$ will not be studied per se. We introduced it mainly to obtain easier definitions in the finite setting of objects and quantities that are naturally considered in the idealized setting of quasi-trees. First notice that the notions of long-range and small-range edges defined in the finite setting coincide with the projections under $\iota$ of the corresponding edges in the covering quasi-tree. Thus we can extend notions of small-range, long-range paths, backtracking, etc. defined above to $G^{\ast}$ by taking their projections under $\iota$. An exception is the long-range distance, which in the finite setting will make sense only if restricting the quasi-tree: let $R \geq 1$ and recall the definition of the restricted quasi-tree $\cG^{(R)}$ \eqref{eq:restricted_QT}. Given $x \in V$ and $l \geq 0$, we define $\BLR^{(G^{\ast},R)}(x,r)$ in $G^{\ast}$ as the projection 
\begin{equation*}
	\BLR^{(G^{\ast},R)}(x,r) := \iota \left(\BLR(O,r) \cap \cG^{(R)}(x) \right).
\end{equation*}
The superscript $(G^{\ast},R)$ is used to distinguish between the two settings. When not necessary, we may drop it from notation and also keep this parameter $R$ implicit. Note from the definition that in $G^{\ast}$, for any vertices $x,y \in V$, we have $d_{\LR}(x,y) = 0$ if and only if $d_{\SR}(x,y) < R$. 

Finally, we introduce a last definition which is specific to $G^{\ast}$: a \emph{long-range cycle} is a non-deviating non-backtracking long-range path whose starting and ending point are at small-range distance at most $R$ from each other. A subgraph of $G^{\ast}$ is said to be \emph{quasi-tree-like} if it does not contain any long-range cycle. A quasi-tree-like subgraph can thus be identified with a neighbourhood of the root in the covering quasi-tree. Lemma \ref{lem:coupling_quasiT} below will establish that typical and hence most vertices have in fact a quasi-tree-like neighbourhood. This relies on a bounded degree property, which is the object of the following paragraph.

\paragraph*{Bounded degrees and transition probabilities} 
Assumption \ref{hyp:bdd_delta} implies that all graphs considered in this paper, $G, G^{\ast}, \cG$ have bounded degrees. Together with Assumption \ref{hyp:bdd_p}, it also implies that all transition probabilities are bounded away from $0$ uniformly in $n$. Consequently, let $\Delta$ denote a uniform bound on all the degrees, and $\delta > 0$ a uniform lower bound on transition probabilities, which will serve throughout the paper. It gives in particular the growth of balls in $G^{\ast}$ and any quasi-tree $\cG$: for all $x \in V$ and $l \geq 0$
\begin{equation}\label{eq:ball_growth}
	\abs{\BSR(x,l)} \leq \abs{B_{\scP}(x,l)} \leq \frac{\Delta^{l+1} - 1}{\Delta - 1}, \qquad \abs{\BLR^{(G^{\ast},R)}(x,l)} \leq \Delta^{R} \frac{\Delta^{R (l+1)} - 1}{\Delta^{R} - 1}.
\end{equation}
In particular, $\Delta^{R}$ can be thought of as the "long-range" degree. 

\subsection{Coupling and sequential generation}\label{subsec:coupling}

The above definitions should make it pretty clear that the forthcoming proofs are based on a coupling between the finite chain $X$ and the chain $\cX$ on an infinite quasi-tree. The quasi-tree in question is similar to the covering quasi-tree but contains much more independence, allowing basically to resample the matching $\eta$ at each new long-range edge. It can be constructed iteratively as follows: $\iota(O)$ is taken uniformly at random in $V$, which determines the small-range component around $O$ by point \ref{enum:SR_component} of the definition. Then to every vertex $x$ whose type $\iota(x)$ is known, adjoin a long-range edge $(x, \eta(x))$ to $x$, with $\iota(\eta(x))$ being taken uniformly in $V \smallsetminus V(\iota(x))$. In words, each new center added to the quasi-tree is taken uniformly at random in the component of $V$ that maintains the alternation between $V_1$ and $V_2$.
We now explain how one can couple the Markov chain $(X_t)_{t \geq 0}$ on $G^{\ast}$ with the Markov chain $(\cX_t)_{t \geq 0}$ on $\cG$.

\medskip

The following procedure generates the neighbourhood of a vertex in either $G^{\ast}$ or $\cG$ up to some given long-range distance.

Let $x \in V$ be the point whose long-range neighbourhood is to be explored up to long-range distance $L \geq 0$. For $t \geq 1$ we write $EQ_t$ for the exploration queue, that is the set of vertices which remain to be explored, and $D_t$ for the set of explored vertices. The initiation is similar for $G^{\ast}$ and quasi-trees: $D_0 := \emptyset$; $EQ_0 := \BSR(x,R)$ in $G$, $EQ_0 := \BSR(x, \infty)$ in $\cG$. Then the procedure repeats the following steps. If one wants to explore a neighbourhood in $G^{\ast}$ sampling is performed without replacement: 
\paragraph*{Sequential generation of $G^{\ast}$, sampling without replacement} 

For $t \geq 0$,
\begin{enumerate}
	\item pick $y \in EQ_t$: sample $\eta(y)$ uniformly at random in $(V \smallsetminus V(y)) \smallsetminus D_t$,
	\item add $y, \eta(y)$ to $D_{t+1}$ and remove them from $EQ_{t+1}$,
	\item add all vertices $z \in \BSR(\eta(y),R) \smallsetminus \{ \eta(y) \}$ such that $z \notin D_{t}$ and $d_{\LR}(x,z) < L$ to $EQ_{t+1}$. 
\end{enumerate}

If performed with replacement, each new value $\eta(y)$ is picked uniformly in $V \smallsetminus V(y)$ independently of previous draws and considered a new vertex in a set $\cV$. Specifically the procedure becomes:  
\paragraph*{Sequential generation of the quasi-tree $\cG$, sampling with replacement} 

For $t \geq 0$,
\begin{enumerate}[label=\arabic*'.]
	\item pick $y \in EQ_t$: sample $\eta(y)$ uniformly at random in $V \smallsetminus V(y)$,
	\item add $y, \eta(y)$ to $D_{t+1}$ and remove $y$ from $EQ_{t+1}$,
	\item add all vertices $z \in \BSR(\eta(y),R) \smallsetminus \{ \eta(y) \}$ such that $d_{\LR}(x,z) < L$ to $EQ_{t+1}$. 
\end{enumerate}
Since the constraint $z \notin D_t$ has been removed in step 3', vertices which would in $G$ be already explored are here added to the exploration queue and thus considered new vertices. The environment explored is the $L$ long-range ball of a quasi-tree as in Definition \ref{def:quasitree}. Taking $L = \infty$ would thus yield a realization of the infinite random quasi-tree $\cG$.
Finally, the above procedures can be adapted to explore balls $B_{\scP}$ in $G^{\ast}$ and $B_{\cP}$ in $\cG$. 

\paragraph*{Sequential generation along trajectories}

Under the annealed law, the two processes can be generated together with the environment. Later our goal will be to couple weights, which require the exploration of the whole long range neighbourhood of depth $L$ around the trajectory. The sequential generation of the environment along trajectories thus consists in the following steps. Consider for instance the finite setting $G$: let the chain be started at $X_0 := x \in V$, and $D_0 := \emptyset$. Then for all $t \geq 0$:
\begin{enumerate}[label=\alph*.]
	\item Explore the $L$-long-range neighbourhood of $X_t$ with the first procedure described above.
	\item This determines completely the transition probabilities \eqref{eq:half_time_steps} at the state $X_t$, allowing to sample $X_{t+1}$. 
\end{enumerate}
If one considers the second procedure one obtains instead $(\cX_t)_{t \geq 0}$. One can very naturally couple the two procedures and hence both processes using rejection sampling: for each $y \in EQ_t$ sample $\eta(y)$ uniformly in $V \smallsetminus V(y)$ and use it for step 1' in the generation of $\cG$. If in addition $\BSR(\eta(y), R) \cap D_t = \emptyset$, it can also be used for step $1$ in the generation of $G^{\ast}$. Otherwise, make a second draw with the first procedure. This rejection sampling scheme yields a coupling of $(X_t)_{t \geq 0}$ and $(\cX_{t})_{t \geq 0}$ until the first time a newly revealed small-range ball $\BSR(\eta(y), R)$ contains vertices that have already been explored, that is when a long-range cycle appears around the trajectory:
 \begin{equation*}
	\tcoup := \inf \left\{ t \geq 0 \ | \ \bigcup_{s = 0}^{t} \BLR(X_s, L) \text{ contains a long-range cycle} \right\}.
\end{equation*}

The following lemma will be not be used but justifies that trajectories of subpolynomial length started at typical vertices can be coupled exactly with trajectories on a quasi-tree.

\begin{lemma}\label{lem:coupling_quasiT}
	For all constants $C_R, C_L > 0$, setting $R = C_{R} \log \log n, L = C_L \log \log n$, for all $x \in V, \e \in (0,1/2)$ and $t = O(n^{1/2 - \e})$, we have
	\begin{equation*}
		\bfP_{x} \sbra{\text{$\bigcup_{s \leq t} \BLR^{(G^{\ast},R)}(X_s, L)$ is not quasi-tree-like}} = o_{\bP}(1).
	\end{equation*}
	Thus $\bfP_{x} \sbra{\tcoup \leq t} = o_{\bP}(1)$.
\end{lemma}

\begin{proof}
	Let the chain be started at $x \in V$ and let $t = o(n^{1/2 - \e})$ for $\e \in (0,1/2)$. By \eqref{eq:ball_growth}, the number of vertices contained in a long-range ball of radius $L$ is $O(\Delta^{R(L+1)})$. This implies that $\bigcup_{s \leq t} \BLR^{(G^{\ast},R)}(X_s, L)$ contains $m = O((t+1) \Delta^{R(L+1)})$ vertices. Therefore, the exploration procedure along the path $X_0 \cdots X_t$ up to long-range distance $L$ requires at most $m$ draws of values $\eta(y)$, each having probability at most $m \Delta^{R} / n$ that the small-range ball $\BSR(\eta(y), R)$ contains already explored vertices. Hence the number of long-range cycles in $\bigcup_{s \leq t} \BLR^{(G^{\ast},R)}(X_s, L)$ is stochastically upper bounded by a binomial $\mathrm{Bin}(m, m \Delta^{R}/n)$ so that by Markov's inequality
	\begin{align*}
		\bP_{x} \sbra{\text{$\bigcup_{s \leq t} \BLR^{(G^{\ast},R)}(X_{s},L)$ is not quasi-tree-like}} &\leq \frac{m^2 \Delta^{R}}{n} = O \left( \frac{(t+1)^{2} \Delta^{2 RL+3R}}{n} \right) \\
		&= O \left( t^2 n^{-1 + o(1)} \right) = o(1)
	\end{align*}
	by the choice of $t = O(n^{1/2 - \e})$ and $R, L = O(\log \log n)$. This bound on the annealed probability implies the quenched result by the first moment argument \eqref{eq:markov_quenched}.
\end{proof}

\section{The entropic method: main arguments}\label{section:main_arguments}

\subsection{Nice trajectories}

Our application of the entropic method consists in finding a definition of nice trajectories designed to be typical trajectories and have their probability concentrated around the mean. The latter will come from a constraint about an entropy-like quantity for the chain, which arises from comparing the trajectories of the finite chain $X$ with the loop-erased trace of $\cX$ in the (infinite) quasi-tree setting. Other properties will thus be required from nice paths, whose goal is basically to make them close to being non-backtracking trajectories in a quasi-tree. All in all, the defining properties of a nice trajectory will essentially be that: 

\begin{enumerate}[label=(\roman*)]
    \item it is contained in a quasi-tree-like portion of the graph, the endpoint having a quasi-tree-like neighbourhood up to $\scP$-distance $\lfloor \alpha \log n \rfloor$ for some $\alpha > 0$ (Lemma \ref{lem:QTlike}),
	\item it does not deviate or backtrack too much, and contains sufficiently many regeneration edges (Lemma \ref{lem:typical_paths}),
    \item the drift and entropy concentrate on this trajectory (Proposition \ref{prop:concentration_G}).
\end{enumerate}

The first part of the argument will consist in proving that the chain is likely to follow nice trajectories. However this may not be true for arbitrary starting vertices, but only for typical ones. For an arbitrary starting vertex, we prove the trajectory becomes nice after some time $s = O(\log \log n)$. Since this is $o(\sqrt{\log n})$, the contribution of these initial steps will be negligible. 

\begin{lemma}\label{lem:QTlike}
	\begin{enumerate}[label=(\roman*)]
		\item For all $C_R, C_L > 0$, there exists $C > 0$ such that for $R := C_R \log \log n$, $L := C_L \log \log n$, $s \geq C \log \log n$ and all $t =o(n^{1/16})$,
		\begin{equation*}
			\max_{x \in V} \bfP_{x} \sbra{ \bigcup_{t' \in [s,s+t]} \BLR^{(G^{\ast},R)}(X_{t'}, L) \text{ is not quasi-tree-like}} = o_{\bP}(1).
		\end{equation*}
		\item For all $C > 0$, there exists $\alpha > 0$ such that for any $t \geq C \log n$,
		\begin{equation*}
			\max_{x \in V} \bfP_x \sbra{B_{\scP}(X_t, \lfloor \alpha \log n \rfloor) \text{ is not quasi-tree-like}} = o_{\bP}(1).
		\end{equation*}
	\end{enumerate}
 \end{lemma}

For the second property we recall that the notions of deviation, backtracking and regeneration edges are given in Definition \ref{def:backtracking_deviation}.

\begin{lemma}\label{lem:typical_paths}
    Let $\Gamma(R, L, M)$ denote the set of paths $\frp$ in $G^{\ast}$ such that $\frp$ does not deviate from a small-range distance more than $R$, backtrack over a long-range distance $L$ or contain a subpath of length $M$ without a regeneration edge.
	There exist constants $C_R, C_L, C_M, C > 0$ such that for $L = C_L \log \log n$, $R = C_R \log \log n$ and $M = C_M \log \log n$, for all $s \geq C \log \log n$ and $t = O(\log n)$:
    \begin{equation*}
        \min_{x \in V} \bfP_{x} \sbra{(X_s \cdots X_{s+t}) \in \Gamma(R,L,M)} = 1 - o_{\bP}(1).
    \end{equation*}
\end{lemma}

\begin{remark}\label{rk:uniform_shifted}
	Notice that for any trajectorial event $A$
    \begin{align*}
        \bfP_{x} \sbra{(X_s, X_{s+1}, \ldots) \in A} &= \sum_{y \in V} \bfP_{x} \sbra{X_s = y} \bfP_{y} \sbra{(X_0, X_1, \ldots) \in A} \\
        &\leq \max_{y \in V} \bfP_{y} \sbra{(X_0, X_1, \ldots) \in A}.
    \end{align*}
    Thus if we prove $A$ holds from time $s$ with probability $1 - o_{\bP}(1)$ uniformly over the starting state this automatically extends to larger times $t \geq s$. 
\end{remark}

\subsection{Concentration of drift and entropy}\label{subsec:concentration}

The third and main property of nice trajectories consists in the concentration of an entropy-like quantity for the finite chain. To that end, we define weights as follows.

Throughout the paper, we will use the letter $\tau$ for a variety of stopping times. It should be clear from the context what the notation refers to. If $x$ is an element or a set then $\tau_x$ will generally denote the hitting time of $x$. If $l \geq 0$ then $\tau_l$ will generally denote the time a certain length or distance $l$ is reached. 
For all $l \geq 0$, consider here
\begin{equation*}
	\tau_l := \inf \{ t \geq 0 \ | \ \abs{\xi(X_0 \cdots X_t)} = l \}
\end{equation*}
and fix $R, L \geq 1$ for the rest of this section. Let
\begin{align*}
	&\TSR := \inf \{t \geq 0 \ | \ (X_0 \cdots X_t) \text{ deviates from a small-range distance $R$} \} \\
	&\TNB := \inf \{ t \geq 0 \ | (X_0 \cdots X_t) \text{ backtracks over a long-range distance $L$}\}.
\end{align*}
These are stopping times which depend on $R, L$. By construction if $\frp$ is a path in $\Gamma(R,L,M)$ as defined in Lemma \ref{lem:typical_paths}, then the stopping $\TSR \wedge \TNB$ does not occur on the trajectory $\frp$.
Given an oriented long-range edge $e$ and $x \in V$ such that $d_{\SR}(x, e^{-}) < R$, define the weights
\begin{equation}\label{eq:def_weights_G}
	\begin{split}
	w_{x, R, L}(e) &:= \bfP_{x} \sbra{\xi(X_0 \cdots X_{\tau_L})_1 = e,  \tau_{L} < \TSR} \\
	w_{R, L}(e \ | \ x) &:= \bfP_{x} \cond{\xi(X_0 \cdots X_{\tau_L})_1 = e}{\tau_{L} < \tau_{\eta(x)} \wedge \TSR}.
	\end{split}
\end{equation}
Here $\xi(X_0 \cdots X_{\tau_L})_1$ denotes the first edge of $\xi(X_0 \cdots X_{\tau_L})$. Then for a non-backtracking long-range path $\xi = \xi_1 \cdots \xi_k$, set
\begin{equation*}
	w_{x,R,L}(\xi) := w_{x,R,L}(\xi_1) \prod_{i=2}^{k} w_{R,L}(\xi_i \ | \ \xi_{i-1}^{+})
\end{equation*}
where empty products are by convention equal to $1$. The notation is consistent with the identification of edges with paths of length $1$. 

\begin{remark}\label{rk:weight_sum_1}
	Note that for fixed $x \in V$, we have $\sum_{e} w_{x,R,L}(e) \leq 1$ and $\sum_{e} w_{R,L}(e \ | \ x) \leq 1$ where the sum is over all long-range edges. By extension the sum of weights over all non-backtracking paths starting from $x$ is at most $1$.
\end{remark}

Given a sequence $u = (u_i)_{i \leq l}$ of length $l$ and $k \geq 1$, we write $(u)_{\leq k} := (u_i)_{i \leq k}$ for the sequence truncated at length $k$. The following lemma shows that weights are good proxies for measuring the probability that the loop-erased trace follows a given non-backtracking path. We will only need the lower bound.

\begin{lemma}\label{lem:probability_weights}
	Let $x \in V$ and $k \geq 1$ be an integer. Suppose that $\BLR^{(G^{\ast},R)}(x, k)$ is quasi-tree-like. Then for all non-backtracking long-range paths $\xi$ of length $k$, started in $\BSR(x, R)$
	\begin{equation}\label{eq:probability_weights}
		\bfP_{x} \sbra{\begin{gathered}\xi(X_0 \cdots X_{\tau_{k + L - 1}})_{\leq k} = \xi, \\ \tau_{k+L-1} < \TSR \wedge \TNB \end{gathered}} \leq w_{x,R,L}(\xi) \leq \bfP_{x} \sbra{\begin{gathered}\xi(X_0 \cdots X_{\tau_{k + L - 1}})_{\leq k} = \xi, \\ \tau_{k+L-1} < \TSR \wedge \TNB \end{gathered}} + u(\xi)
	\end{equation}
	where $u(\xi) \geq 0$ is such that 
	\begin{equation*}
		\sum_{\xi} u(\xi) \leq \bfP_{x} \sbra{\TSR \wedge \TNB \leq \tau_{k+L-1}},
	\end{equation*}
	the sum being over non-backtracking long-range paths of length $k$ from $x$.
\end{lemma}

\begin{proof}
	The proof is by induction on $k \geq 1$. For $k = 1$ the inequalities are in fact equalities by definition, since $\TNB > \tau_L$ necessarily. 

	To ease notation, drop the parameters $R, L$ from the stopping times and weights for the rest of the proof. Suppose that the result holds for $k \geq 1$ and let $\xi$ be of length $k + 1$.  Let $e := \xi_k, f := \xi_{k+1}$ and write $L_k$ for the last time $t$ that $\abs{\xi(X_0 \cdots X_t)} = k$ after $\tau_{k}$ and before $\tau_{k+L}$ or coming back to $e^{-}$. Then having $\xi(X_0 \cdots X_{\tau_{k + L -1}})_{\leq k} = (\xi)_{\leq k}$ requires that $X_{L_k} = e^{+}$, after which the chain crosses a long-range distance $L-1$ without backtracking to $e^{-}$. Thus by the induction hypothesis,
	\begin{align*}
		w_{x}((\xi)_{\leq k}) &\geq \bfP_{x} \sbra{\xi(X_0 \cdots X_{\tau_{k + L - 1}})_{\leq k} = \xi, \tau_{k + L - 1} < \tau_{\SR} \wedge \tau_{\NB} } \\
		&= \bfP_{x} \sbra{X_{L_k} = e^{+}, L_k < \tau_{\SR} \wedge \tau_{\NB}} \bfP_{e^{+}} \sbra{ \tau_{L - 1} < \tau_{e^{-}} \wedge \tau_{\SR}}.
	\end{align*}
	On the other hand,
	\begin{equation*}
		w(f \ | \ e^{+}) = \frac{\bfP_{e^{+}} \sbra{\xi(X_0 \cdots X_{\tau_{L}})_1 = f, \tau_{L} < \tau_{e^{-}} \wedge \tau_{\SR}}}{\bfP_{e^{+}} \sbra{\tau_L < \tau_{e^{-}} \wedge \tau_{\SR}}}.
	\end{equation*}
	Since $\bfP_{e^{+}} \sbra{ \tau_{L-1} < \tau_{e^{-}} \wedge \tau_{\SR} } \geq \bfP_{e^{+}} \sbra{ \tau_{L} < \tau_{e^{-}} \wedge \tau_{\SR}}$, we deduce that 
	\begin{align*}
		w_{x}(\xi) &= w_{x}((\xi)_{\leq k}) \, w(f \ | \ e^{+}) \\
		&\geq \bfP_{x} \sbra{X_{L_k} = e^{+}, \tau_{\SR} \wedge \tau_{\NB}  > L_{k}} \bfP_{e^{+}} \sbra{\xi(X_0 \cdots X_{\tau_{L}})_1 = f, \tau_{L} < \tau_{\SR} \wedge \tau_{e^{-}}}.
	\end{align*}
	This is the probability that after $L_k$, the chain crosses a long-range distance $L$ using the edge $f$, without reaching the boundary of a small-range ball of radius $R$ and without coming back to $e^{-}$. Hence on this event $\xi(X_0 \cdots X_{\tau_{k + L}})_{k+1} = f$, with $\tau_{k+L} < \tau_{\SR} \wedge \tau_{\NB}$, which proves the lower bound. 
	
	For the upper bound, the induction hypothesis yields this time 
	\begin{align*}
		w_{x}((\xi)_{\leq k}) &\leq \bfP_{x} \sbra{X_{L_k} = e^{+}, L_k < \tau_{\SR} \wedge \tau_{\NB}} \bfP_{e^{+}} \sbra{ \tau_{L - 1} < \tau_{e^{-}} \wedge \tau_{\SR}} + u((\xi)_{\leq k}).
	\end{align*}
	Then use that
	\begin{equation*}
		\bfP_{e^{+}} \sbra{ \tau_{L-1} < \tau_{e^{-}} \wedge \tau_{\SR}} \leq \bfP_{e^{+}} \sbra{ \tau_{L} < \tau_{e^{-}} \wedge \tau_{\SR}} + \bfP_{e^{+}} \sbra{ \tau_{L-1} < \tau_{e^{-}} \wedge \tau_{\SR} \leq \tau_{L}} 
	\end{equation*}
	to bound
	\begin{align*}
		w_{x}(\xi) &\leq \bfP_{x} \sbra{X_{L_k} = e^{+}} \bfP_{e^{+}} \sbra{\xi(X_0 \cdots X_{\tau_{L}})_1 = f, \tau_{L} < \tau_{\SR} \wedge \tau_{e^{-}}} \\
		&+ \bfP_{x} \sbra{X_{L_k} = e^{+}} \bfP_{e^{+}} \sbra{ \tau_{L-1} < \tau_{e^{-}} \wedge \tau_{\SR} \leq \tau_{L}} w(f \ | \ e^{+}) + u((\xi)_{\leq k})w(f \ | \ e^{+}).
	\end{align*}
	The first term is that of the lower bound. Regroup the two other terms as $u(\xi)$. Since weights sum to $1$, summing over $\xi$, which involves in particular summing over $e$ and $f$, yields that 
	\begin{equation*}
		\sum_{\xi} u(\xi) \leq \sum_{e} \bfP_{x} \sbra{X_{L_k} = e^{+}} \bfP_{e^{+}} \sbra{ \tau_{L-1} < \tau_{e^{-}} \wedge \tau_{\SR} \leq \tau_{L}} + \bfP_{x} \sbra{\tau_{\SR} \wedge \tau_{\NB} \leq \tau_{k+L-1}}.
	\end{equation*}
	Observe now the first term corresponds to backtracking or deviating after reaching level $k+L-1$ but before reaching level $k+L$ from $x$. Thus the two terms correspond to disjoint events which both imply $\tau_{\SR} \wedge \tau_{\NB} \leq \tau_{k+L}$, hence the upper bound.
\end{proof}

The first part of the proof of Theorem \ref{thm:reversible} consists in proving the following quenched concentration phenomenon. 

\begin{proposition}\label{prop:concentration_G}
	There exist $\mathscr{d}, h = \Theta(1)$ and large constants $C_R , C_L, C > 0$ such that the following holds. Letting $R := C_R \log \log n, L := C_L \log \log n$, for all $\e > 0$ there exist constants $C_{\LR}(\e), C_{h}(\e) > 0$ such that for all $s \geq C \log \log n, t \gg 1$, $t = O(\log n)$ with high probability, 
	\begin{equation*}
		\min_{x \in V} \bfP_{x} \sbra{ \begin{array}{c} \abs{ \abs{\xi(X_{s} \cdots X_{s+t})} - \mathscr{d} t} \leq C_{\LR} \sqrt{t} \\
		\abs{ - \log w_{X_s, R, L}( \xi(X_s \cdots X_{s+t})) - h t} \leq C_{h} \sqrt{t}
		 \end{array}} \geq 1 - \e. 
	\end{equation*}
\end{proposition}

\subsection{Concentration of nice paths}

The second part of the argument uses the properties of nice paths to show that the probability of following a nice trajectory concentrates around its mean. The latter can be computed, providing an approximate stationary distribution $\hat{\pi}$ for $\scP$ on $V$.

Given $u \in V$, $L \geq 1$ consider again $\tau_L := \inf \{ t \geq 0 \ | \ \abs{\xi(X_0 \cdots X_t)} = L \}$ and
	\begin{equation}\label{eq:bfQu_G}
		\bfQ^{(L)}_{u} := \bfP \cond{\cdot}{X_{1/2} = u, \tau_{\eta(u)} > \tau_L}.
	\end{equation}
In words, this measure considers trajectories immediately after a regeneration time, i.e. a time at which a regeneration edge is crossed (with horizon $L$). By Lemma \ref{lem:typical_paths}, if one takes $L = \Theta(\log \log n)$, the conditioning on $\tau_{L} < \tau_{X_0}$ essentially forbids the chain to come back at all to $u$ on a time scale $O(\log n)$. If $\nu$ is a probability measure on $V$, write $\bfQ^{(L)}_{\nu} := \sum_{u \in V} \nu(u) \bfQ^{(L)}_{u}$ and $\bfE^{(L)}_{\bfQ_{\nu}}$ for the expectation with respect to this measure.
All in all, the whole argument of this paper is summarized in the following proposition. 

\begin{proposition}\label{prop:nice_approx}
	There exist a deterministic probability measure $\nu$ on $V$, a deterministic $s_0 =\Theta(\log n)$, constants $C_L, C_M, C > 0$ and for all $x, y \in V, t \in \bN$ a set $\frN^{t}(x,y)$ of length $t$ paths between $x$ and $y$ for which the following holds. Let $L := C_L \log \log n, M := C_M \log \log n$ and write $\scP_{\frN}^{t}(x,y) := \sum_{\frp \in \frN^{t}(x,y)} \scP(\frp)$. Consider the random probability measure
		\begin{equation}\label{eq:pihat}
			\hat{\pi}(v) := \frac{1}{\bfE_{\bfQ_{\nu}^{(L)}} \sbra{T_1 \II_{T_{1}^{(G, L)} \leq M}}} \sum_{r =0}^{M} \bfQ_{\nu}^{(L)} \sbra{X_{r+s_0} = v, r < T_1 \leq M}
		\end{equation}
		where $T_1$ denotes the first regeneration time with horizon $L = C_L \log \log n$. For all $\e > 0$, there exists $C(\e) >0$ constant in $n$, such that for $t = \log n / h + C(\e) \sqrt{\log n}$,
	\begin{enumerate}[label =(\roman*)]
		\item for all $s \geq C \log \log n$, with high probability
		\begin{equation*}
			\min_{x \in V} \sum_{y \in V} \scP^{s} \scP^{t}_{\frN} \, (x,y) \geq 1 - \e.
		\end{equation*}\label{enum:not_nice_o1}
		\item there exists $c = (c_v)_{v \in V}$ such that $\sum_{v \in V} c_v = o_{\bP}(1)$ and with high probability, for all $x,y \in V$,
		\begin{equation*}
			\scP_{\frN}^{t}(x,y) \leq (1 + \e) \hat{\pi}(y) + c(y) + \frac{\e}{n}.
		\end{equation*}\label{enum:nice_upper_bound}
	\end{enumerate}
\end{proposition}

\begin{proof}[Proof of Theorem \ref{thm:reversible}]
	Recall $\scP$ is the transition matrix of the two-lift chain on $V$, which projects to a transition matrix $\bar{\scP}$ on $[n]$. 

	\medskip

	Start with the upper bound on the mixing time. Let $\e > 0$ and consider $s= C \log \log n, t := \log n / h + C(\e) \sqrt{\log n}$. Since $\scP \geq \scP_{\frN}$ entry-wise, for all $x \in V$
	\begin{align*}
		\TV{\hat{\pi} - \scP^{s+t}(x,\cdot)} &= \sum_{z \in V} \left[ \hat{\pi}(z) - \scP^{s+t}(x,z) \right]_{+} \leq \sum_{y,z \in V} \scP^{s}(x,y) \left[ \hat{\pi}(z) - \scP^{t}(y,z) \right]_{+} \\
		&\leq \sum_{y,z \in V} \scP^{s}(x,y) \left[(1 + \e) \hat{\pi}(z) + c(z) + \frac{\e}{n} - \scP_{\frN}^{t}(y,z) \right]_{+}.
	\end{align*}
	Point \ref{enum:nice_upper_bound} of the above proposition implies that with high probability, the right hand side summands are non-negative for all $x \in V$, so the sum can be computed to obtain that with high probability, for all $x \in V$, 
	\begin{align*}
		\TV{\hat{\pi} - \scP^{s+t}(x,\cdot)} \leq 1 - \scP^{s} \scP_{\frN}^{t}(x,y) + 3 \e.
	\end{align*}
	Using point \ref{enum:not_nice_o1}, with high probability
	\begin{equation*}
		\max_{x \in V} \TV{\scP^{s+t}(x, \cdot) - \hat{\pi}} \leq 4 \e
	\end{equation*}
	which projects by \eqref{eq:upper_bound_2lift} to 
	\begin{equation*}
		\max_{x \in V} \TV{\bar{\scP}^{s+t}(x, \cdot) - \bar{\hat{\pi}}} \leq 4 \e
	\end{equation*}
	with $\bar{\hat{\pi}}$ the projection onto $[n]$ of $\hat{\pi}$.
	Since this estimate is uniform in the starting state, it extends to any starting distribution and in particular to a stationary distribution. Thus for any stationary distribution $\pi$ of $\bar{\scP}$,
	\begin{equation}\label{eq:bound_hatpi}
		\TV{\pi - \bar{\hat{\pi}}} \leq 4 \e
	\end{equation}
	and from triangular inequality we obtain that with high probability
	\begin{equation*}
		\max_{x \in [n]} \TV{\bar{\scP}^{s+t}(x, \cdot) - \pi} \leq 8 \e.
	\end{equation*}
	Since this is valid for any invariant measure, the latter must be unique and the chain irreducible and aperiodic. Noticing that $s = o(\sqrt{\log n})$, this proves the upper bound of Theorem \ref{thm:reversible}. 

\medskip

For the proof of the lower bound we make use of the explicit knowledge of the invariant measure, although this is not necessary, see Remark \ref{rk:pihat_flat}. Without loss of generality, suppose $\pi$ is the unique invariant measure of $\scP$, which was anyway proved above to be true with high probability.

For all $s,t \geq 0, \theta > 0$ and $x,y \in [n]$,
\begin{align*}
	\bar{\scP}^{s+t}(x,y) &= \scP^{s+t}(x,y) + \scP^{s+t}(x, \eta(y)) \\
	&\geq \bfP_{x} \sbra{X_{s+t} \in \{y, \eta(y) \}, w_{x,R,L}(\xi(X_s \cdots X_{s+t})) \leq \theta}.
\end{align*}
If equality holds, then 
\begin{equation*}
	\bar{\pi}(y) - \bfP_{x} \sbra{X_{s+t} \in \{y, \eta(y) \}, w_{x,R,L}(\xi(X_s \cdots X_{s+t})) \leq \theta} \leq \left[ \bar{\pi}(y) - \bar{\scP}^{s+t}(x,y) \right]_{+}.
\end{equation*}
If equality does not hold, there must exist a non-backtracking long-range path $\xi$ between $x$ and $y$ or $x$ and $\eta(y)$ for which $w(\xi) > \theta$, in which case
\begin{multline*}
	\bar{\pi}(y) - \bfP_{x} \sbra{X_{s+t} \in \{y, \eta(y) \}, w_{x,R,L}(\xi(X_s \cdots X_{s+t})) \leq \theta} \leq \pi(y) \, \II_{\exists \xi: w_{x,R,L}(\xi) > \theta} \\
	+ \pi(\eta(y)) \, \II_{\exists \xi: w_{x,R,L}(\xi) > \theta}.
\end{multline*}
We did not specify where $\xi$ lies to ease notation in the indicator functions but it depends on $y$ and $\eta(y)$ respectively.
Combining the two inequalities we obtain that in either case 
\begin{multline*}
	\bar{\pi}(y) - \bfP_{x} \sbra{X_{s+t} \in \{y, \eta(y) \}, w_{x,R,L}(\xi(X_s \cdots X_{s+t})) \leq \theta} \leq \left[ \bar{\pi}(y) - \bar{\scP}^{s+t}(x,y) \right]_{+} \\ + \pi(y) \, \II_{\exists \xi: w_{x,R,L}(\xi) > \theta} + \pi(\eta(y)) \, \II_{\exists \xi: w_{x,R,L}(\xi) > \theta}.
\end{multline*}
By Cauchy-Schwarz inequality,
\begin{equation*}
	\sum_{y \in V} \pi(y) \II (\exists \xi: w_{x,R,L}(\xi) > \theta) \leq \left( \sum_{y \in V} \pi(y)^{2} \right)^{1/2} \left( \sum_{\xi} \II (w_{x,R,L}(\xi) > \theta) \right)^{1/2},
\end{equation*}
where the second sum is over non-backtracking long-range paths from $x$. By Remark \ref{rk:weight_sum_1} weights sum up to at most $1$ so this sum contains at most $\theta^{-1}$ positive terms and
\begin{equation}\label{eq:mixing_lower_bound}
	\bfP_{x} \sbra{w_{x,R,L}(\xi(X_s \cdots X_{s+t})) > \theta} \leq \TV{\bar{\scP}^{s+t}(x, \cdot) - \bar{\pi}} + \sqrt{\frac{1}{\theta} \sum_{y \in V} \pi(y)^{2}}.
\end{equation}
To complete the proof, let $\e \in (0, 1)$ and specialize to $s := C \log \log n, t := \log n / h - C_1 \sqrt{\log n}$ and  $\theta := n^{-1} \exp(C_{2} \sqrt{\log n})$ for some $C_1(\e)$, $C_2(\e) > 0$. 
Choosing the constants $C, C_1$ large enough, $\exp(-t h - C_{h}(\e) \sqrt{t}) = n^{-1} \exp((C_1 - C_{h} / \sqrt{h} ) \sqrt{\log n} - o(\sqrt{\log n})) > \theta$ for large enough $n$, hence Proposition \ref{prop:concentration_G} implies that the left hand-side of \eqref{eq:mixing_lower_bound} is at least $1 - \epsilon$ with high probability. On the other hand, as explained in the proof outline (Section \ref{subsec:outline}) the boundedness assumptions \ref{hyp:bdd_delta}, \ref{hyp:bdd_p} imply that $\pi(x) = \Theta(1/n)$ for all $x \in V$. Hence the square-root term in the right hand side is $o(1)$. All in all, this proves that with high probability, 
\begin{equation*}
	\TV{\bar{\scP}^{s+t}(x, \cdot) - \bar{\pi}} \geq 1 - \epsilon.
\end{equation*}
\end{proof}

\begin{remark}\label{rk:pihat_flat}
	From the explicit formula for the measure $\hat{\pi}$ \eqref{eq:pihat}, we can show it can be decomposed as $\hat{\pi} = \hat{\pi}_1 + \hat{\pi}_2$ with 
		\begin{equation*}
			\sum_{x \in V} \hat{\pi}_{1}(x)^{2} = o_{\bP}( (\log n)^{b} / n), \qquad \hat{\pi}_2(V) = o_{\bP}(1)
		\end{equation*}
		for some $b > 0$. This is sufficient to prove the lower bound on the mixing time without resorting to the explicit knowledge of the invariant measure $\pi$, as we know $\hat{\pi}$ is close to $\pi$ in total variation \eqref{eq:bound_hatpi} and the proof only used the fact that $\pi$ has a small $\ell^2$ norm.
\end{remark}

\begin{remark}\label{rk:simple_graphs}
	Let us comment on how the proof also accommodates the superposition of simple graphs. Observe that if two edges are aligned under the matching $\eta$, i.e. if there exist $x,y \in V$ such that $P(x,y) > 0$ and $P(\eta(x), \eta(y)) > 0$, these constitute an obstruction to the quasi-tree likeness of the neighbourhood of $x$. As the nice trajectories considered in Proposition \ref{prop:nice_approx} require having a quasi-tree-like neighbourhood, they avoid in particular these edges, so it makes no difference in the end to adjust the conductances of these edges or not.
\end{remark} 

\section[Quasi-tree I: Escape probabilities]{Analysis on the quasi-tree I: escape probabilities}\label{section:QT1}

The objective of the three following sections is to prove the concentration of the drift and entropy, along with the other nice properties of Section \ref{section:main_arguments}, for the Markov chain $(\cX_t)_{t \geq 0}$ on a random infinite quasi-tree $\cG$. This will allow us to deduce the corresponding statements for the finite chain $(X_t)_{t \geq 0}$ thanks to the coupling presented in Section \ref{subsec:coupling}. As in \cite{berestycki2018random, hermon2020universality}, the argument is based on the existence of regeneration times, which in the infinite setting can be defined as times at which $\cX_t$ visits a long-range edge for the first and last time. A first step towards this objective is to lower bound the probability of escaping to infinity in the quasi-tree, which is the object of this whole section. 

\subsection{Escape probabilities}

The following notations will be used for the three next sections. $\cG$ will denote a random quasi-tree, which under $\bP$ has the law of the quasi-tree described in Section \ref{subsec:coupling}. Its vertex set is $\cV$.

\begin{definition}\label{def:escape}
	Given a non-center vertex $x \in \cV$, let
	\begin{equation*}
		\qesc(x) := \bfP_{x} \sbra{\forall t \geq 1: \cX_{t} \in \cG_x}. 
	\end{equation*}
	be the quenched probability that the chain enters the subquasi-tree of $x$ and never leaves it. If $x$ is a center, it is useful to also consider starting at time $1/2$ and let
	\begin{equation*}
		\qesc(x) := \bfP_{x} \sbra{\forall t \geq 0: \cX_{t} \in \cG_x} \wedge \bfP \cond{\tau_{\eta(x)} = \infty}{\cX_{1/2} = x}.
	\end{equation*}
	We call these quantities the \emph{escape probability} at $x$.
\end{definition}

\begin{remark}\label{rk:escape_center}
	Note that if $x$ is not a center,
	\begin{equation*}
		\qesc(x) \geq q(x, \eta(x)) \bfP \cond{\tau_{x} = \infty}{\cX_{1/2} = \eta(x)}.
	\end{equation*} 
	Similarly, if $x$ is a center, Assumption \ref{hyp:cc3} asserts there exists $y \neq x$ in the same small-range component as $x$, so that
	\begin{equation*}
		\bfP_{x} \sbra{\forall t \geq 0: \cX_{t} \in \cG_x} \geq p(x, \eta(x)) \, P(x,y) \, \qesc(y). 
	\end{equation*}
	By Assumptions \ref{hyp:bdd_delta} and \ref{hyp:bdd_p} the entries of $p$ and $q$ are bounded. Thus to lower bound escape probabilities, it matters little whether we consider a center or non-center vertex, and whether we start at integer
	or half-integer time.
\end{remark}

In \cite{hermon2020universality}, the authors prove in their model that the escape probability is lower bounded uniformly in $n$, conditional on $\cG$. This extends to the reversible case thanks to a comparison argument. 
\begin{proposition}\label{prop:escape_reversible}
	There exists a constant $q_{0} > 0$ such that for all realizations of $\cG$, for all $x \in \cV$ we have $\qesc(x) \geq q_0$. 
\end{proposition}

\subsection{Lower bound on escape probabilities}\label{subsec:escape_reversible}

We now establish Proposition \ref{prop:escape_reversible}. The proof is based on ideas to prove transience of a Markov chain, which is fundamental for our purpose and is already not clear in our model. The reversibility assumption considerably augments the available toolbox.  In particular Rayleigh's monotonicity principle is a well-known result which states that the effective conductance to infinity increases monotonically with individual conductances (see \cite{lyons2016probability}[Chapter 2]). While this result is generally used qualitatively to establish transience or recurrence of a given chain, it is used here to obtain a quantitative comparison between escape probabilities.  

Let us recall briefly the notions we are going to use. We refer to \cite{lyons2016probability}[Chapters 2, 3, 5] for a detailed account. Let $(\cY_t)_{t \geq 0}$ be an irreducible, reversible Markov chain on a state space $W$, with reversible measure $\mu$, represented as a random walk on an electrical network with vertex set $W$ and $c$ a family of conductances on the edges of this graph. We write hitting times as $\tau$, return times as $\tau^{+}$, with the corresponding subset or vertex as index. A classical result for reversible chains is that for $a \in W$ and $Z \subset W$, the probability to reach $Z$ before returning to $a$ can be expressed as
\begin{equation*}
	\mu(a) \bP_{a} \sbra{ \tau_Z < \tau_{a}^{+}} = \scC(a \leftrightarrow Z)
\end{equation*} 
where $\scC(a \leftrightarrow Z)$ is the effective conducance between $a$ and $Z$, which can be computed using network reductions. If $W$ is infinite, we can very well take $Z = \{ \infty \}$ thanks to a limit argument, to obtain 
\begin{equation}\label{eq:eff_conductance}
	\mu(a) \bP_{a} \sbra{\tau_{a}^{+} = \infty} = \scC(a \leftrightarrow \infty).
\end{equation}
In particular the chain is transient if and only if $\scC(a \leftrightarrow \infty) > 0$. Effective conductances satisfy a simple yet powerful monotonicity property, in that they are monotonous with respect to individual conductances. This is called Rayleigh's monotonicity principle, see p.35 of \cite{lyons2016probability}. This monotonicity can be made quantitative as follows. The proof is identical to that of \cite{lyons2016probability}[p.35]. 

\begin{lemma}\label{lem:rayleigh_monotonicity}
	Let $G$ be an infinite connected graph with vertex set $W$ and two sets of conductances $c_1, c_2$ Suppose there exists $\lambda > 0$ such that $c_1 \geq \lambda c_2$ edge-wise. Then for all $a \in W$
	\begin{equation*}
		\scC_{1}(a \leftrightarrow \infty) \geq \lambda \scC_{2}(a \leftrightarrow \infty).
	\end{equation*}
\end{lemma}

\begin{remark}
	If $G' \subset G$ is a subgraph, one can always consider the Markov chain on $G'$ to have state space $W$ as well by setting zero conductances outside $G'$. The lemma thus gives in particular quantitative comparisons between a reversible chain and the chain restricted to a subnetwork.
\end{remark}

The last required notion is that of branching number. The branching number $\mathrm{br} \ T$ is a parameter than can be associated to any tree $T$ that basically counts the average number of children per vertex. For instance for a $d$-regular tree, $d \geq 1$, the branching number is $d-1$. We refer to \cite{lyons2016probability} for a more general definition. We will only need the following facts, which can be found in Chapters 3 and 5 of the same reference (the first point is not proved but follows from the definition on the branching number): 

\begin{proposition}\label{prop:branching_number}
	\begin{enumerate}[label=(\roman*)]
	\item Given a tree $T$ and $k \geq 1$, let $T^{(k)}$ be the tree obtained by keeping only the vertices of $T$ which are at depth a multiple of $k$ from the root, and where vertices are joined by an edge if one is the ancestor of the other in $T$. Then $\mathrm{br} \ T^{(k)} = (\mathrm{br} \ T)^{k}$.
	\item \label{enum:SRW_br} The simple random walk on a tree $T$ is transient if and only if $\mathrm{br} \ T > 1$.
	\end{enumerate} 
\end{proposition}

\begin{proof}[Proof of Proposition \ref{prop:escape_reversible}]
	We prove that the root has uniformly lower bounded escape probability, but the same arguments can be applied to any vertex of the quasi-tree. Recall the equivalence relation $x \simeq \eta(x)$ on $\cV$ which identifies the endpoints of each long-range edge. Note that the chain $\cX$ is not reversible but only its projection $\bar{\cX}$ to $\cV / \sim$ is. Clearly, lower bounding the escape probabilities of $\cX$ or $\bar{\cX}$ is equivalent. Consider the graph obtained by pruning the long-range edges of $\cG_O$ and adding an edge $(\dagger, O)$ between an extra vertex $\dagger$ and the root. Then consider the Markov chain $\tilde{\cX}$ which goes from $\dagger$ to $O$ with probability $1$ and otherwise has the same transition probabilities as $\bar{\cX}$, with the probability of going from $O$ to $\dagger$ being that of $\bar{\cX}$ leaving the image of $\cG_O$ in the quotient. Since $\bar{\cX}$ is reversible, the chain $\tilde{\cX}$ is also a reversible chain. Letting $\tau_{\dagger}, \tau_{\dagger}^{+}$ denote the hitting and return time to $\dagger$ of the chain $\tilde{\cX}$, \eqref{eq:eff_conductance} implies
	\begin{equation*}
		\bfP_{O} \sbra{\tau_{\dagger} = \infty} = \bfP_{\dagger} \sbra{\tau_{\dagger}^{+} = \infty} = \frac{\scC_{\tilde{\cX}}(\dagger \leftrightarrow \infty)}{\mu(\dagger)}
	\end{equation*}
	where $\scC_{\tilde{\cX}}(\dagger \leftrightarrow \infty)$ is the effective conductance for the chain $\tilde{\cX}$, and $\mu$ is the invariant measure defined by the conductances $\tilde{c}$ of $\tilde{\cX}$, that is $\mu(x) = \sum_{y \in \cV} \tilde{c}(x,y)$ ($\mu$ does not need to be a probability measure here). By assumptions \ref{hyp:bdd_delta} and \ref{hyp:bdd_p}, conductances and degrees are bounded uniformly in $n$ thus so is $\mu(\dagger)$ and it suffices to lower bound the effective conductance.
	Now let $\cV'$ be the set of vertices in $\cG_O$ that are at distance at most $2$ from their centers and consider any spanning tree $\cT$ of the subgraph spanned by $\cV'$. Let $\tilde{\cT}$ be the graph spanned by its projection on $\cV / \sim$ together with the edge $(\dagger,O)$, which remains a tree, and $\scC_{\tilde{\cT}}(\dagger \leftrightarrow \infty)$ the effective conductance of the simple random walk on $\tilde{\cT}$. Then as conductances of $\tilde{\cX}$ are bounded  Lemma \ref{lem:rayleigh_monotonicity} implies that
	\begin{equation*}
		\scC_{\tilde{\cX}}(\dagger \leftrightarrow \infty) \geq \lambda \scC_{\tilde{\cT}}(\dagger \leftrightarrow \infty)
	\end{equation*}
	where $\lambda > 0 $ is a constant independent of $n$. 

	We now claim that $\scC_{\tilde{\cT}}(\dagger \leftrightarrow \infty)$ is bounded away from $0$ by a universal constant, which will prove the result. Observe that the tree $\cT$ must contain every long-range edge leaving from a vertex of $\cV'$. Hence $\cT$ has no leaves and every vertex other than $O$ has degree at least $2$. Furthermore, from Assumption \ref{hyp:cc3}, each small-range component in $\tilde{\cG}$ contains at least two vertices and every other which identifies with a component of $V_1$ actually contains at least three vertices. In these $V_1$ components, there is thus a vertex which is connected to a long-range edge and at least two other vertices. Consequently, we can see that every sequence of three consecutive vertices in $\tilde{\cT}$ contains at least one vertex with degree more than $3$. This implies that the power tree $\tilde{\cT}^{(3)}$, as defined in Proposition \ref{prop:branching_number}, contains the $2$-forward regular tree $T_2$, i.e.  the tree where every vertex has $2$ children. From Proposition \ref{prop:branching_number} $\tilde{\cT}$ contains thus a subtree with branching number $2^{1/3} > 1$, so the simple random walk on $\tilde{\cT}$ is transient. Using again Lemma \ref{lem:rayleigh_monotonicity}, we can in the end lower bound $\scC_{\tilde{\cT}}(O \leftrightarrow \infty)$ by a universal constant.
\end{proof}

\begin{remark}
	The proof shows more generally that any reversible Markov chain supported by a quasi-tree, (whether it makes half-integer time steps or not) is transient and has uniformly lower bounded escape probabilities, provided the quasi-tree is sufficiently branching.
\end{remark}

\section[Quasi-tree II: Markovian regeneration structure]{Analysis on the quasi-tree II: Markovian regeneration structure}\label{section:QT2}

Proposition \ref{prop:escape_reversible} implies in particular that the Markov chain $(\cX_t)_{t \geq 0}$ is almost surely transient. As a consequence, the shortest path from $O$ to $\cX_t$ eventually has to go through a unique sequence of long-range edges $(\xi_{i})_{i=1}^{\infty}$, which is called the \emph{loop-erased chain}. Among the edges of the loop-erased chain, some have the property of being crossed only once. Thanks to this property, these so-called \emph{regeneration edges} yield a Markov decomposition of the quasi-tree which we will use in the next section to prove concentration of the drift and entropy. The regeneration process will also be used later to compute the approximate invariant measure $\hat{\pi}$.  

\subsection{Markov renewal processes}\label{subsec:mrp}

We start with general results about Markov renewal processes that will be necessary in the sequel. The theory of Markov renewal processes is certainly not new, however we could not find references proving the results established in this section.

\begin{definition}\label{def:M1_property}
	Let $S$ be a countable state space and $E$ a Polish space. Consider a process $(Y,Z) = (Y_k, Z_k)_{k \geq 0}$ taking values in $S \times E$, satisfying for all $k \geq 1$, $y_0, \ldots y_{k} \in S, z_0, \ldots z_k \in E$, 
	\begin{equation}\tag{M1}\label{eq:M1_property}
		\bP \cond{Y_{k} = y_k, Z_k = z_k}{(Y_i,Z_i)_{i=1}^{k-1} = (y_i,z_i)_{i=1}^{k-1}} = \bP \cond{Y_1 = y_{k}, Z_1 = z_k}{Y_0 = y_{k-1}},
	\end{equation}
This process is thus a Markov chain with a stronger Markov property that the usual one, in that the dependence on the previous state occurs only through the $Y$-coordinate. In particular $Y=(Y_k)_{k \geq 0}$ is a Markov chain on $S$. 

\medskip

A \emph{Markov renewal process} is a process $(Y_k, T_k)_{k \geq 0}$ taking values in $S \times \bN$ such that $(Y_k, T_{k} - T_{k-1})_{k \geq 0}$ satisfies \eqref{eq:M1_property} and $T_k - T_{k-1} \geq 1$ a.s. for all $k \geq 1$, taking $T_{-1} := 0$. The delay $T_0$ can have arbitrary distribution. Then for each $t \geq 1$, set
\begin{equation*}
	Q_{t}(x,y) := \bP \cond{Y_1 = y, T_1 = t}{Y_0 = x, T_0 = 0} = \bP \cond{Y_1 = y, T_1 - T_0 = t}{Y_0 = x}.
\end{equation*}
The family $(Q_{t})_{t \geq 1}$ is called the transition kernels of the Markov renewal process. Notice that $Y$ has transition matrix $Q := \sum_{t \geq 1} Q_t$. 
\end{definition}

\begin{remark}\label{rk:mrp_notation}
	Like Markov's property, the \eqref{eq:M1_property} property does not put any constraint on the law of the initial pair $(Y_0, Z_0)$. If $\nu$ is a probability measure on $S \times E$ we write 
	\begin{equation*}
		\bP_{\nu} := \sum_{y \in S} \int \bP \cond{\cdot}{Y_0 = y, Z = z} \nu(y, dz).
	\end{equation*}
	On the other hand, \eqref{eq:M1_property} implies that $(Y_k, Z_k)_{k \geq 1}$ conditional on $Y_0$ is independent of $Z_0$. Thus if we are interested in a quantity that is measurable only with respect to $(Y_k, Z_k)_{k \geq 1}$, we will slightly abuse notation by writing $\bP_{u} := \bP \cond{\cdot}{Y_0 = u}$ (and write similarly for expectation), and by extension $\bP_{\mu} := \sum_{y \in S} \mu(y) \bP_{y}$ for a measure $\mu$ on $S$. In particular, note that if $\mu$ is an invariant measure for $Y$, the process $(Y_k, Z_k)_{k \geq 1}$ becomes stationary under $\bP_{\mu}$. In the sequel if $Y$ is positive recurrent, we only consider invariant measures which are probability distributions.
\end{remark}

The next results are specific to the setting where the process $Z$ takes integer values. In particular we state analogs of classical renewal theorems in the context of a Markov renewal process $(Y,T)$. The following generalizes the so-called elementary renewal theorem and can be proved in the same way.

\begin{proposition}\label{prop:lln_mrp}
	Let $(Y,T)$ be a Markov renewal process with state space $S$. Suppose $Y$ is positive recurrent with stationary distribution $\mu$ and $\max_{u \in S} \bE_u \sbra{T_1} < \infty$. Given $t \geq 0$ let
	\begin{equation*}
		N_t := \sup \{k \geq 0, T_k \leq t \}.
	\end{equation*}
	Then a.s.
	\begin{equation*}
		\lim_{k \rightarrow \infty} \frac{T_k}{k} = \bE_{\mu} \sbra{T_1 - T_0}
	\end{equation*}
	and
	\begin{equation*}
		\lim_{t \rightarrow \infty} \frac{N_t}{t} = \frac{1}{\bE_{\mu}\sbra{T_1 - T_0}}.
	\end{equation*}
\end{proposition}

\begin{proof}[Proof of Proposition \ref{prop:lln_mrp}]
	The laws of large numbers are consequences of the ergodic theorem applied to two different Markov chains.

	For $T_k$, consider the pair $(Y_k, T_k - T_{k-1})_{k \geq 0}$, with $T_{-1} := 0$, which by definition satisfies \eqref{eq:M1_property}. It was noted in Remark \ref{rk:mrp_notation} that $\bP_{\mu}$ is a stationary distribution for this chain. Hence the law of large numbers follows from the ergodic theorem applied to the ergodic averages of the projection onto the second coordinate.  

	Let us move to $N_t$. Consider the Markov chain $(U_{k})_{k \geq 0}$ on $S \times \bN$ defined by the transition probabilities
	\begin{equation*}
		K((x,0),(y,t-1)) = Q_{t}(x,y) \qquad K((x,t),(x,t-1)) = 1
	\end{equation*}
	for all $t \geq 1$ and $x,y \in S$. This chain can be thought of as representing $Y$ together with the waiting time until the next renewal: letting $\tau_{k}$ be the $k$-th successive hitting time of $S \times \{ 0 \}$, the $S$-coordinate of $(U_{\tau_k})_{k \geq 1}$ has the distribution of $(Y_k)_{k \geq 0}$, given $Y_0 = U_0$. The number of jumps $N_t$ made before time $t \geq 0$ is thus the number of such hitting times that occurred before $t$. By the ergodic theorem for Markov chains, one immediately gets
	\begin{equation*}
		\lim_{t \rightarrow \infty} \frac{N_t}{t} = \tmu(S \times \{ 0 \})
	\end{equation*}
	where $\tmu$ is the unique stationary measure of the Markov chain $(U_{k})_{ k \geq 0}$. As can be checked easily, it is given by:
	\begin{equation}
		\tmu(x,t) = \frac{\bP_{\mu}\sbra{X_1 = x, T_1 - T_0 > t}}{\bE_{\mu} \sbra{T_1 - T_0}}
	\end{equation}
	for all $x \in S, t \geq 0$. In particular $\tmu(S \times \{ 0 \}) = \bE_{\mu} \sbra{T_1 - T_0}^{-1}$.
\end{proof}

The following two propositions establish mixing results for a Markov renewal process $(Y,T)$. These are not necessary for the asymptotic analysis on the quasi-tree but will be used to essentially compute the annealed laws of $\cX$ and $X$ and obtain the value of the limiting measure $\hat{\pi}$ in Proposition \ref{prop:nice_approx}. Proposition \ref{prop:markov_renewal_thm} is an analog of the classical renewal theorem, with a quantitative bound on the speed of convergence. To prove it, we first establish a stronger result in Proposition \ref{prop:mixing_mrp}, namely a mixing property for the whole process $(Y,T)$, where $T$ is allowed here to take negative values. Since $T$ is unbounded and does not converge to an invariant distribution, mixing is here understood as the fact that we can couple processes started from different distributions to make them coalesce.

\begin{proposition}\label{prop:mixing_mrp}
	Let $(Y_k, T_{k} - T_{k-1})_{k \geq 0}$ satisfy \eqref{eq:M1_property}, with $T_{-1} := 0$ and $T_{k} - T_{k-1} \in \bZ$ a.s. for all $k \geq 0$. Suppose $Y$ is positive recurrent, irreducible and aperiodic. Let $Q_{t}(x,y) := \bP \cond{Y_1 = y, T_1 - T_0 = t}{Y_0 = x}$ for all $t \in \bZ$ and
	\begin{equation*}
		\alpha := \min_{x \in S} \sum_{\substack{t \in \bZ \\ y \in S}} \left( Q_{t}(x,y) \wedge Q_{t+1}(x,y) \right).
	\end{equation*}
	Consider two starting probability distributions $\nu_1, \nu_2$ on $S \times \bZ$. Given $\e \in (0,1)$, let $\tmix^{(Y)}(\e)$ denote the $\e$-mixing time of $(Y_k)_{k \geq 0}$ and let $K(\e) \geq 0$ be the minimal integer such that 
	\begin{equation*}
		\sup_{\nu} \bP_{\nu} \sbra{\abs{T_0 - T'_0} > K(\e)} \leq \e,
	\end{equation*}
	where the supremum is over all couplings $((Y_0,T_0), (Y'_0,T'_0))$ of $\nu_1$ and $\nu_2$.
	Then for all $\e \in (0,1)$, there exists $C(\e) \geq 0$ such that 
	\begin{equation*}
		\TV{\bP_{\nu_1} \sbra{(Y_k, T_k) = \cdot} - \bP_{\nu_2} \sbra{(Y_k, T_k) = \cdot}} \leq \e
	\end{equation*}
	for all 
	\begin{equation}\label{eq:mixing_mrp}
		k \geq \frac{C(\e)}{\alpha} \left( \tmix^{(Y)}(\e) \max_{u \in S} \bE_{u} \abs{T_1} + K(\e) \right)^{2}.
	\end{equation}
\end{proposition}

\begin{proof}[Proof of Proposition \ref{prop:mixing_mrp}]
	The proof is based on coupling arguments. If the $Y_k$ are i.i.d., then $T_k$ becomes the sum of i.i.d. random variables on $\bZ$, i.e. a random walk on $\bZ$. Two random walks on $\bZ$ identically distributed but started from a distance $A$ apart can be coupled to meet at a random time $\tau$ which satisfies $\bP \sbra{\tau > k} \leq C A / \sqrt{k}$ for some constant $C > 0$. The argument can be found in \cite{lindvall1996coupling} and \cite{lindvall2002lectures}[II.14], where the coupling is called Mineka coupling. In the general case where $Y_k$ are not i.i.d., the idea is to first couple the chains $Y_k$ started at different states to make them coincide, after which one can adapt the Mineka coupling to make the subsequent variables $T_k$ coalesce.

	Suppose $((Y,T), (Y',T'))$ is a coupling of two versions of the process $(Y_k, T_k)_{k \geq 0}$ started at $\nu_1$ and $\nu_2$ respectively. Let $\tau := \inf \{k \geq 0: (Y_k, T_k) = (Y'_k, T'_k) \}$. If the coupling is such that the two processes coincide after the coalescence time $\tau$, then one has 
	\begin{equation*}
		\TV{\bP \sbra{(Y_k, T_k) = \cdot} - \bP \sbra{(Y'_k, T'_k) = \cdot}} \leq \bP \sbra{\tau > k}.
	\end{equation*}
	We now specify such a coupling.	Let $\e \in (0,1)$ and $k_0 := \tmix^{(Y)}(\e)$. The coupling is actually started at $k_0$: couple $(Y_{k_0}, T_{k_0})$ and $(Y'_{k_0}, T'_{k_0})$ in order to have an optimal coupling of $Y_{k_0}$ and $Y'_{k_0}$. Thus
	\begin{equation}\label{eq:optimal_coupling_Y}
		\bP \sbra{Y_{k_0} \neq Y'_{k_0}} = \TV{\bP \sbra{Y_{k_0} = \cdot} - \bP \sbra{Y'_{k_0} = \cdot}} \leq 2 \TV{\bP \sbra{Y_{k_0} = \cdot} - \mu} \leq 2 \e,
	\end{equation} 
	where $\mu$ is the stationary measure of $Y$. Then for all $k \geq k_0$, conditional on $Y_k = Y'_k$ draw $Y_{k+1}, Y'_{k+1}, S_{k+1}, S'_{k+1}$ according to the distribution:
	\begin{align*}
		&\bP \cond{Y_{k+1} = Y'_{k+1} = y, S_{k+1} = t-1, S'_{k+1} = t}{Y_k = Y'_k = x} = \alpha_{t-1}(x,y) \\
		&\bP \cond{Y_{k+1} = Y'_{k+1} = y, S_{k+1} = t, S'_{k+1} = t-1}{Y_k = Y'_k = x} = \alpha_{t-1}(x,y) \\
		&\begin{array}{l l} \bP \cond{Y_{k+1} = Y'_{k+1} = y, S_{k+1} = t, S'_{k+1} = t}{Y_k = Y'_k = x} = &Q_{t}(x,y) - \alpha_{t-1}(x,y) \\
		&- \alpha_{t}(x,y) \end{array}
	\end{align*}
	writing $\alpha_{t}(x,y) := Q_{t}(x,y) \wedge Q_{t+1}(x,y)$.
	It is readily seen that 
	\begin{equation*}
		\bP_{x} \cond{Y_{k+1} = y, S_{k+1} = t}{Y_{k} = x} = Q_{t}(x,y) = \bP_{x} \cond{Y'_{k+1} = y, S'_{k+1} = t}{Y'_{k} = x}.
	\end{equation*}
	Therefore if $Y_k = Y'_k$ setting $T_{k+1} := T_k + S_{k+1}$, $T'_{k+1} := T'_k + S'_{k+1}$ yields a coupling of $(Y_{k+1},T_{k+1}) , (Y'_{k+1}, T'_{k+1})$ with $Y_{k+1} = Y'_{k+1}$. Consequently, this can be used to couple the two processes $(Y,T), (Y',T')$ for every step after the coalescence of $Y$ and $Y'$, so that the $Y$-coordinate stays identical. Once the $T$-coordinates coalesce, we couple the two processes so that they coincide indefinitely. We now bound the tail of the stopping time $\tau$ under this coupling.

	Let $\tau_1 := \inf \{k \geq k_0: Y_{k} = Y'_{k} \}$ and for all $k \geq 0$, $Z_k := \sum_{i=\tau_1}^{k} S_{i} - S'_{i}$. Observe that $(Y_{\tau_1 + k}, Z_{k} - Z_{k-1})_{k \geq 0}$ (with $Z_{-1} := 0$) satisfies the \eqref{eq:M1_property} property. In addition, for every $k \geq 0$, conditional on $Y_{\tau_1 + k} = x$, $Z_{k+1} - Z_k$ has symmetric distribution in $\{-1, 0 ,1 \}$ a.s. and 
	\begin{equation}\label{eq:stoch_bd_alpha}
		\bP \cond{Z_{k+1} - Z_k = 1}{Y_k = x} = \sum_{t \geq 1} \sum_{y \in S} \alpha_{t}(x,y) \geq \alpha.
	\end{equation}
	On the event $\{ \tau_{1} = k_0 \}$, for all $k \geq 0$ we can decompose
	\begin{equation*}
		T_{k_0 + k} - T'_{k_0+ k} = T_{k_0} - T'_{k_0} + Z_k
	\end{equation*}
	which implies that $\tau = k_0 + \inf \{ k \geq 0: Z_k = T'_{k_0} - T_{k_0} \}$. 
	Letting $\tilde{\tau}_{a}$ be the hitting time of $a \in \bZ$ by the process $Z$, we deduce from Markov's property that for all $k \geq 0$ and $A > 0$,
	\begin{align*}
		\bP \sbra{\tau > k_0 + k} \leq \bP \sbra{Y_{k_0} \neq Y'_{k_0}} + \bP \sbra{\abs{T_{k_0} - T'_{k_0}} > A} + \max_{a \in [-A, A]} \max_{u \in S} \bP_{(u,0)} \sbra{\tilde{\tau}_{a} > k}.
	\end{align*}
	By \eqref{eq:optimal_coupling_Y} the first term is smaller than $2 \e$. For the second term, let $A := K(\e) + B$ and bound 
	\begin{align*}
		\bP \sbra{\abs{T_{k_0} - T'_{k_0}} > K(\e) + B} &\leq \bP \sbra{\abs{T_0 - T'_0} > K(\e) } + \bP \sbra{\abs{(T_{k_0} - T_0)  - (T'_{k_0} - T'_0)} > B}. 
	\end{align*}
	By definition of $K(\e)$ the first term is bounded by $\e$. The second can be bounded by triangle inequality and Markov's inequality to obtain
	\begin{equation*}
		\bP \sbra{\abs{T_{k_0} - T'_{k_0}} > K(\e) + B} \leq \e + \frac{2 k_0 \max_{u} \bE_u \sbra{T_1}}{B} \leq 2 \e
	\end{equation*}
	for $B := 2 k_0 \max_{u \in S} \bE_{u} \sbra{T_1} / \e$. 
	The last term is bounded as follows. First, notice that since the variables $Z_k$ are bounded by $1$ in absolute value, the maximal probability is obtained for $a = \pm A$, and from the symmetry of $Z_{k+1} - Z_k$ we can suppose $a = A$. Then we claim that there exists a constant $C \geq 1$, such that for all $k \geq 0$
	\begin{equation}\label{eq:tail_coupling_tauA}
		\max_{u \in S} \bP_{(u,0)} \sbra{\tilde{\tau}_A > k} \leq \frac{C A}{\sqrt{\alpha k}}.
	\end{equation}
	Provided the claim holds, we get that the right-hand side is below $\e$ for $k \geq C^2 A^2 / (\alpha \e^2)$. Combining with the previous choices of $k_0$ and $A$, this yields eventually that 
	\begin{equation*}
		\TV{\bP \sbra{(Y_k, T_k) = \cdot} - \bP \sbra{(Y'_k, T'_k) = \cdot}} \leq 5 \e
	\end{equation*}
	for 
	\begin{equation*}
		k \geq \tmix^{(Y)}(\e) + \frac{C^2 (K(\e) + 2 \tmix^{(Y)}(\e) \max_{u \in S} \bE_{u} \sbra{T_1})^2}{\alpha \e^4},
	\end{equation*}
	proving the result.
	
	Let us now prove \eqref{eq:tail_coupling_tauA}: for all $k \geq 1$, let $N_k := \abs{\{ i \leq k: Z_{i} - Z_{i-1} \neq 0 \}}$ and define $\tZ_k$ as the sum of the first $k$ non-zero variables $Z_{i} - Z_{i-1}$. By \eqref{eq:stoch_bd_alpha}, for all $u \in S$, $N_k$ stochastically dominates a binomial random variable $\mathrm{Bin}(k, 2 \alpha)$ and from the symmetry of the increments the process $\tZ$ is the simple random walk on $\bZ$. Letting $\tau^{(\mathrm{SRW})}_A$ denote the hitting time of $A$ by $\tZ$, \cite{feller1967introduction}[III 7.5] shows
	\begin{equation*}
		\bP_{0} \sbra{\tau^{(\mathrm{SRW})}_A = k} = \frac{A}{n} \binom{k}{\frac{A+k}{2}} 2^{-k} \leq C \frac{A}{k^{3/2}},
	\end{equation*}
	for some constant $C > 0$ independent of $k$ and $A$. Summing over $k$ implies
	\begin{equation*}
		\bP_{0} \sbra{\tau^{(\mathrm{SRW})}_A > k} \leq \frac{C' A}{\sqrt{k}} 
	\end{equation*}
	for some other constant $C' > 0$. 
	Thus
	\begin{align*}
		\bP_{(u,0)} \sbra{\tilde{\tau}_A > k} &\leq \bP_{u} \sbra{\tilde{\tau}_A > k, N_k \geq \alpha k} + \bP_{u} \sbra{N_k < \alpha k} \\
		&\leq \bP_{0} \sbra{\tau^{(\mathrm{SRW})}_A > \lfloor \alpha k \rfloor } + \bP \sbra{\abs{\mathrm{Bin}(k, 2 \alpha) - 2 \alpha k} > \alpha k } \\
		&\leq \frac{C' A}{\sqrt{\alpha k}} + \frac{2 (1-2 \alpha)}{\alpha k}
	\end{align*}
	using Chebychev's inequality, which proves the claim.
\end{proof}

\begin{proposition}\label{prop:markov_renewal_thm}
	Let $(Y,T)$ be a Markov renewal process with state space $S$, such that $Y$ is positive recurrent with stationary distribution $\mu$, irreducible and aperiodic, and $\max_{u \in S} \bE_{u} \sbra{T_{1}^{2}} < \infty$. Let $\alpha$ be as in Proposition \ref{prop:mixing_mrp} and suppose that $\alpha > 0$. Then for all probability distribution $\nu$ on $S \times \bN$, 
	\begin{equation*}
		\sum_{y \in S} \abs{\bP_{\nu}\sbra{\exists k \geq 0: Y_k = y, T_k = t} - \frac{\mu(y)}{\bE_{\mu}\sbra{T_1 - T_0}}} \xrightarrow[t \rightarrow \infty]{} 0.
	\end{equation*}
	More precisely, given $\e \in (0,1)$, let $K_{\nu}(\e)$ be the minimal integer such that
	\begin{equation*}
		\bP_{\nu} \sbra{T_0 > K_{\nu}(\e)} \leq \e.
	\end{equation*}
	There exists $C(\e) > 0$ such that for all probability distributions $\nu$ on $S \times \bN$,
	\begin{equation*}
		\sum_{y \in S} \abs{\bP_{\nu}\sbra{\exists k \geq 0: Y_k = y, T_k = t} - \frac{\mu(y)}{\bE_{\mu}\sbra{T_1-T_0}}} \leq \e
	\end{equation*}
	for all
	\begin{equation}\label{eq:tmix_renewal}
		t \geq \frac{C(\e)}{\alpha} \left( \tmix^{(Y)}(\e) \max_{u \in S} \bE_{u} \sbra{T_1} + K_{\nu}(\e) + \frac{\bE_{\mu} \sbra{(T_{1}-T_0)^{2}}}{\bE_{\mu} \sbra{T_1-T_0}}  \right)^{2} \max_{v} \bE_{v} \sbra{T_1}
	\end{equation}
	where $\tmix^{(Y)}(\e)$ denotes the $\e$-mixing time of $Y$.
\end{proposition}

\begin{proof}[Proof of Proposition \ref{prop:markov_renewal_thm}]
	Consider the Markov chain $(U_k)_{k \geq 0}$ on $S \times \bN$ considered in the proof of Proposition \ref{prop:lln_mrp}. Then for all probability measure $\nu$ on $S \times \bN$ and $t \geq 0$
	\begin{equation}\label{eq:renewal_proba_Z}
		u_{t}(\nu, y) := \bP_{\nu} \sbra{\exists k \geq 0: Y_k = y, T_k = t} = \bP_{\nu} \sbra{U_t = (y,0)}.
	\end{equation} 
	It was proved in the proof of Proposition \ref{prop:lln_mrp} that $U$ has unique invariant measure given by $\tmu(x,t) = \bP_{\mu} \sbra{Y_1 = y, T_1 - T_0 > t} / \bE_{\mu} \sbra{T_1 - T_0}$. Thus if one can prove that $U$ is aperiodic, the convergence theorem for Markov chains directly implies that $u_{t}(\nu,y) \rightarrow \tmu(y,0) = \mu(y) / \bE_{\mu} \sbra{T_1 - T_0}$. It is easily proved that $U$ is aperiodic from the assumptions that $Y$ converges to $\mu$ and $\alpha > 0$. Actually, we do not even need to check aperiodicity. The coupling argument that we use afterwards implies the convergence, which implies aperiodicity. 
	
	Let $\e \in (0,1)$. Let $k_0 = k_0(\e)$ be the right-hand side of \eqref{eq:mixing_mrp}. By Proposition \ref{prop:mixing_mrp} there exists a coupling of two versions of $(Y,T)$ started at $\nu$ and $\tmu$, with coalescence time $\kappa$ such that 
	\begin{equation*}
		\bP \sbra{\kappa > k_0} \leq \e.
	\end{equation*}
	These processes can in turn be coupled in an obvious way with two versions $U^{(\nu)}, U^{(\tmu)}$ of $U$ started with distributions $\nu$ and $\tmu$ respectively. Then notice that the two processes $U^{(\nu)}, U^{(\tmu)}$ coincide after time $T_{\kappa}$, therefore $T_{\kappa}$ is a coupling time and we deduce
	\begin{equation*}
		\TV{\bP_{\nu} \sbra{U_t = \cdot} - \tmu} \leq \bP \sbra{T_{\kappa} > t}.
	\end{equation*}
	Since the sequence $(T_i)_{i \geq 0}$ is increasing, the previous tail probability can be bounded as
	\begin{equation*}
		\bP \sbra{T_{\kappa} > t} \leq \bP \sbra{\kappa > k_0} + \bP_{\nu} \sbra{T_{k_0} > t}.
	\end{equation*}
	The first term is bounded by $\e$. The second term can be bounded by Markov's inequality as
	\begin{equation*}
		\bP_{\nu} \sbra{T_{k_0} > t} \leq \frac{k_0 \max_{u \in S} \bE_{u} \sbra{T_1} }{t} \leq \e
	\end{equation*}
	for $t \geq k_0 \max_{u \in S} \bE_{u} \sbra{T_1} / \e$. Let us now specify the value of $k_0$. We first bound for the worst possible coupling of $\nu$ and $\tmu$,
	\begin{align*}
		\bP \sbra{\abs{T_{0} - T'_{0}} > K_{\nu}(\e) + k} &\leq \bP_{\nu} \sbra{T_0 > K_{\nu}(\e)} + \bP_{\tmu} \sbra{T_0 > k} \\
		&\leq \e + \bE_{\mu} \sbra{(T_1 - T_0 - k - 1 )_{+}} / \bE_{\mu} \sbra{T_1 - T_0}
	\end{align*}
	using the definition of $K_{\nu}(\e)$ and the expression of $\tmu$. Then by Markov's inequality
	\begin{align*}
		\bE_{\mu} \sbra{(T_1 - T_0 - k - 1 )_{+}} &= \sum_{t \geq k +1} \bP_{\mu} \sbra{T_1 - T_0 > t} \\
		&\leq \sum_{t \geq k + 1} \frac{\bE_{\mu} \sbra{(T_{1}-T_0)^{2}}}{t^2} \leq \frac{C \, \bE_{\mu} \sbra{(T_{1}-T_0)^{2}}}{k}
	\end{align*}
	for some constant $C > 0$. Thus $\bP \sbra{\abs{T_{0} - T'_{0}} > K_{\nu}(\e) + k} \leq 2 \e$ for all 
	\begin{equation*}
		k \geq  C \, \bE_{\mu} \sbra{(T_{1}-T_0)^{2}} / (\bE_{\mu} \sbra{T_1 - T_0} \e),
	\end{equation*}
	which gives the value of $K(\e)$ in \eqref{eq:mixing_mrp} and 
	\begin{equation*}
		k_0 \leq \frac{C(\e)}{\alpha} \left( \tmix^{(Y)}(\e) \max_{u \in S} \bE_{u} \abs{T_1} + K_{\nu}(\e) + \frac{\bE_{\mu} \sbra{(T_{1}-T_0)^{2}}}{\bE_{\mu} \sbra{T_1-T_0}} \right)^{2}.
	\end{equation*}
	We deduce eventually that taking $t$ as \eqref{eq:tmix_renewal} yields 
	\begin{equation*}
		\TV{\bP_{\nu} \sbra{U_t = \cdot} - \tmu} \leq \e
	\end{equation*}
	for large enough $C(\e)$.
	Finally using \eqref{eq:renewal_proba_Z} yields
	\begin{equation*}
		\frac{1}{2} \sum_{y \in S} \abs{u_t(\nu, y) - \frac{\mu(y)}{\bE_{\mu} \sbra{T_1 - T_0}}} \leq \e.
	\end{equation*}
\end{proof}

The last tool we introduce about Markov renewal processes is a variance bound that uses spectral arguments. Let us gather a few known facts about spectral theory for Markov chains. Consider a Markov chain $Y = (Y_k)_{k \geq 0}$ on a countable state space $S$ with transition kernel $Q$, which is irreducible, positive recurrent with stationary distribution $\mu$.

Let $\ell^{2}(\mu)$ be the Hilbert space of real-valued functions on $S$ which are square-integrable with respect to the measure $\mu$, equipped with the inner product
\begin{equation*}
	\innerprod{f}{g}_{\mu} := \sum_{x \in S} \mu(x) f(x) g(x).
\end{equation*}
$Q$ defines a contracting linear operator on $\ell^{2}(\mu)$ by $Q f(x) := \sum_{y \in S} Q(x,y) f(y)$. Its spectrum is defined as 
\begin{equation*}
	\mathrm{Spec}(Q) := \{ \lambda \in \bC \ | \ \lambda I - Q \text{ is not invertible as a bounded linear operator} \}.
\end{equation*}
The adjoint operator of $Q^{\ast}$ is given by 
\begin{equation*}
	Q^{\ast}(y,x) = \frac{\mu(x) Q(x,y)}{\mu(y)} 
\end{equation*}
for all $x,y \in S$. Reversibility of $Q$ with respect to $\mu$ is equivalent to self-adjointness of $Q$. In this case, the spectrum of $Q$ is included in the interval $[-1, 1]$. Writing $\II$ for the constant function equal to $1$, the fact that $\mu$ is a probability measure implies that $\II \in \ell^{2}(S)$ and is an eigenvector of $Q$ associated with the eigenvalue $1$. The absolute spectral gap of $Q$ is then defined as 
\begin{equation*}
	\gamma := 1 - \sup \{ \abs{\lambda}, \lambda \in \mathrm{Spec}(Q) \}
\end{equation*}
if $1$ has multiplicity $1$ and $\gamma := 0$ otherwise.
It is well-known that the absolute spectral gap is related to mixing properties of $Y$, see for instance \cite[Chap. 12]{levin2017markov}. In the non-reversible case, spectral arguments can be applied by considering reversibilizations of the chain. From a theoretical point of view, the optimal parameter to consider is the so-called \emph{pseudo spectral gap}, defined as 
\begin{equation*}
	\gamma_{ps} := \max_{k \geq 1} \left\{ \frac{\gamma((Q^{\ast})^{k} Q^{k})}{k} \right\}.
\end{equation*}
If $Q$ is reversible obviously $1 - \gamma_{ps} = (1- \gamma)^2$. The pseudo spectral gap is a rather natural quantity introduced in \cite{paulin2015concentration} but it may have appeared beforehand in other places under different names. In \cite{bordenave2021cutoff}, the quantity $1 - \gamma_{ps}$ is considered under the name singular radius.

As for the classical spectral gap, the pseudo spectral gap is intimately related with the mixing properties of the chain, as shown by the following proposition.

\begin{proposition}[{\cite[Prop. 3.4]{paulin2015concentration}}]\label{prop:psg_mixing}
	Let $(Y_k)_{k \geq 0}$ be an irreducible, positive recurrent Markov chain on a countable state space $S$, with stationary distribution $\mu$ and $\e$-mixing time $\tmix(\e)$. Suppose it is uniformly ergodic, in the sense that there exists $C > 0$ and $\rho \in (0,1)$ such that 
	\begin{equation*}
		\sup_{x \in S} \TV{ \bP_{x} \sbra{Y_t = \cdot} - \mu } \leq C \rho^{t}.
	\end{equation*}
	Then for all $\e \in [0,1)$, 
	\begin{equation*}
		\gamma_{ps} \geq \frac{1- \e}{\tmix(\e)}.
	\end{equation*} 
	Furthermore if $S$ is finite,
	\begin{equation*}
		\tmix(\e) \leq \frac{1 + 2 \log((2 \e)^{-1}) + \log(\mu_{\min}^{-1})}{\gamma_{ps}}
	\end{equation*}
	where $\mu_{\min} := \min_{x \in S} \mu(x)$.
\end{proposition}

The pseudo spectral gap can also be used to precisely handle correlations between different steps of a Markov chains. Thus we obtain the following variance bound for Markov processes with the \eqref{eq:M1_property} Markov property. 

\begin{proposition}\label{prop:variance_markov}
	Let $S$ be countable and $E$ a Polish space. Let $(Y_k, Z_k)_{k \geq 0}$ be a process on $S \times E$ satisfying \eqref{eq:M1_property}. Suppose that $Y$ is irreducible, positive recurrent, with invariant measure $\mu$ and pseudo spectral gap $\gamma_{ps}$.
	Let $(f_i)_{i \geq 1}$ be a family of functions such that for all $i \geq 1$, $f_i: S \times E \rightarrow \bR$ and $\max_{u \in S} \bE_{u} \sbra{f_i (Y_1, Z_1)^{2}} < \infty$. Then for all $k \geq 1$
	\begin{equation*}
		\Var_{\mu} \sbra{ \sum_{i=1}^{k} f_{i}(Y_i, Z_i) } \leq \frac{6}{\gamma_{ps}} \sum_{i=1}^{k} \Var_{\mu} \sbra{f_{i}(Y_1, Z_1)}.
	\end{equation*}
\end{proposition}

The previous result can be applied to the case of a Markov renewal process $(Y_k, T_k)_{k \geq 0}$, to get that 
\begin{equation}\label{eq:variance_mrp}
	\Var_{\mu} \sbra{T_k - T_0)} \leq \frac{6}{\gamma_{ps}} k \Var_{\mu} \sbra{T_1 - T_0}.
\end{equation}
for all $k \geq 1$, provided $\max_u \bE_u \sbra{(T_1 - T_0)^2} < \infty$.

\begin{proof}
	First, observe that if $(Y,Z)$ satisfies \eqref{eq:M1_property} so does $(Y,(Y,Z))$. Thus up to changing the second coordinate of the process considered it suffices to prove the result for functions $f_i$ of $Z_i$ only.

	We use the same arguments as in \cite{paulin2015concentration}.	Let $Q$ denote the transition kernel of the chain $Y$ and given $i \in [m]$, $x,y \in S$ let 
	\begin{equation*}
		Q_{i}(x,y) := \bE_{x} \sbra{f_i(Z_1) \II_{Y_1 = y}}.
	\end{equation*}
	Using matrix notations, for all $i \in [m], x \in S$, $Q_i \II (x) = \sum_{y \in S} Q_i(x,y) = \bE_{x} \sbra{f_i(Z_1)}$. We can suppose without loss of generality that $\bE_{\mu}\sbra{f_i(Z_i)} = \bE_{\mu}\sbra{f_i(Z_1)} = 0$, which can be written matricially as 
	\begin{equation}\label{eq:fi_centered}
		\mu Q_i \II = 0.
	\end{equation}
	The random variable $\sum_{i=1}^{m} f_i(Z_i)$ is thus also centered and from stationarity one has
	\begin{align*}
		\Var_{\mu} \sbra{\sum_{i=1}^{m} f_i(Z_i)} &= \sum_{i,j=1}^{m} \bE_{\mu} \sbra{f_i(Z_i) f_j(Z_j)} \\
		&= \sum_{i,j = 1}^{m} \bE_{\mu} \sbra{f_i(Z_1) f_j(Z_{\abs{j-i}+1})}.
	\end{align*}
	For any $l \geq 1$
	\begin{align*}
		\bE_{\mu} \sbra{f_i(Z_1) f_j(Z_{l+1})} &= \sum_{x,y \in S} \mu(x)  Q_{i}(x,y)  \bE_y \sbra{f_j(Z_{l})} \\
		&= \sum_{x,y,z \in S} \mu(x) Q_{i}(x,y) Q^{l-1} Q_{j}(y,z) \\
		&= \innerprod{Q_i Q^{l-1} Q_j \II}{\II}_{\mu} \\
		&= \innerprod{Q_i \left( Q^{l-1} - \II \mu \right) Q_j \II}{\II}_{\mu}
	\end{align*}
	where in the last line we used \eqref{eq:fi_centered}. For any $l \geq 1$, since $\mu$ is stationary one has $Q^{l} - \II \mu = (Q - \II \mu)^{l}$, thus by Cauchy-Schwarz inequality
	\begin{align*}
		\innerprod{Q_i (Q - \II \mu)^{l-1} Q_j \II}{\II}_{\mu} &= \innerprod{(Q - \II \mu)^{l-1} Q_j \II}{Q_{i}^{\ast} \II}_{\mu} \\
		&\leq \norm{(Q - \II \mu)^{l-1}} \norm{Q_j \II} \norm{Q_{i}^{\ast} \II}.
	\end{align*}
	Then by Jensen's inequality, 
	\begin{equation*}
		\norm{Q_{j} \II}^2 = \sum_{x \in S} \mu(x) \abs{\bE_{x} \sbra{f_{j}(Z_1)}}^{2} \leq \bE_{\mu} \sbra{f_{j}(Z_1)^{2}} = \Var_{\mu}(f_j(Z_1)).
	\end{equation*}
	Similarly, 
	\begin{equation*}
		Q_{i}^{\ast} \II (y) = \sum_{x \in S} \frac{\mu(x) Q_i(x,y)}{\mu(y)} = \bE_{\mu} \cond{f_{i}(Z_1)}{Y_1 = y} 
	\end{equation*}
	so 
	\begin{align*}
		\norm{Q_{i}^{\ast} \II}^2 \leq \sum_{y \in S} \mu(y) \bE_{\mu} \cond{f_i(Z_1)^{2}}{Y_1= y} = \bE_{\mu} \sbra{f_{i}(Z_1)^{2}} = \Var_{\mu} \sbra{f_i(Z_1)^{2}}.
	\end{align*}
	Hence
	\begin{align*}
		\bE_{\mu} \sbra{f_i(Z_1) f_{j}(Z_{l+1})} &\leq \norm{(Q - \II \mu)^{l-1}} \Var_{\mu}\sbra{f_i(Z_1)}^{1/2} \Var_{\mu}\sbra{f_j(Z_1)}^{1/2} \\
		&\leq \frac{1}{2} \norm{(Q - \II \mu)^{l-1}} \left( \Var_{\mu}\sbra{f_i(Z_1)} +  \Var_{\mu}\sbra{f_j(Z_1)} \right).
	\end{align*}

	It remains to sum over $i, j$. Let ${k} \geq 1$ be such that $\gamma(Q^{k} (Q^{k})^{\ast}) = \gamma_{ps} k$. Then by observing that $\norm{(Q - \II \mu)^{k}}^{2} = 1 - \gamma(Q^{k} (Q^{k})^{\ast})$ one can deduce
	\begin{align*}
		\norm{(Q - \II \mu)^{l}} &\leq \norm{(Q - \II \mu)^{k}}^{\lfloor l / k \rfloor} \\
		&= (1 - k \gamma_{ps})^{\lfloor l / k \rfloor / 2}
	\end{align*}
	Consequently
	\begin{align*}
		&\sum_{i \neq j = 1}^{m} \bE_{\mu} \sbra{f_i(Z_1) f_j(Z_{\abs{j-i}+1})} \\
		&\qquad \leq \frac{1}{2} \sum_{i \neq j = 1}^{m} \norm{(Q - \II \mu)^{\abs{j - i} - 1}} \left( \Var_{\mu}\sbra{f_i(Z_1)} +  \Var_{\mu}\sbra{f_j(Z_1)} \right) \\
		&\qquad \leq 2 \sum_{i=1}^{m} \sum_{l = 0}^{\infty} \norm{(Q- \II \mu)^{l}} \Var_{\mu}\sbra{f_i(Z_1)} \\
		&\qquad \leq 2 \sum_{i=1}^{m} \Var_{\mu}\sbra{f_i(Z_1)} \sum_{l = 0}^{\infty} (1 - k \gamma_{ps})^{\lfloor l / k \rfloor / 2} \\
		&\qquad \leq 4 \sum_{i=1}^{m} \Var_{\mu}\sbra{f_i(Z_1)} \sum_{l = 0}^{\infty} (1 - k \gamma_{ps})^{\lfloor l / k \rfloor} \\
		&\qquad = \frac{4}{\gamma_{ps}} \sum_{i=1}^{m} \Var_{\mu}\sbra{f_i(Z_1)}
	\end{align*}
	and combining with the diagonal terms $i = j$ the result follows as $\gamma_{ps} \leq 2$.
\end{proof}

\subsection{Regeneration structure}

Let us come back to the setting of quasi-trees.

\begin{definition}
	A time $t \in \bN + 1/2$ is called a regeneration time if $(\cX_{t-1/2}, \cX_{t})$ is a long-range edge crossed for the first and last time at time $t$.
\end{definition}

From the uniform lower bound on escape probabilities, it is easy to deduce that there is an infinite number of regeneration times and levels. Thus frow now on, let $T_0 := 0, Y_0 := \iota(O), L_0 := 0$ and $(T_k)_{k \geq 1}$ be the sequence of successive regeneration times of $\cX$. For all $k \geq 1$, define 
\begin{equation*}
	Y_k := \iota(\cX_{T_{k} - 1/2}) \qquad L_k := d(O, \cX_{T_k}).
\end{equation*}
$(Y_{k})_{k \geq 1}$ is the Markov chain on $V$ that dictates the law of the environments between successive regeneration times. To describe this Markov chain, introduce the measures
\begin{equation*}
	\bQ_{u} = \bP \cond{\cdot}{\cX_{1/2} = \eta(O), \iota(O) = u, \tau_{O} = \infty }
\end{equation*}
for all $u \in V$.

\begin{remark}\label{rk:comparison_Qu}
	Obviously, the lower bound on escape probabilities of Proposition \ref{prop:escape_reversible} also holds under the law $\bQ_{u}$, for any $u \in V$. Furthermore, the law of the centers added to the quasi-trees remains essentially uniform in the sense that for all $v \in V$ such that $V(v) \neq V(u)$
	\begin{equation}\label{eq:etaO_unif}
		\bQ_u \sbra{\iota(\eta(O)) = v} = \Theta(1/n).
	\end{equation}
	Indeed if $q_0 > 0$ is a constant such that all escape probabilities are lower bounded by $q_0$ for any realization of the quasi-tree then  
	\begin{equation*}
		\bQ_{u} \sbra{\iota(\eta(O)) = v} \leq q_0^{-1} \bP \cond{\eta(O) = v}{\iota(O) = u} = q_0^{-1} / n
	\end{equation*}
	and 
	\begin{align*}
		\bQ_{u} \sbra{\iota(\eta(O)) = v} &\geq \bE \cond{\II_{\iota(\eta(O)) = v}\II_{\tau_O = \infty}}{\iota(O) = u, \cX_{1/2} = \eta(O)} \\
		&\geq \bE \cond{\II_{\iota(\eta(O)) = v} \, q_0}{\iota(O) = u} \\
		&\geq q_0 / n.
	\end{align*}
\end{remark}

\begin{remark}
	Section \ref{subsec:mrp} only considered integer valued processes for the time component of Markov renewal processes. To apply the results of this section we will thus implicitely identify regeneration times with an integer-valued process. 
\end{remark}

From the uniform lower bound on escape probabilities, we can also deduce that regeneration times and levels have exponential tails conditional on the environment.

\begin{lemma}\label{lem:regeneration_tails}
	For any realization of $\cG$, for any starting vertex of the chain $\cX$, $L_1$ and $T_1$ have exponential tails (which can be bounded independently of $n$).
\end{lemma}

The following lemma is the analog of Lemma 3.6 in \cite{hermon2020universality} and is proved in a similar way.

\begin{lemma}\label{lem:markov_decomposition}
	\begin{itemize}
		\item The sequences $(Y_k, T_k)_{k \geq 1}$ and $(Y_k, L_k)_{k \geq 1}$ are Markov renewal processes. 
		\item The sequence $\left(Y_{k+1}, \cG_{\cX_{T_k}} \smallsetminus \cG_{\cX_{T_{k+1}}}, (\cX_{t})_{T_{k} \leq t < T_{k+1}} \right)_{k \geq 0}$ is a Markov chain whose transition probabilities only depend on the first coordinate $Y_k$, i.e. the law of this triplet at time $k+1$ conditional on time $k$ is only measurable with respect to $Y_k$.
		\item For all $k \geq 1$, conditional on $Y_k$, the pair $(\cG_{\cX_{T_{k}}}, (\cX_{t})_{t \geq T_k})$ has the law of $(\cG_O, \cX)$ under the probability $\bQ_{Y_k}$. 
	\end{itemize}
\end{lemma}

\begin{remark}
	Note that although $(Y,T)$ starts at time $0$, it is a Markov renewal process from time $1$, and the delay is thus $T_1$. However by the lemma, under the measure $\bQ_{u}$, for any $u \in V$, $T_1$ has distribution given by a transition probability and thus $(Y,T)$ becomes a Markov renewal process already from time $0$, with $0$ delay. We will often use this observation in the sequel. 
\end{remark}

\subsection{Mixing of the regeneration chain}\label{subsec:mixing_regeneration}

We now investigate the mixing properties of the Markov chain $Y$ underlying the regeneration process, which we call the regeneration chain. 
Lemma \ref{lem:mixing_regen} below will establish that $Y$ has mixing time of constant order, allowing us later to get moment bounds similar to the iid case. The lemma actually proves a stronger mixing property of both the chain $Y$ and the time process $T$ that will be used to derive an approximation of the invariant measure of $X$. 

\begin{lemma}\label{lem:mixing_regen}
	Let $(Q_t)_{t \geq 1}$ denote the transition kernels of the Markov renewal process $(Y,T)$ and $Q := \sum_{t \geq 1} Q_t$, after identification of $T$ with a process in $\bN$. 
	\begin{enumerate}[label=(\roman*)]
		\item For all $u, v \in V$,
		\begin{equation}\label{eq:doeblin}
			Q(u,v) = \Theta(1/n).
		\end{equation}
		As a consequence for all $\e \in (0,1)$, the chain $Y$ has mixing time $O_{\e}(1)$ as $n \rightarrow \infty$. 
		\item For all $u \in V$, 
		\begin{equation}\label{eq:lower_bound_alpha}
			\sum_{t \geq 1} \sum_{v \in V} Q_{t}(u, v) \wedge Q_{t+1}(u,v) = \Theta(1).
		\end{equation}
	\end{enumerate}
\end{lemma}

\begin{proof}[Proof of Lemma \ref{lem:mixing_regen}]
	Let $\e \in (0,1)$. The statement about the mixing time of $Y$ is easily deduced from \eqref{eq:doeblin}, as summing on $v$ yields the following Doeblin's condition for $Y$: there exists a constant $c > 0$ such that
	\begin{equation*}
		\TV {Q(u,\cdot) - Q(u', \cdot)} \leq 1 - c - o(1),
	\end{equation*}
	for all $u, u' \in V$. It is then easy to obtain from Doeblin's condition that $Y$ has mixing time $\tmix^{(Y)}(\e) \leq (c^{-1} + o(1)) \log \e^{-1}$.	

	Let us now prove \eqref{eq:doeblin}. Reversibility will here simplify the argument. Consider $u,v \in V$ with $v \in S$. By definition for all $t \geq 1$
	\begin{align*}
		Q_{t}(u,v) &= \bQ_{u} \sbra{Y_1 = v, T_1 = t} \\
		&= \bP \cond{T_1 = t, \iota(\cX_{t - 1/2}) = v}{\tau_{O} = \infty, \iota(O) = u, \cX_{1/2} = \eta(O)}.
	\end{align*}
	Probabilistic statements below will thus be made with respect to $\bQ_u$ unless stated otherwise. 
	
	Observe that if $V(u) \neq V(v)$, it is possible to realize $Y_1 = v$ with regeneration occurring at the first long-range edge crossed by the chain, which can be reached in half a step from $\eta(O)$ by Assumption \ref{hyp:cc3}. Let $w \in V$ such that $P(w,v) > 0$. From Remark \ref{rk:comparison_Qu}, $\bQ_{u} \sbra{\iota(\eta(O)) = w} = \Theta(1/n)$. Using that transition probabilities are bounded uniformly in $n$ by \ref{hyp:bdd_delta}, \ref{hyp:bdd_p}, we deduce that $\bQ_{u} \sbra{Y_1 = v} \geq \Theta(1/n)$. On the other hand, if $V(u) = V(v)$ the alternation between $V_1$ and $V_2$ forbids a regeneration at the first long-range edge. Thanks to reversibility the chain can backtrack and prevent regeneration at the first long-range. Thus using the above arguments we can lower bound by $\Theta(1/n)$ the probability that the chain goes from $\eta(O)$ to a vertex $y$ with $\iota(y) = v$ using one long-range edge, comes back to $\eta(O)$, which ensures no regeneration occurred on the long-range edge, and returns to $y$, all that within a bounded number of steps. The chain can then escape in $\cG_y$ with lower bounded probability by Proposition \ref{prop:escape_reversible} which will imply $Y_1 = v$. 

	\medskip

	The proof of \eqref{eq:lower_bound_alpha} is similar but requires taking time into account. Suppose $u \in V_1$ and let $v \in V_2$. Our argument is illustrated in Figure \ref{fig:mixing_argument}.
	\begin{figure}
		\centering
		\includegraphics[scale=.7]{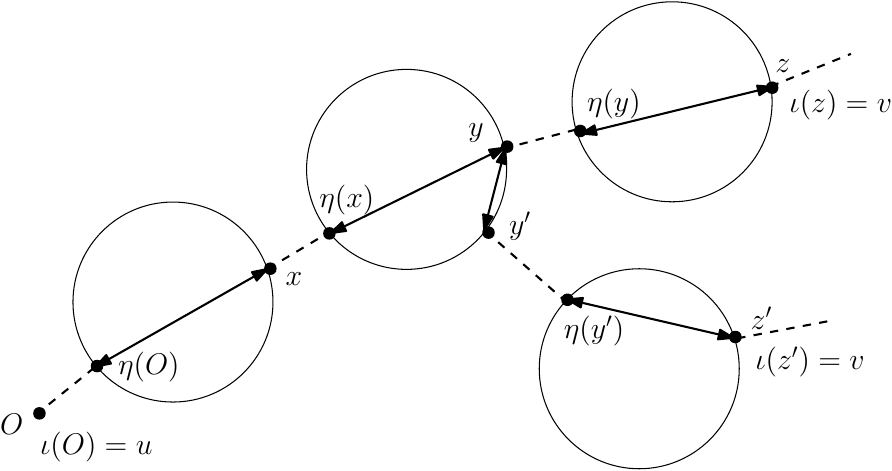}
		\caption{Argument of the proof of Lemma \ref{lem:mixing_regen}}
		\label{fig:mixing_argument}
	\end{figure}
	The idea is to use an intermediary component of $V_1$ to make the regeneration time shift by $1$. 
	Consider the set
	\begin{equation*}
		S_1 := \{ x \in V  \ | \exists y \neq x \in V:  P^{2}(x,y) > 0 \}.
	\end{equation*}
	From Assumption \ref{hyp:cc3} communicating classes of $V_1$ have size at least $3$. Consequently if $x \in V_1$, either $x \in S_1$ or $x$ is in a communicating class of size at least $3$ with $P^{2}(x,x) = 1$. However in that case there exist $y \neq z$ such that $P^{2}(y,z), P^{2}(z,y) > 0$, so $y,z \in S_1$. Thus $S_1$ has size at least $2n/3$ and probability $\Theta(1)$ under the uniform law on $V_1$ or $\bQ_w$, for any $w \in V_2$ by Remark \ref{rk:comparison_Qu}.
	
	To prove \eqref{eq:lower_bound_alpha}, we want to consider a realization of the quasi-tree as in Figure \ref{fig:mixing_argument}. From $O$, there exists a path made of a small-range and a long-range to $\eta(x)$. If $\iota(\eta(x)) \in S_1$ there exist $y,y' \in \cV$ such that $P(\iota(x), \iota(y)) > 0$ and $P(\iota(y), \iota(y')) > 0$. Figure \ref{fig:mixing_argument} shows a case where $y \neq y'$ but they may be equal. Then let $z$ be a small-range neighbour of $\eta(y)$. Thanks to reversibility, the chain started at time $1/2$ can go from $\eta(O)$ to $z$ in $2 + 1/2$ steps, from $z$ to $y'$ in $2$ steps and from $y'$ to $\eta(O)$ in $3$ steps. This will ensure none of the long-range edges $(x, \eta(x)), (y, \eta(y))$ are regeneration edges. Then from $\eta(O)$ the chain can go back to $z$ in $3$ more steps and escape to infinity with lower bounded probability, yielding $Y_1 = \iota(z)$ and $T_1 = 10 + 1/2$. 

	On the other hand, instead of $y$ the chain could have gone through $y'$ to some $z'$ reachable from $\eta(O)$ in just one additional step, while the loop around $\eta(O)$ which passes through $z'$ takes the same number of steps. Thus the previous argument can be applied similarly with just one time step difference to obtain $Y_1 = \iota(z'), T_1 = 11 + 1/2$. This holds for all values of $\iota(\eta(O)) \in V, \iota(\eta(x)) \in S$. Since $\iota(\eta(x)) \in V_1$ if $u \in V_1$, it is in $S$ with probability $\Theta(1)$ by what precedes while $\iota(z) = v$ and $\iota(z') = v$ each with probability $\Theta(1/n)$. Hence summing on the values of $\iota(\eta(O)), \iota(\eta(x))$, we obtain that 
	\begin{equation*}
		\sum_{t \geq 1} Q_t(u,v) \wedge Q_{t+1}(u,v) \geq \Theta(1/n).
	\end{equation*}
	The case where $u \in V_2$ is similar, with one less necessary long-range edge to cross to reach a component of $V_1$ so this case is simpler. The bound applies thus to all $u \in V$. Summing over $v \in V_2$ yields \eqref{eq:lower_bound_alpha}. 
\end{proof}

Let $\mu$ denote the stationary distribution of the Markov chain $(Y_k)_{k \geq 0}$ and 
\begin{equation*}
	\bQ_{\mu} := \sum_{u \in V} \mu(u) \bQ_{u}.
\end{equation*}
In the sequel $\bE_{\bQ_{\mu}}$ denotes the expectation with respect to $\mu$. Later, we use similar notations for variance, covariance, etc.

The mixing of the regeneration chain will be used conditional on some already revealed parts of the environment, which in turn requires conditioning by a neighbourhood of the root. From the previous lemma, we can prove the following. 

\begin{proposition}\label{prop:mixing_annealed_QT}
	Let $\nu$ be the law of $\iota(\eta(O))$ under $\bQ_{\mu}$. Given $L \geq 0$, let $\tT_1 = \tT_1(L)$ be the first regeneration time outside $\BLR(O,L)$, while for $k \geq 2$ let $\tT_{k}$ be the first regeneration time after $\tT_{k-1}$. 
	For all $\e \in (0,1)$ there exists a constant $C(\e)$ such that for all $L \geq 0$ and $t \in \bN + 1/2$, for all $x \in \BLR(O,L)$ with $d_{\LR}(O,x) = L$, if $t \geq C(\e)$ then
	\begin{equation*}
		\sum_{v \in V} \abs{ \bP_{x} \cond{\exists k \geq 0: \tT_k = t, \iota(\cX_t) = v}{\BLR(O,L)} - \frac{\nu(v)}{\bE_{\bQ_{\mu}} \sbra{T_1}}} \leq \e.
	\end{equation*}
	and 
	\begin{equation*}
		\sum_{v \in V} \abs{ \bP \cond{\exists k \geq 0: \tT_k = t, \iota(\cX_t) = v}{\cX_{1/2} = x, \BLR(O,L)} - \frac{\nu(v)}{\bE_{\bQ_{\mu}} \sbra{T_1}}} \leq \e.
	\end{equation*}
\end{proposition}

\begin{proof}
	Let $\e \in (0,1), t \geq 1$. For all $k \geq 1$ let $\tY_k := \iota(\cX_{\tT_{k} - 1/2})$. Apply Proposition \ref{prop:markov_renewal_thm} with the Markov renewal process of regeneration times considered here. Note that by the Markov property of $(Y,T)$, conditioning on $\BLR(O,L)$ yields the same transition kernels as $(Y,T$) and only affects the law of $(\tY_1, \tT_1$). Lemmas \ref{lem:mixing_regen} and \ref{lem:regeneration_tails} imply that the two quantities $\alpha$ and $\max_{u \in S} \bE_{u} \sbra{T_1}$ in this Proposition are bounded uniformly in $n$, as is the mixing time $\tmix^{(Y)}(\e)$ of $Y$ for all $\e \in (0,1)$. On the other hand, note that if $x$ is started precisely at the boundary of the ball $\BLR(O,L)$, the lower bound on escape probabilities provided by Proposition \ref{prop:escape_reversible} implies that $\tT_1$ has (quenched) exponential tail independent of $L$ and $n$. Thus for some $C_1(\e) >0$
	\begin{equation*}
		\bP_x \cond{\tT_1 \geq C_1(\e)}{\BLR(O,L)} \leq \e.
	\end{equation*}
	This gives the value of the quantity $K_{\nu}(\e)$ considered in Proposition \ref{prop:markov_renewal_thm}, which thus proves that there exists $C(\e)$ such that for $t \geq C(\e)$, 
	\begin{equation*}
		\sum_{u \in V} \abs{ \bP_{x} \cond{\exists k \geq 0: \tT_k = t-1/2, \tY_k = u}{\BLR(O,L)} - \frac{\mu(u)}{\bE_{\bQ_{\mu}} \sbra{T_1}}} \leq \e.
	\end{equation*}
	Then note that conditional on $\tT_k = t - 1/2, \tY_k = u$, $\iota(\cX_{t})$ is distributed as $\iota(\eta(O))$ under $\bQ_u$. Finally the arguments apply in the same way if the chain is started at time $1/2$ instead of $0$. 
\end{proof}

\section[Quasi-tree III: concentration of drift and entropy]{Analysis on the quasi-tree III: concentration of drift and entropy}\label{section:QT3}

In this section we finally establish "nice properties" for the chain $\cX$, proving in particular concentration of the drift and entropy. We only sketch some of the proofs, or do not give a proof at all, as once a uniform bound on escape probabilities is established the arguments are similar to those used in \cite{hermon2020universality}. 

\subsection{Typical paths in quasi-trees}

We can start with an analog of Lemma \ref{lem:typical_paths}. As usual, we make a slight abuse of notations to emphasize analogies.

\begin{lemma}\label{lem:typical_paths_QT}
	Let $\Gamma(R, L, M)$ denote the set of paths $\frp$ in $G^{\ast}$ such that $\frp$ does not deviate from a small-range distance more than $R$, backtrack over a long-range distance $L$ or contain a subpath of length $M$ without a regeneration edge. 
	There exists $C > 0$ such that for all $R, L, M \geq 0$, for all $s, t \geq 0$, 
	\begin{equation*}
		\bfP_O \sbra{\cX_{s} \cdots \cX_{s+t} \notin \Gamma(R,L,M)} \leq t e^{-C (R \wedge L \wedge M)}.
	\end{equation*}
\end{lemma}

\begin{proof}
	We only sketch the proof: as escape probabilities are everywhere lower bounded by a constant $q_0 > 0$ by Proposition \ref{prop:escape_reversible}, it should be clear that the probability to reach small-range distance $R$ or backtrack over long-range distance $L$ from a fixed starting state which is a center is exponentially small in $R$ or $L$ respectively, while Lemma \ref{lem:regeneration_tails} established that the regeneration times have exponential tails. The additional factor $t$ comes from union bound, as at most $t$ centers are visited by time $t$.  
\end{proof}

\subsection{Concentration of the drift}

\begin{proposition}\label{prop:drift}
	Let $\mathscr{d} := \frac{\bE_{\bQ_{\mu}} \sbra{L_1}}{\bE_{\bQ_{\mu}} \sbra{T_1}}$. Then for all $s \geq 0$, a.s. 
	\begin{equation}\label{eq:cv_drift}
		\frac{d_{\LR}(\cX_s, \cX_{s+t})}{t} \xrightarrow[t \rightarrow \infty]{} \mathscr{d}.
	\end{equation}
	Furthermore, there exists a constant $c_0 > 0$ for which the following holds. For all $\e > 0$ there exists a constant $C=C(\e)$ such that for all $s,t \geq 0$, for all values of $\iota(O), \iota(\eta(O))$
	\begin{equation}\label{eq:clt_drift}
		\bP_O \cond{ \abs{d_{\LR}(\cX_s, \cX_{s+t}) - \mathscr{d} t} > C \sqrt{t} }{\iota(O), \iota(\eta(O))} \leq \e + C \sqrt{s} e^{- c_0 t}.
	\end{equation}
\end{proposition}

\begin{proof}
		For notational simplicity we omit writing the conditioning on $\iota(O)$ and $\iota(\eta(O))$. As can be checked this conditioning does not affect the proof as the technical results that will be used hold even conditional on the long-range edge at the root.
		For all $t \geq 0$, let 
		\[
			N_t := \max \{k \geq 0 \ | \ T_k \leq t \}.
		\]
		Then 
		\begin{equation*}
			L_{N_t} \leq d_{\LR}(O, \cX_t) \leq L_{N_{t}} + T_{N_t +1} - T_{N_t}.
		\end{equation*}
		It is easy to prove that $(T_{N_{t} + 1} - T_{N_t}) / t \rightarrow 0$, hence the law of large numbers \eqref{eq:cv_drift} follows from Lemma \ref{lem:markov_decomposition} and Proposition \ref{prop:lln_mrp} which prove $N_t / t \xrightarrow[t \rightarrow \infty]{} 1/ \bE_{\mu}\sbra{T_1}$ and $L_k / k \xrightarrow[k \rightarrow \infty]{} \bE_{\mu} \sbra{L_1}$ a.s..

		We only sketch the proof of \eqref{eq:clt_drift}, which comes from fluctuation bounds for the processes $(L_t)$ and $(T_k)$: for all $\e > 0$ there exists $C > 0$ such that for all $t,k \geq 0$
		\begin{align}
			&\bP_O \sbra{\abs{N_t - \frac{t}{\bE_{\mu} \sbra{T_1}}} > C \sqrt{t}} \leq \e \label{eq:fluctations_Ns} \\
			&\bP_O \sbra{\abs{L_k -\bE_{\mu} \sbra{L_1} k} > C \sqrt{k}} \leq \e. \nonumber
		\end{align}
		These bounds are then easily combined, using also the monotonicity of regeneration levels, to obtain the result for $s = 0$. The above estimates come themselves from Bienaymé-Chebychev bounds for the processes $(L_k)$ and $(T_k)$. These require to show that $\Var(L_k) = O(k)$, which is the consequence of the variance bound for Markov chains \eqref{eq:variance_mrp} combined with Proposition \ref{prop:psg_mixing} and the fact that the regeneration chain $Y$ mixes in $O_{\e}(1)$ steps (Lemma \ref{lem:mixing_regen}). 
		
		Finally the case $s > 0$ is obtained from applying the same arguments to a shifted version of the process. For $s \geq 0$ fixed, consider
		\begin{equation*}
			T^{(s)}_{k} := T_{N_s + k} - s, \qquad L^{(s)}_{k} := L_{N_s + k} - d_{\LR}(O,\cX_s)
		\end{equation*}
		if $k \geq 1$ and $T^{(s)}_0 := 0, L^{(s)}_0 := 0$. These processes still satisfy the conclusions of Lemma \ref{lem:markov_decomposition} and have the same increments as the usual regeneration times. Thus the only thing to be careful about is the law of the first regeneration time that now depends on $s$. Using the concentration \eqref{eq:fluctations_Ns} for $N_s$, union bound and Lemma \ref{lem:regeneration_tails}, we have for all $m \geq 0$,
		\begin{align*}
			\bP \sbra{T^{(s)}_1 > m } &= \sum_{k \geq 0} \bP \sbra{N_s = k, T_{k+1} - s > m} \\
			&\leq \bP \sbra{\abs{N_s - s / \bE_{\bQ_{\mu}} \sbra{T_1}} > C \sqrt{s}} \\
            &\qquad + \bP \sbra{\exists k: \abs{k - s / \bE_{\bQ_{\mu}} \sbra{T_1}} \leq C \sqrt{s}, T_{k+1} - T_k > m} \\
			&\leq \e + 2 C \sqrt{s} e^{- c_0 m}. 
		\end{align*}
		for some constant $c_0 > 0$. 
		
\end{proof}

\subsection{Concentration of the entropy}

The concentration of the entropy in Proposition \ref{prop:nice_approx} is based on the convergence of the entropy for the loop-erased chain in the quasi-tree, that is the convergence of $- \log \bfP \cond{\xi'_k = \xi_k}{\xi} / k$ towards a deterministic quantity, the entropic rate of the chain. Such convergence is well-known in the context of groups or random walks on Galton-Watson trees, see \cite{kaimanovich1983random, lyons1995ergodic}. We will however not prove this result but establish concentration directly for a notion of weights similar to those of \eqref{eq:def_weights_G}. Of course, these are designed to mimick the law of the loop-erased chain, so the argument is similar. In fact the first step is to prove the convergence and concentration of the loop-erased chain when restricted to regeneration steps. 

\begin{lemma}\label{lem:entropy_regen}
	Let $\xi'$ be an independent copy of the loop-erased chain $\xi$. 
	There exists $h' = \Theta(1)$ such that a.s.
	\begin{equation*}
		\lim_{k \rightarrow \infty} \frac{- \log \bfP \cond{ \xi'_{L_k} = \xi_{L_k} }{\cX}}{k} = h'.
	\end{equation*}
	Furthermore, for all $\e > 0$, there exists $C(\e) >0$ such that for all $k, l \geq 1$, for all values of $\iota(O), \iota(\eta(O))$
	\begin{equation}\label{eq:clt_regen}
		\begin{gathered}
			\bP \cond{\abs{- \log \bfP \cond{ \xi'_{L_k} = \xi_{L_k} }{\cX} - h' k} > C(\e) \sqrt{k}}{\iota(O), \iota(\eta(O))} \leq \e, \\
			\bP \cond{\abs{- \log \bfP \cond{ \xi'_{L_{k+l}} = \xi_{L_{k+l}} }{\cX, \xi'_{L_k} = \xi_{L_k}} - h' l} > C(\e) \sqrt{l}}{\iota(O), \iota(\eta(O))} \leq \e.
		\end{gathered}
	\end{equation}
\end{lemma}

We do not give a proof, but refer to that of Lemma 3.14 in \cite{hermon2020universality} as it uses similar arguments. The only difference lies in the additional Markovian property of the regeneration, which is dealt with as for the drift using the variance bound \eqref{eq:variance_mrp}.

To relate the previous concentration with the weights \eqref{eq:def_weights_G}, we define similar weights in the quasi-tree. Let $\tau_l$ denote here the first time $t$ such that $d_{\LR}(\cX_0, \cX_t) = l$. For all long-range edges $e \in \cG$, write $\cG_e$ for the subquasi-tree at any endpoint of $e$ (they give the same quasi-tree). 
Given $R, L \geq 0$, $x \in \cV$ and a long-range edge $e$ at long-range distance $0$ from $x$ 
\begin{equation}\label{eq:def_weights_quasiT}
	\begin{split}
	w_{x,R,L}(e) &:= \bfP_{x} \sbra{\cX_{\tau_L} \in \cG_{e}, \tau_{L} < \TSR} \\
	w_{R,L}(e \ | \ x) &:= \bfP \cond{\cX_{\tau_L} \in \cG_{e}, \tau_{L} < \TSR}{\cX_{1/2} = x, \tau_{L} < \tau_{\eta(x)}}
	\end{split}
\end{equation}
Then if $e=(e_i)_{i=1}^{k}$ is a long-range non-backtracking path starting from $\BLR(x,0)$, set
\begin{equation*}
	w_{x,R,L}(e) := w_{x,R,L}(e_1) \prod_{i=2}^{k} w_{R,L}(e_i \ | e_{i-1}^{+})
\end{equation*}
where the product is taken equal to $1$ if empty and the edges are oriented away from $x$. 

Consider the measure 
	\begin{equation*}
		\bfQ_{u,g} := \bP \cond{\cdot}{\cX_{1/2} = \eta(O), \iota(O) = u, \tau_{O} = \infty, \cG_O = g}
	\end{equation*}
where $u \in V$ and $g$ is a possible realization of the subquasi-tree $\cG_O$. We use a notation which may be reminiscent of \eqref{eq:bfQu_G} as these two measures are very similar, although note that here $u$ is not the type of the starting state of the chain but its long-range neighbour. There should be no risk of confusion as the measure of \eqref{eq:bfQu_G} will not be used until Section \ref{section:nice}. It is easily seen that
	\begin{equation}\label{eq:harmonic_measure_product}
	\bfP_{x} \sbra{\xi_1 = e_1, \ldots, \xi_k = e_k} = \bfP_{x} \sbra{\xi_1 = e_1} \prod_{i=2}^{k} \bfQ_{\iota(e_{i-1}^{+}), \cG_{e_{i-1}}} \sbra{\xi_{1} = e_i}.
\end{equation} 
The following Lemma establishes thus a bound on individual weights. 

\begin{lemma}\label{lem:weight_approx}
	There exist constants $C_0,C_1,C_2 > 0$ such that for all $R, L > 0$ the following holds: 
	\begin{enumerate}[label=(\roman*)]
		\item for all long-range edges $e$ such that $d_{\SR}(O,e^{-}) < R$,
			\begin{equation*}
			\abs{\log w_{O,R,L}(e) - \log \bfP_O \sbra{\xi_1 = e}} \leq C_0 \, e^{- C_1 L + C_2 R}, \label{eq:weight_approx_O} 
			\end{equation*}
		\item for all $x \in \cV$, for all long-range edges $e$ of $\cG_x$ such that $d_{\SR}(x, e^{-}) < R$
			\begin{equation*}
				\abs{\log w_{R,L}(e \ | \ x) - \log \bfQ_{\iota(e^{-}), \cG_e} \sbra{\xi_1 = e}} \leq C_0 \, e^{- C_1 L + C_2 R}. \label{eq:weight_approx_conditional}
			\end{equation*}
	\end{enumerate}
\end{lemma}

\begin{proof} 
	We only prove the first bound in detail.  
	Notice that to have one of the two events $e \in \xi$ or $\cX_{\tau_L} \in \cG_e$ realized exclusively, the chain needs to backtrack from level $L$ to level $0$. Thus by Lemma \ref{lem:typical_paths_QT} 
	\begin{equation*}
		\abs{w_{O,R,L}(e) - \bfP_O \sbra{\xi_1 = e}} \leq \bfP_O \sbra{\TSR \wedge \TNB < \infty} \leq e^{- C (R \wedge L)}
	\end{equation*}
	for some constant $C > 0$. On the other hand, recall that all transition probabilities of the chains $X, \cX$ are lower bounded by some constant $\delta > 0$. Hence if $e$ has one endpoint in $\BSR(O, R)$, then
	\begin{equation*}
		w_{O,R,L}(e) \wedge \bfP_O \sbra{\xi_1 = e} \geq \delta^{R} q_0
	\end{equation*}
	as the right-hand side is a lower bound on the probability to go from $O$ to $e$ and then escape to infinity, which forces both $e \in \xi$ and $\cX_{\tau_{L}} \in \cG_e$. Using the inequality $\abs{\log x - \log y} \leq \abs{x-y} / (x \wedge y)$, we deduce
	\begin{equation*}
		\abs{\log w_{O,R,L}(e) - \log \bfP_O \sbra{\xi_1 = e}} \leq C_0 \, e^{-C_1 L + C_2 R}
	\end{equation*}
	for some $C_0, C_1, C_2 > 0$. 
\end{proof}

Our final entropic concentration result in the quasi-tree is the following. Note that for $R,L = O(\log \log n)$ and $t = \Theta(\log n)$ the right-hand side of \eqref{eq:clt_weights} is $O(\sqrt{t})$ if the implicit constant of $L$ is large enough.

\begin{lemma}\label{lem:concentration_weights_QT}
	There exists $h = \Theta(1)$ and constants $C, C_1, C_2, C_3 > 0$ such that the following holds. For all $\e > 0$ there exist $C_0 = C_0(\e), C_h=C_h(\e)$ such that for all $s,t \geq 0$, for all $R,L > 0$, with probability at least $ 1 - \e - 2 (s+t) e^{-C (R \wedge L)}$ conditional on $\iota(O), \iota(\eta(O))$, 
	\begin{equation}\label{eq:clt_weights}
		\abs{- \log w_{\cX_s, R, L}(\xi(\cX_s \cdots \cX_{s+t})) - h t} \leq C_{h} \sqrt{t} + C_0 \, t \, e^{-C_1 L + C_2 R} + C_3 R L. 
	\end{equation}
\end{lemma}

\begin{proof}
	Fix $R,L$ for the rest of the proof. For notational simplicity, we drop subscripts $R,L$ from the weights and omit writing the conditioning by $\iota(O), \iota(\eta(O))$ but the probabilistic statements below should be interpreted conditional on these. Recall $\cG^{(R)}$ is the quasi-tree truncated at the $R$-boundary of small-range components. Note that weights are positive only for edges in $\cG^{(R)}$.
	
	Given $t \geq 0$, let 
	\begin{equation*}
		N_t := \max \{k \geq 0 \ | \ T_k \leq t \}.
	\end{equation*}
	We first argue that there exists a constant $C_3 > 0$ such that for all $s,t \geq 0$ if $\cX_{0} \cdots \cX_{s+t}$ is included in $\cG^{(R)}$ then
	\begin{multline}\label{eq:weight_truncation}
		\abs{\log w_{\cX_s}(\xi(\cX_s \cdots \cX_{s+t})) - \log w (\xi_{L_{N_{s}+2}} \cdots \xi_{L_{N_{s+t}}} \ | \ \xi_{L_{N_s}+1})} \leq C_3 R \left( T_{N_{s} +2} - T_{N_s} \right. \\ \left. + T_{N_{s+t} +1} - T_{N_{s+t}}\right).
	\end{multline}
	Note first that $\xi_{L_{N_{s}+1}} \cdots \xi_{L_{N_{s+t}}}$ is necessarily part of the path $\xi(\cX_s \cdots \cX_{s+t})$. Since weights are below one, we can easily lower bound  
	\begin{equation*}
		- \log w_{\cX_s}(\xi(\cX_s \cdots \cX_{s+t})) \geq  - \log w (\xi_{L_{N_{s}+2}} \cdots \xi_{L_{N_{s+t}}} \ | \ \xi_{L_{N_{s}+1}}).
	\end{equation*}
	On the other hand the path $\xi(\cX_s \cdots \cX_{s+t})$ contains at most $T_{N_{s} + 2} - s \leq T_{N_{s} +2} - T_{N_s}$ edges until it reaches $\xi_{L_{N_{s}+2}}$ and similarly it contains at most $T_{N_{s+t} +1} - T_{N_{s+t}}$ after it leaves $\xi_{L_{N_{s+t}}}$. Then we bound the weights of these edges. If $x \in \cV$ and $e$ is a long-range edge such that $d_{\SR}(x, e^{-}) < R$, then $\delta^{R}$ lower bounds the probability the chain goes from $x$ to $e$ and leaves by this edge. As the transition and the quenched escape probabilities are all uniformly bounded away from $0$, there exists $C_3 > 0$ such that $- \log w_{x}(e) \leq C_3 R$ for all $x,e$ in $\cG^{(R)}$, for any realization of the quasi-tree. We deduce that
	\begin{multline*}
		- \log w_{\cX_s}(\xi(\cX_s \cdots \cX_{s+t})) + \log w_{\cX_s} (\xi_{L_{N_{s+2}}} \cdots \xi_{L_{N_{s+t}}}) \leq C_3 R \left( T_{N_{s} +2} - T_{N_s} \right. \\ \left. + T_{N_{s+t} +1} - T_{N_{s+t}}\right).
	\end{multline*}
	
	Next fix $\e \in (0,1)$ and let $s,t \geq 0$. Suppose $C_0, C_1, C_2 > 0$ are such that the bounds of Lemma \ref{lem:weight_approx} hold. Consider the following events: 
	\begin{align*}
		&E_1 := \{ \TSR > s + t \} \\
		&E_2 := \{ \abs{L_{N_{s+t}} - L_{N_s} - \mathscr{d} t } \leq C_{\LR} \sqrt{t} \}
	\end{align*}
	with $C_{\LR} = C_{\LR}(\e) > 0$ to determine, and write $E := E_1 \cap E_2$. By Lemma \ref{lem:typical_paths_QT} there exists $C_4 > 0$ such that $\bfP_O \sbra{E_1^{c}} \leq (s+t) e^{- C_4 R}$, whereas from the fluctuation bounds \eqref{eq:fluctations_Ns}, $C_{\LR}(\e)$ can be taken so that $\bP_O \sbra{E_2^c} \leq \e$. Hence $\bP_O \sbra{E^c} \leq \e + (s+t) e^{-C_4 R }$.

	Then let $l_{+} := \lceil \mathscr{d} t + C_{\LR} \sqrt{t} \rceil$. On the event $E$, using \eqref{eq:harmonic_measure_product} and Lemma \ref{lem:weight_approx}
	\begin{multline*}
		\abs{\log w (\xi_{L_{N_{s}+2}} \cdots \xi_{L_{N_{s+t}}} \ | \ \xi_{L_{N_{s} +1}}) - \log \bfP \cond{\xi_{L_{N_{s}+2}} \cdots \xi_{L_{N_{s+t}}} \in \xi'}{\cX, \xi'_{L_{N_{s}+1}} = \xi_{L_{N_{s}+1}}}} \\ \leq C_0 l_{+} e^{-C_1 L + C_2 R}.
	\end{multline*}	
	Letting $h'$ be the constant of Lemma \ref{lem:entropy_regen}, we claim then that for $h = h' / \bE_{\bQ_{\mu}} \sbra{T_1}$ and some constant $C_h = C_h(\e)$
	\begin{equation*}
		\abs{- \log \bfP \cond{\xi_{L_{N_{s}+2}} \cdots \xi_{L_{N_{s+t}}} \in \xi'}{\cX, \xi'_{L_{N_{s}+1}} = \xi_{L_{N_{s}+1}}} - h t} \leq C_{h} \sqrt{t}
	\end{equation*}
	with probability at least $\e$. This can be done using the same arguments as to prove Proposition \ref{prop:drift}, combining the fluctations for $N_{s+t}$ \eqref{eq:fluctations_Ns} with Lemma \ref{lem:entropy_regen}.

	Finally, we are left with bounding the right hand side of \eqref{eq:weight_truncation}. Using again the fluctations for $N_{s+t}$ \eqref{eq:fluctations_Ns} with the exponential tail of regeneration times, there exists $C_5 > 0$ such that
	\begin{equation*}
		\bP \sbra{ T_{N_{s} +1} - T_{N_s} + T_{N_{s+t} +1} - T_{N_{s+t}} > L} \leq \e + \sqrt{s+t} e^{- C_5 L}.
	\end{equation*}
	All in all, we have thus proved that for some $C > 0$, with probability at least $1 - 2 \e - 2 (s+t) e^{-C (R \wedge L)}$,
	\begin{equation*}
		\abs{\log w_{\cX_s}(\xi(\cX_s \cdots \cX_{s+t})) - ht } \leq C_h \sqrt{t} + C_0 \, l_{+} e^{-C_1 L + C_2 R} + C_3 R L.
	\end{equation*}
	Since $l_+ = O_{\e}(\sqrt{t})$ this proves the result.
\end{proof}


\section{First steps towards nice trajectories}\label{section:nice_first}

In this section we come back to the finite setting and prove Lemmas \ref{lem:QTlike}, \ref{lem:typical_paths} and Proposition \ref{prop:concentration_G}. These will be used in the next section to establish that if one looks the chain at time $s+t$ with $s = C \log \log n$ and $t = \log n / h + C_0 \sqrt{\log n}$, the last $t$ steps form a nice path with high probability, for large enough $C$. At this stage if one gathered all the observations made so far about the trajectories of $\cX$ on a quasi-tree we could use the coupling with the finite setting of Section \ref{subsec:coupling} to prove that for a fixed and hence typical starting state $x$, $X_t$ is likely to follow a nice trajectory. To obtain results for arbitrary starting vertices, we need to strengthen the bounds on the annealed probability of bad events to $o(1/n)$, as explained in Section \ref{subsec:first_moment}.

\subsection{From $o(1)$ to $o(1/n)$: bootstrapping annealed bounds with parallel chains}

The basic strategy is again to relate the quenched and annealed laws by means of Markov's inequality. By union bound, to show that a trajectorial event holds with high probability uniformly over the starting point under the quenched law, it suffices to show the complement event has annealed probability $o(1/n)$. Arguments used so far only established error in $o(1)$. Following ideas from \cite{bordenave2018random,bordenave2019cutoff}, one strategy to improve these error bounds to $o(1/n)$ consists in using higher order moments in Markov's inequality, which leads to an argument of "parallelizing chains" on the same environment. Namely for all $\e > 0$ and $k \geq 1$, for any trajectorial event $A$
    \begin{align}
        \bP \sbra{\max_{x \in V} \bfP_{x} \sbra{A} > \e} &\leq \frac{1}{\e^{k}} \bE \sbra{\sum_{x \in V} \bfP_x \sbra{A}^{k} } \nonumber \\
        &\leq \frac{n}{\e^{k}} \max_{x \in V} \bP_{x} \sbra{ \bigcap_{i=1}^{k} \{ X^{(i)} \in A \} }, \label{eq:higher_markov}
    \end{align}
where $X^{(1)}, \ldots,  X^{(k)}$ are $k$ versions of the chain $X$ generated independently conditional on the same environment. These $k$ trajectories can be generated sequentially with the environment, sampling the environment only when exploring parts of the environment not already visited by the previous chains. These annealed trajectories can then be studied thanks to a coupling with a quasi-tree, or rather a generalization of it, allowed to contain a cycle. 

\subsection{Quasi-trees with a cycle}

The first error bound that needs to be strengthened is in the coupling with a quasi-tree (Lemma \ref{lem:coupling_quasiT}). From the comparison with a binomial variable, we see that one long-range cycle must be allowed to obtain a $o(1/n)$ error bound. We thus formalize a new model of quasi-trees model allowed to contain one long-range cycle. 

\begin{figure}
	\centering
	\includegraphics[scale=0.5]{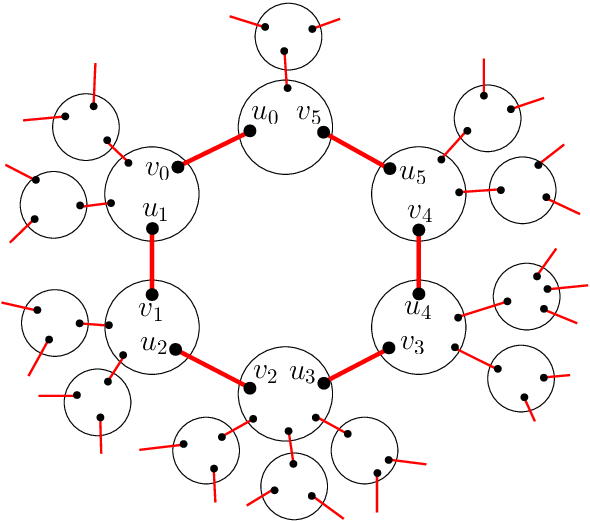}
	\caption{A quasi-tree with a cycle of length $6$.}
	\label{fig:QT_cycle}
\end{figure}

Let $(u_i)_{i=0}^{l-1}, (v_i)_{i=0}^{l-1}, l \geq 1$ be two sequences of vertices in $V$ such that $u_{i+1} \in \BSR(v_i, R)$ for all $i \in [0,l-1]$ and $\BSR(v_{l-1}, R) \cap \BSR(u_0, R) \neq \emptyset$. This sequence will form the long-range cycle, as we set $\eta(u_i) := v_{i}$ for all $i \in [0,l-1]$. At this stage we thus have $l$ small-range components joined by long-range edges: for every $i \in [0,l-1]$, $v_i$ and $u_{(i+1) \mod l}$ are the two vertices of component $i$ on this path, and we require them to be at small-range distance at most $R$ from each other. The last property makes this path a "cycle" by having the small-range balls around $v_l$ and $u_0$ intersect. Then for every $i$ and vertex $z \in \BSR(u_i, R) \smallsetminus \{v_i, u_{(i+1) \mod l} \}$, we add a long-range edge $(z, \eta(z))$ and grow a one-sided quasi-tree rooted at $\eta(z)$. The graph $\cG = (\cV, \cE)$ thus obtained will be called a quasi-tree with a cycle $((u_i,v_i))_{i =0}^{l-1}$. An example is illustrated in Figure \ref{fig:QT_cycle}, where only the long-range edges are represented in red. As a usual quasi-tree, it is given by maps $\iota, \eta$ such that $\iota$ identifies vertices of $\cV$ with vertices in $V$, while $\eta: \cV \rightarrow \cV$ is an involution. The values of $\eta$ along the small-range components of the cycle have been defined above, while for the rest of $\cV$ we extend $\eta$ by using the involution on each quasi-tree. By growing the subquasi-trees at random as explained Section \ref{subsec:coupling}, we obtain a construction which is random except for the cycle. A root $O$ of $\cG$ can be chosen, but it is not required to be on the cycle. 

Consider $\cG$ to be a random realization of a quasi-tree with a cycle as described above. The definition of the Markov chain $\cX$ \eqref{eq:QT_MC} extends directly to this setting. Conditional on the cycle, the quasi-trees that are added to it are generated according to the model studied in Sections \ref{section:QT1} - \ref{section:QT3}, therefore the analysis of the chain $\cX$ directly extends to that case, conditional on the cycle. The chain ultimately leaves the cycle, after which it stays in a genuine quasi-tree. Given a vertex $x$ let $\qesc(x)$ denote the probability to escape to infinity using the long-range edge $(x, \eta(x))$. If $x$ is not on the cycle, this escape probability is the same as those considered in Section \ref{section:QT1} and can be lower bounded using Proposition \ref{prop:escape_reversible}. If $x$ is on the cycle, it is at bounded distance from a vertex that is not thanks to assumption \ref{hyp:cc3}, so its escape probability can be lower bounded similarly. All in all, we deduce that all escape probabilities can be lower bounded by a constant $q_0$ regardless of the realization of $\cG$. 

\subsection{Coupling of parallel chains}

Let $k \geq 1$ and $X^{(1)}, \ldots, X^{(k)}$ be realizations of the chain $X$, generated independently conditional on the same environment. Let $G^{\ast}(k,t)$ be the environment explored by the trajectories of these chains up to time $t$ along with their $L$-long-range neighbourhoods. We can reapply the stochastic comparison used in the proof of Lemma \ref{lem:coupling_quasiT}, but using now that 
\begin{equation*}
     \bP \sbra{Z \geq 2} \leq \frac{m^{4} \Delta^{2R}}{n^{2}}
\end{equation*}
if $Z$ is a binomial $\mathrm{Bin}(m, m \Delta^{R} / n)$, we deduce that
\begin{equation}\label{eq:no_2_cycles}
    \bP \sbra{G^{\ast}(k,t) \text{ contains $2$ long-range cycles or more}} \leq (kt)^{4}/n^{7/4} = o(\e^{k} / n)
\end{equation}
if for instance $t =o(n^{1/16})$ and $k = \lfloor (\log n) / (3 \log (\e^{-1}) ) \rfloor$. 

This motivates a coupling with chains $\cX^{(1)}, \ldots, \cX^{(k)}$ lying on infinite quasi-trees, potentially containing a cycle, as follows. We start by coupling $X^{(1)}$, $\cX^{(1)}$ exactly as in Section \ref{subsec:coupling}. Then for all $i \geq 2$, conditional on $G^{\ast}(i-1,t)$, the transitions of $\cX^{(i)}$ inside $G^{\ast}(i-1,t)$ are taken to be the same, while outside we use again the coupling of Section \ref{subsec:coupling} to sample new long-range edges. Let $\cG(i,t)$ denote the union of $G^{\ast}(i-1,t)$ and the environment explored by the trajectory of $\cX^{(i)}$ up to time $t$ along with its $L$-long-range neighbourhood.
Repeating the arguments of Lemma \ref{lem:coupling_quasiT}, we deduce that conditional on $G^{\ast}(i-1,t)$, the chains $X^{(i)}$ and $\cX^{(i)}$ are coupled so that 
\begin{equation}\label{eq:coupling_uniform} 
    \bP \cond{\begin{array}{c} \text{$G^{\ast}(i,t)$ and $\cG(i,t)$ are isomorphic} \\ \forall s \leq t: X^{(i)}_s = \iota(\cX^{(i)}_s) \end{array}}{G^{\ast}(i-1,t)} = 1- o(1)
\end{equation} 
with $k, t$ as above.

By construction, the chain $\cX^{(i)}$ evolves on an environment made of random infinite quasi-trees grown on the long-range edges of $G^{\ast}(i-1,t)$ which have not been matched already. By what precedes, for the above choice of $t$, $k$ and $i \leq k$ this environment is with probability $1- o(\e^{k} / n)$ a quasi-tree with possibly one cycle. 

\subsection{Quasi-tree-like trajectories}

\begin{proof}[Proof of Lemma \ref{lem:QTlike}]
    Start with the first point. By Remark \ref{rk:uniform_shifted} we need only to establish the result for $s = C \log \log n$, $C > 0$ being a constant to determine. Let $x \in V, L = C_L \log \log n$ for a fixed constant $C_L > 0$ and $t \ll n^{1/16}$. We use the strategy of "parallelized chains". Consider $k \geq 1$ independent chains $X^{(1)}, \ldots, X^{(k)}$ and for all $i \in [k]$ consider the event 
    \begin{equation*}
        A_i := \left\{ \text{$\bigcup_{s' \in [s,s+t]} \BLR \left( X^{(i)}_{s'},L \right)$ is not quasi-tree-like} \right\}.
    \end{equation*}
    From \eqref{eq:higher_markov} the result will be proved if we establish that $\bP_x \sbra{\bigcap_{i=1}^{k} A_i} = o(\e^{k} / n)$ for any $\e > 0$ and $k = \lfloor (\log n) / (3 \log (\e^{-1}) ) \rfloor$. As
    \begin{equation*}
        \bP_x \sbra{\bigcap_{i=1}^{k} A_i} = \bE_x \sbra{ \II_{\bigcap_{i=1}^{k-1} A_i} \bP_x \cond{A_k}{G^{\ast}(k-1,t)}}
    \end{equation*}  
    it suffices to prove that $\bP_x \sbra{A_1} = o(1)$ and $\bP_x \cond{A_i}{G^{\ast}(i-1,t)} = o(1)$ for all $i \leq k$.
    By \eqref{eq:no_2_cycles}, 
    \begin{equation*}
        \bP_x \sbra{\bigcap_{i=1}^{k} A_i} = \bP_x \sbra{\bigcap_{i=1}^{k} A_i \cap \{ \text{$G^{\ast}(k,t)$ has at most $1$ long-range cycle} \}} + o(\e^{k} / n)
    \end{equation*}
    thus in the following we can suppose that every $G^{\ast}(i-1,t)$ has at most one long-range cycle.
    Consider the above coupling with $k$ chains $\cX^{(1)}, \ldots, \cX^{(k)}$ and let 
    \begin{equation*}
        B_i := \left\{ \text{$\bigcup_{s' \in [s,s+t]} \BLR \left( \cX^{(i)}_{s'},L \right)$ is not quasi-tree-like} \right\}.
    \end{equation*}
    By Lemma \ref{lem:coupling_quasiT}, $X^{(1)}, \cX^{(1)}$ are coupled so that $\bP \sbra{A_1} = o(1)$. Then, by \eqref{eq:coupling_uniform}, for all $i \leq k$, 
    \begin{equation*}
        \bP_x \cond{A_i}{G^{\ast}(i-1,t)} = \bP_x \cond{B_i}{G^{\ast}(i-1,t)} + o(1).
    \end{equation*}
    Now recall that we have ruled out the possibility of $G^{\ast}(i-1,t)$ having two long-range cycles or more.
    If $G^{\ast}(i-1,t)$ has no cycle,  $\cX^{(i)}$ evolves on a quasi-tree with no cycle and the right hand side is trivially $o(1)$. Finally if $G^{\ast}(i-1,t)$ has a cycle, then $\cX^{(i)}$ evolves on an infinite quasi-tree with one cycle. We can now reason conditional on a realization of this quasi-tree and show that $\cX^{(i)}$ moves away from the cycle. From the fact that escape probabilities are everywhere bounded uniformly away from $0$ and that every vertex is at bounded distance from a genuine subquasi-tree containing no cycle, it is easy to deduce that by time $s$, $\cX^{(i)}$ has already made at least $C s$ steps of the loop-erased chain away from the cycle with probability $1-e^{- C' s}$, for some constants $C, C' > 0$. Since $s = \Theta(L)$, choosing the implicit constant defining $s$ large enough thus ensures the chain remains at distance at least $L$ from the cycle.
    
    
    Let us prove the second point. Consider $t = C_0 \log n$ for some arbitrarily small constant $C_0 >0$ and let $l = \lfloor (\log n) / (10 \log \Delta) \rfloor$. Using the same coupling as above and a comparison with a binomial random variable we deduce that with high probability the neighbourhood $B_{\scP}(x, 2l)$ around any vertex $x$ in $G^{\ast}$ contains at most one long-range cycle. It is thus a potential realization of a $2l$ neighbourhood of a point $O$ in a quasi-tree with a cycle $\cG$. The two chains $X$ and $\cX$ can then be coupled until they exit this neighbourhood, which cannot occur before time $t$ if for instance $C_0 < (10 \log \Delta)^{-1} / 2$. Considering any such realization of $\cG$, all escape probabilities are lower bounded by some constant $q_0 > 0$. Using Chernoff's bound we deduce that for some constant $\alpha=\alpha(q_0,C_0) > 0$, the quenched probability that $\cX_t$, and consequently $X_t$, is at distance less than $\alpha l$ from the cycle is exponentially small in $l$.
\end{proof}

\begin{proof}[Proof of Lemma \ref{lem:typical_paths}]
    The argument follows the same reasoning as in the previous proof so we omit some of the details. Fix $\e > 0$ and a starting state $u \in V$. Let $s = \log \log n$, $t = \Theta(\log n)$. By a slight abuse of notation, let $\Gamma(R,L,M)$ denote also the set of paths in a quasi-tree  $\cG$ possibly containing a cycle that satisfy the same requirements, regeneration being here understood as having infinite horizon. Clearly regeneration times for $\cX$ with infinite horizon are also regeneration times with horizon $L$. Thus under the coupling described above, a sufficient condition to have $X_s \cdots X_{s+t} \in \Gamma(R,L,M)$ is that the coupling did not fail by time $s+t$ and $\cX_{s} \cdots \cX_{s+t} \in \Gamma(R,L,M)$. Let $k := \lfloor (\log n) / (3 \log (\e^{-1})) \rfloor$ and consider the coupling of $k$ independent versions $X^{(1)}, \ldots, X^{(k)}$ of $X$ conditional on the same environment with $\cX^{(1)}, \cdots, \cX^{(k)}$. For all $i \in [k]$ consider the event 
    \begin{equation*}
        B_i := \{(\cX^{(i)}_s \cdots \cX^{(i)}_{s+t}) \notin \Gamma(R,L,M) \}.
    \end{equation*}
    As in the previous proof, the choice of $k,s,t$ makes the coupling unlikely to fail by time $s+t$ by \eqref{eq:coupling_uniform} and it suffices to establish that $\bP_O \cond{B_1}{\iota(O) = u} = o(1)$ as well as $\bP_O \cond{B_i}{G^{\ast}(i-1, t)} = o(1)$. We can actually establish quenched statements valid for any realization of the quasi-trees, even if those contain a cycle.
       
    For large enough constants $C_R, C_L$, paths that are not in $\Gamma(R,L,M)$ because they hit the boundary of a small-range ball or because of backtracking have quenched probability $o(1)$ by Lemma \ref{lem:typical_paths_QT}. 
    On the other hand, from the quenched exponential tails of regeneration times (Lemma \ref{lem:regeneration_tails}), there exists $C_M > 0$ large enough such that if $M \geq C_M \log \log n$, then for all $k \geq 1$, $\bfP_O \sbra{T_{k+1} - T_{k} \geq M} = o(1/\log n) = o((s+t)^{-1})$. Thus by union bound
	\begin{align*}
		&\bfP_O \sbra{\exists t' \leq s+t: [t', t' + M] \cap \{T_{k}, k \geq 1\} = \emptyset} \\
		&\qquad \leq \bfP_O \sbra{\exists k \leq s+t: T_{k+1} - T_{k} \geq M} \\
		&\qquad \leq (s+t) \max_{k \leq s+t} \bfP_O \sbra{T_{k+1} - T_{k} \geq M},
	\end{align*} 
	which is $o(1)$ by what precedes.
    \end{proof}

\subsection{Concentration of drift and entropy from an arbitrary vertex}

The final step in this section is to prove concentration of the drift and entropy. 

\begin{proof}[Proof of Proposition \ref{prop:concentration_G}]
    Let $L := C_L \log \log n$, $R = C_R \log \log n, s = C_0 \log \log n$ for some $C_R, C_L, C_0 > 0$ to be determined, and $t = O(\log n)$, $t \rightarrow \infty$. Fix $\e > 0$ and a starting state $u \in V$. Let $k := \lfloor (\log n) / (3 \log (\e^{-1})) \rfloor$ and consider $k$ independent versions $X^{(1)}, \cdots, X^{(k)}$ of the chain $X$ conditional on the same environment, coupled with $\cX^{(1)}, \cdots, \cX^{(k)}$ evolving in infinite environments, as described above. For constants $C_{\LR}=C_{\LR}(\e), C_h=C_h(\e)$ to be determined, let
    \begin{equation*}
        \Gamma_{\Ent} := \left\{ \frp \ \left| \ \begin{array}{c} \abs{\abs{\xi(\frp)} - \mathscr{d} \abs{\frp} } \leq C_{\LR}\sqrt{\abs{\frp}} \\
            \abs{- \log w_{\frp_0,R,L} (\xi(\frp)) - h \abs{\frp} } \leq C_h \sqrt{\abs{\frp}} \end{array} \right. \right\}
    \end{equation*}
    where $\frp_0$ denotes the starting vertex of the chain. This definition can be applied to both paths in either $G^{\ast}$ or in a quasi-tree $\cG$ possibly containing a cycle, using the two notions of weights in $G^{\ast}$ \eqref{eq:def_weights_G} and $\cG$ \eqref{eq:def_weights_quasiT}. For all $i \in [k]$ let 
    \begin{equation*}
        A_i =: \{ X^{(i)}_{s} \cdots X^{(i)}_{s+t} \notin \Gamma_{\Ent} \}
    \end{equation*}
    and
    \begin{equation*}
        B_i := \{ \cX^{(i)}_{s} \cdots \cX^{(i)}_{s+t} \notin \Gamma_{\Ent} \}.
    \end{equation*}
    Thanks to \eqref{eq:higher_markov} our goal is to establish that $\bP_u \sbra{\bigcap_{i=1}^k A_i} = o(\e^k / n)$. We claim that for adequate constants $C_{\LR}(\e), C_h(\e)$, 
    \begin{equation*}
        \bP_u \sbra{A_1} \leq \e^{3} / 2 +  o(1) \quad \text{and} \quad \bP_u \cond{A_{i}}{G^{\ast}(i-1,s+t)} \leq \e^{3} / 2 + o(1)
    \end{equation*}
    for all $i \in [2, k]$.
    Noting that
    \begin{equation*}
        \bP_u \sbra{\bigcap_{i=1}^k A_i} = \bE_u \sbra{ \II_{\bigcap_{i=1}^{k-1} A_i} \bP_u \cond{A_k}{G^{\ast}(k-1,s+t)}},
    \end{equation*}
    the claims imply that $\bP_u \sbra{\bigcap_{i=1}^k A_i} = (\e^{3} / 2 + o(1))^{k} = o(\e^{k}/n)$ as desired.
    
    Let us prove the claims. By \eqref{eq:no_2_cycles} it suffices to bound the probability of these events when $G^{\ast}(k, s+t)$ has at most one long-range cycle. Observe that for each $i \leq k$ long-range distances and weights computed on the trajectories of $X^{(i)}, \cX^{(i)}$ up to time $s+t$ are measurable with respect to the graphs $G^{\ast}(i,s+t)$ and $\cG(i,s+t)$ and give the same quantity if computed on a common path, when these graphs are isomorphic. This is exactly ensured by \eqref{eq:coupling_uniform} and the choice of $k, s, t$, which also gives that $(X^{(i)}_{t'})_{t' \leq s+t} = \left(\iota(\cX^{(i)}_{t'})\right)_{t' \leq s+t}$ coincide with probability $1-o(1)$ conditional on $G^{\ast}(i-1, t)$. Thus $\bP_u \sbra{A_1} = \bP_O \cond{B_1}{\iota(O) = u} + o(1) $ and $\bP_u \cond{A_i}{G^{\ast}(i-1,s+t)} = \bP_u \cond{B_i}{G^{\ast}(i-1,s+t)} + o(1)$ for all $i \in [2, k]$. We thus establish the claimed bounds with $B_i$ in place of $A_i$.

    The bound for $B_1$ is a consequence of Proposition \ref{prop:drift} and Lemma \ref{lem:concentration_weights_QT}. 
    Let us prove the second bound. Let $i \geq 2$ and suppose $G^{\ast}(i-1, s+t)$ has at most one long-range cycle. We argue that the last exit time of $G^{\ast}(i-1, s+t)$ occurs before $s$ with high probability. If a cycle is present the chain $\cX^{(i)}$ eventually leaves the cycle and the loop-erased chain $(\xi^{(i)}_m)_{m}$ is well-defined. Furthermore, escape probabilities are everywhere lower bounded uniformly by some constant $q_0 > 0$. Thus there a constant $C$ such that for any starting vertex, by time $s$, the chain has already made $l = \Theta(s)$ steps of the loop-erased chain with quenched probability at least $1-e^{- C s}$. It thus suffices to show that these $l$ steps remain in $G^{\ast}(i-1, s+t)$ with probability exponentially small in $l$.  
    
    Consider two long-range edges $e,f$ that have an endpoint in a common small-range component. Using the uniform lower bound on escape probabilities, there exists $c=c(q_0)$ such that
    \begin{equation*}
        \bfP \cond{\xi_{m+1} = f}{\xi_m = e} \leq 1 - c,
    \end{equation*} 
    unless $f$ is the only long-range edge accessible from $e$. Since components corresponding to a communicating class of $V_1$ have at least three vertices, we deduce the existence of a constant $C_1 > 0$ such that for all long-range paths $\zeta$ of length $l$
    \begin{equation*}
        \bfP \sbra{\xi^{(j)}_{l} = \zeta_l} \leq e^{-C_1 l}
    \end{equation*}
    provided $l$ is large enough.
    Now, since $G^{\ast}(i-1,s+t)$ contains at most one long-range cycle this subgraph contains at most $2 k (s+t)$ long-range paths. By union bound the probability that the loop-erased trace $\xi^{(i)}$ follows one of these paths up to length $l$ is thus bounded by $2 k (s+t) e^{-C_1 l}$, proving the result.

    Let $L^{(i)}$ denote the last exit time of $G^{\ast}(i-1,s+t)$ by $\cX^{(i)}$. Conditional on $L^{(i)}$, the quasi-tree that contains the subsequent trajectory needs then to be generated conditional on the chain not going back. Since the latter probabilty is lower bounded by $q_0$, this only affects the usual law of the quasi-tree by a constant factor. Thus an appropriate choice of the constants $C_h, C_{\LR}(\e)$ yields $\bP_O \cond{B_j}{\iota(O) = u, B_{i-1}} \leq \e^{3} /2 + o(1)$, using Proposition \ref{prop:drift} and Lemma \ref{lem:concentration_weights_QT} with the fact that the trajectory outside $\cG(i-1,s+t)$ contains at least $t$ steps. 
\end{proof}
\section[Nice paths]{Approximation by nice paths: proof of Proposition \ref{prop:nice_approx}}\label{section:nice}

We now move to the proof of Proposition \ref{prop:nice_approx}: the approximation of the Markov kernel by nice paths. For earlier works that used this argument, see for instance \cite{benhamou2017cutoff,bordenave2018random}.

\subsection{Forward neighbourhood}

Nice paths between $x$ and $y$ have their first steps and last steps contained respectively in some quasi-tree-like neighbourhoods of $x$ and $y$. We define here the forward neighbourhood of $x$.

Let $x \in V$, $l \geq L$ an integer and $w_{\min} \geq 0$. The forward graph $K(x,l, w_{\min})$ is designed essentially as a "spanning quasi-tree" of the ball $\BLR(x,l)$, obtained by exploring this ball algorithmically, giving priority to paths with large weights and truncating whenever cycles are encountered. This process will thus build iteratively a sequence $(K_m)_{m=0}^{\tau}$ of subsets of $\BLR(x,l)$, until it stops at a random time $\tau$ to yield $K(x,l, w_{\min}) := K_{\tau}$. Unless the procedure is initiated at a vertex $x$ whose ball $\BLR(x,L)$ is not quasi-tree-like, $K_m$ remains quasi-tree-like at all times. In this case for every long-range edge $e \in E_m$ there exists a unique long-range path $\xi(e)$ from $x$ to $e$ contained in $K_m$. Define its cumulative weight as $\hat{w}(e) := w_x(\xi(e))$. Because weights require the knowledge of $(L-1)$-neighbourhoods, the exploration queue will consist of subsets $E_m$ of long-range edges for which the whole long-range $(L-1)$-neighbourhood is contained in $K_m$, so that cumulative weights can be computed from $K_m$ only. Finally, a constraint of minimal weights is added during the procedure, in order to keep the number of vertices explored as $o(n)$.

\paragraph*{Exploration of the forward neighbourhood}

The procedure is initiated by exploring $K_0 := \BLR(x,L)$. If $K_0$ contains a long-range cycle, $E_0 := \emptyset$ and the procedure stops. Otherwise let $E_0$ be the set of long-range edges at (long-range) distance $0$ from $x$. Then for all $m \geq 0$ the $(m+1)$-th step goes as follows: 
\begin{enumerate}
	\item Among all long-range edges $e$ in $E_m$ at long-range depth at most $l-L$ from $x$ in $K_m$ and such that $\hat{w}(e) \geq w_{\min}$, pick the edge $e_{m+1}$ with maximal cumulative weight, using an arbitrary ordering of the vertices to break ties. If there is no such edge, the procedure stops.
	\item Explore the depth-$L$ neighbourhood of $e_{m+1}$: for each descendant $z \in \partial K_m$ at long-range distance $L-1$ from $e_{m+1}$, reveal $\BSR(\eta(z),R)$. This exploration phase stops if a revealed edge violates the quasi-tree structure: this occurs if for some $z$ the small-range ball $\BSR(\eta(z), R)$ has a non-empty intersection with $K_m$ or one of the previously revealed balls.
	\item If the previous exploration phase stopped because of intersecting small-range balls, then $E_{m+1} := E_m \smallsetminus \{ e_{m+1} \}$ and $K_{m+1} := K_m$. If it stopped because of an intersection with $K_m$, let $Z_{m}$ be this intersection. Then $E_{m+1}$ is obtained by deleting from $E_m$ every long-range edge which has either a descendant or an ancestor in $Z_m$, as well as the edge $e_{m+1}$, and set $K_{m+1} := K_m$. Finally, if the exploration ended without a violation of the quasi-tree structure, $E_{m+1}$ is obtained by removing $e_{m+1}$ but adding all long-range edges emanating from the small-range component of the endpoint of $e_{m+1}$ which is furthest from $x$, whereas all newly revealed vertices are added to $K_{m+1}$.
\end{enumerate}

When the procedure ends, by construction the set $E_{\tau}$ consists of edges at long-range distance $l - L$ from $x$, which contain no long-range cycles in their $(L-1)$-neighbourhood and whose weights are measurable with respect to $K(x,l,w_{\min})$.   

\begin{lemma}
	Let $\kappa_m$ denote the number of long-range edges revealed during the first $m$ steps and $\kappa(x,l, w_{\min}) := \kappa_{\tau}$ the total number of long-range edges revealed during the construction of $K(x,l, w_{\min})$. Suppose $\BLR(x,L)$ is quasi-tree-like so that $\tau \geq 1$. There exists a constant $C > 0$ such that for all $m \in [1,\tau]$
	\begin{equation}\label{eq:bound_kappa_m}
		\hat{w}(e_m) \leq C \frac{l \Delta^{R(L+1)} }{\kappa_m}.
	\end{equation}
	In particular
	\begin{equation}\label{eq:bound_kappa}
		\kappa(x,l, w_{\min}) \leq C \frac{l \Delta^{R(L+1)} }{w_{\min}}.
	\end{equation}
\end{lemma}

\begin{proof}
	The set $E_{\leq m} := \bigcup_{k \leq m} E_k$ is the set of all long-range edges contained in $K_m$ which have their $L$ long-range neighbourhood contained in $K_m$. Furthermore, at time $m$ the procedure did not explore beyond long-range distance $(L-1)$ from these edges. Since long-range balls of radius $L$ contain $O(\Delta^{R(L+1)})$ vertices, one has $\kappa_{m} \leq C \abs{E_{\leq m}} \Delta^{RL}$ for some constant $C > 0$.
	
	The quasi-tree structure implies that the set $E_{\leq m}$ can be arranged as a rooted tree, where long-range edges are linked to the same vertex if they are at long-range distance $0$ from each other. The set $E_m$ is the set of edges furthest from the root, which lead to leaves. For every $e \in E_m$, there is a unique shortest path from the root to $e$, and conversely every edge is on such a path. These paths correspond to long-range paths in $K_m$ from $x$ to $e$, which have long-range length at most $l$, so we can deduce $\abs{E_{\leq m}} \leq l \abs{E_m}$. 

	On the other hand, let $E'_{m}$ be the penultimate layer of the tree, i.e. the edges which lead to $E_{m}$. Then we can bound $\abs{E_{m}} \leq O(\Delta^{R}) \abs{E'_{m}}$ and notice that the sum of weights over all shortest long-range paths from the root to $E'_m$ is bounded by $1$. Furthermore, the choice of a maximal weight in step $1$ ensures that the weights consecutively chosen are non-increasing. Therefore every shortest path from the root to $E'_m$ must have weight at least $\hat{w}(e_m)$, so we deduce $\abs{E'_m} \hat{w}(e_m) \leq 1$. All in all this shows that $\kappa_m \leq C' l \Delta^{R(L+1)} / \hat{w}(e_m)$ for some constant $C' > 0$.
	Finally, the consideration of $e_m$ during the exploration processes requires that $\hat{w}(e_m) \geq w_{\min}$ for all $m$, in particular when the procedure stops at $m = \tau$, hence the bound on $\kappa_{\tau} = \kappa(x, l, w_{\min})$.
\end{proof}

The following lemma will bound the probability to exit $K(x,l,w_{\min})$ at an edge where the quasi-tree structure was violated.

\begin{lemma}\label{lem:forward_martingale}
	If $m < \tau$, let $\mathrm{cycle}(e_{m+1})$ be the event that the exploration of the $L$-long-range neighbourhood of the long-range edge $e_{m+1} \in E_{m}$ considered at the $(m+1)$-th step revealed a cycle. Let $\e \in (0,1)$. Consider the following process. If $\BLR(x,L)$ is not quasi-tree-like, let $W_m := 1$ for all $m \geq 0$, otherwise set $W_0 := 0$ and for all $m \geq 0$ define
	\begin{equation*}
		W_{m+1} := W_{m} + \hat{w}(e_{m+1}) \II_{m < \tau} \II_{\mathrm{cycle}(e_{m+1})} \II_{\hat{w}(e_{m+1}) \leq \e/2}.
	\end{equation*}
	This is the total cumulative weight of edges that violated the quasi-tree structure at step $m+1$ and whose weight is below $\e / 2$. Suppose $l = O(\log n)$, $R,L = O(\log \log n)$ and that $w_{\min} \geq e^{ (\log \log n)^{3}} / n$. Then for all $s=s(n)$, with high probability, for all $x \in V$,
	\begin{equation}
		W_{\tau} \leq W_{s} + \e.
	\end{equation}
\end{lemma}

\begin{proof}
	Fix $\e \in (0,1)$, $x \in V$ and $s=s(n)$. Suppose $\BLR(x,L)$ is quasi-tree-like, otherwise the result is trivial.
	Let $(\cF_m)_{m \geq 0}$ be the standard filtration of the random graphs $(K_m)_{m}$. The choice of the edge $e_{m+1}$ is $\cF_m$-measurable. Averaging conditional on $\cF_m$, the generation of the $L+1$ long-range neighbourhood of the edge $e_{m+1}$ requires the sampling of at most $\Delta^{RL}$ long-range edges. For the first edge, there are exactly $n - \kappa_m$ possibilities. The quasi-tree structure is violated if an edge is sampled whose endpoint is at small-range distance at most $R$ from a previous ball explored in the same phase or from $K_m$. In this case it is necessarily at small-range distance at most $R$ from a long-range edge of $K_m$. Hence the conditional probability of $\mathrm{cycle}(e_{m+1})$ is upper bounded by 
	\begin{equation*}
		\bP \cond{\mathrm{cycle}(e_{m+1})}{\cF_m} \leq \frac{\Delta^{R} (\Delta^{R L} + \kappa_m)}{n - \kappa_m - \Delta^{R L}}.
	\end{equation*} 
	Bounding $\hat{w}(e_{m+1}) \leq \hat{w}(e_{m})$ and
	\begin{equation*}
		\bE \cond{W_{m+1} - W_{m}}{\cF_{m}} \leq \II_{m < \tau} \hat{w}(e_{m}) \bP \cond{\mathrm{cycle}(e_{m+1})}{\cF_m}
	\end{equation*}
	\eqref{eq:bound_kappa_m} implies that for some constant $C > 0$
	\begin{align*}
		\bE \cond{W_{m+1} - W_{m}}{\cF_{m}} &\leq \II_{m < \tau} \frac{C l \Delta^{2 R(L+1)}}{n - \kappa_\tau - \Delta^{R L}}, \\
		\bE \cond{(W_{m+1} - W_{m})^{2}}{\cF_{m}} &\leq \II_{m < \tau} \frac{C l^2 \Delta^{3 R(L+1)}}{\kappa_m (n - \kappa_\tau- \Delta^{R L})}.
	\end{align*}
	Furthermore, since every iteration of the procedure reveals at least one long-range edge, $\kappa_m \geq m$, in particular $\tau \leq \kappa_{\tau}$, hence summing over $m$ yields
	\begin{align*}
		a := \sum_{m=s}^{\tau - 1} \bE \cond{W_{m+1} - W_{m}}{\cF_{m}} &\leq C' \kappa_{\tau} \frac{l \Delta^{2 R(L+1)}}{n - \kappa_\tau - \Delta^{R L}}, \\
		b := \sum_{m=s}^{\tau - 1}  \bE \cond{(W_{m+1} - W_{m})^{2}}{\cF_{m}} &\leq C' \log (\kappa_{\tau}) \frac{l^2 \Delta^{3 R(L+1)}}{n - \kappa_\tau - \Delta^{R L}}
	\end{align*}
	for some other constant $C' > 0$. Using \eqref{eq:bound_kappa}, the assumptions made on the different parameters imply that $a = o(1)$ and $b = n^{-1 + o(1)}$.
	Consider now the martingale $M_k$ defined by 
	\begin{equation*}
		M_k := \frac{2}{\e} \left( W_{k} - W_{s} - \sum_{m=s}^{k-1} \bE \cond{W_{m+1} - W_{m}}{\cF_{m}} \right).
	\end{equation*}
	Its increments are bounded by $1$ and by construction $W_{\tau} - W_{s} = \frac{\e}{2} (M_{\tau} + a)$. Since $a = o(1)$, we infer that for large enough $n$
	\begin{equation*}
		\bP \sbra{W_{\tau} - W_{s} > 2 \e} \leq \bP \sbra{M_{\tau} > 2} \leq \bP \sbra{\exists k > 0: M_{k} > 2}.
	\end{equation*}
	Thus we can apply Theorem 1.6 of \cite{freedman1975tail} to bound 
	\begin{equation*}
		\bP \sbra{\exists k > 0: M_{k} > 2} \leq e^{2} \left( \frac{4b/\e}{2 + 4b/\e} \right)^{2 + 4b/\e} \leq (2 b \e^{-2})^{2}.
	\end{equation*}
	Since $b = n^{-1+o(1)}$, the right-hand side is $o(1/n)$.
\end{proof}

\subsection{Nice paths: definition}

We now define nice paths in order to prove Proposition \ref{prop:nice_approx}.
For the rest of this section set
\begin{equation*}
	R := C_R \log \log n, \qquad L := C_L \log \log n, \qquad M := C_M \log \log n
\end{equation*}
where $C_R, C_L, C_M > 0$ are constants chosen large enough so that the conclusions of Lemma \ref{lem:typical_paths} and Proposition \ref{prop:concentration_G} hold. We will also require the quantities $h$ and $\mathscr{d}$ of Proposition \ref{prop:concentration_G}. By the second point of Lemma \ref{lem:QTlike}, there exists a constant $\alpha > 0$ such that with high probability, from any starting point the chain $\scP$ has quasi-tree-like neighbourhood of radius $\lfloor \alpha \log n \rfloor$ after $\log (n) / (2h)$ steps. Fix $\e \in (0,1)$ for the rest of this section. In the sequel, we consider several constants $C_0, \ldots, C_5$, defined in terms of the constants $C_{\LR}(\e), C_h(\e)$ of Proposition \ref{prop:concentration_G}, which can in particular depend on $\e$. They are indexed in the reverse order in which they are fixed, so $C_0$ is chosen after $C_1$, which is chosen after $C_2$, etc. Let
\begin{equation}\label{eq:nice_parameters}
	\begin{gathered}
		t := \left\lfloor \frac{\log n}{h} + C_0 \sqrt{\log n} \right\rfloor, \quad s := \left\lfloor \alpha \log n \right\rfloor \wedge \left\lfloor \frac{\log n}{10 h} \right\rfloor, \quad l_1 := \mathscr{d} (t-s) - C_4 \sqrt{t} \\ w_{\min} := e^{- h (t-s) - C_h \sqrt{t}} \qquad w_{\max} := e^{-h t + C_1 \sqrt{t}}.
	\end{gathered}
\end{equation}
Given $x,y \in V, r \in \bN / 2$, $L \leq r \leq M$ and $l_3 \in [\mathscr{d} s - C_5 \sqrt{s}, \mathscr{d} s + C_5 \sqrt{s}]$, consider the following three-stage exploration of the environment: 
\begin{enumerate}
	\item Explore $K=K(x,l_1, w_{\min})$ as explained in the previous section. Let $E := E_{\tau}$ be the set of long-range edges remaining in the exploration queue at the end of the procedure, and consider the set $E'$ of boundary vertices at long-range distance $l_1$ from $x$, whose image under $\eta$ is yet to be determined. 
 	\item Explore the backward neighbourhood $B = B(y,r+s,l_3) := B_{\scP}(y,r+s) \cap \BLR(y,l_3)$. 
  	\item Finally, reveal everything else.
\end{enumerate}

It will be crucial in the sequel to control the numbers $N_1, N_2$ of long-range edges revealed during the two first stages. By definition $N_1 := \kappa(x,l_1, w_{\min})$. Observe that for any $\e' < \alpha h \wedge \frac{1}{10}$, we have $h (t-s) < (1-\e') \log n + O (\sqrt{\log n})$. Thus by \eqref{eq:bound_kappa}, 
\begin{equation*}
	N_1 = O( l_1 \Delta^{RL} e^{h (t-s) + C_h \sqrt{t}}) = O(\log n \ \Delta^{C_R C_L (\log \log n)^2} n^{1 - \e'} e^{C \sqrt{\log n}})
\end{equation*}
for some constant $C$. The cardinality of $B$ is bounded by that of $B_{\scP}(y,r+s)$. Up to choosing $\alpha < (10 \log \Delta)^{-1}$, this is $O(\Delta^{r+s}) = O(n^{2/10})$ as $r \leq M = O(\log \log n)$. 
All in all, for any $\e' < \alpha h \wedge \frac{1}{10}$, we have that
\begin{equation}\label{eq:stages}
	\begin{split}
		&N_1 = O(n^{1-\e'}), \\
		&N_2 = O(n^{1/5}).
	\end{split}
\end{equation}

Let $\cF_{r,l_3}$ be the $\sigma$-algebra generated by the long-range edges revealed during the two first stages. Unless the procedure stopped immediately, $K$ is quasi-tree-like, so a non-backtracking long-range path $\xi$ from $x$ to $E'$ entirely contained in $K$ must cross a unique edge of $E$. Let $\xi_E$ denote this edge and define $w_E(\xi) := w_{x,R,L}(\xi_1 \cdots \xi_E)$. This is essentially the weight $w_x(\xi)$, but where the last steps of the path were truncated to keep a weight that is $\cF_{r,l_3}$-measurable. 

In $B$, let $F'$ be the set of boundary vertices which are at long-range distance exactly $l_3$ from $y$, for which the shortest long-range path to $y$ is unique and has a tree-like neighbourhood in $B_{\scP}(y,r+s)$, i.e for each such path $\xi$, $B \cap \left( \bigcup_{z \in \xi} \BLR(z,L) \right)$ contains no long-range cycle, where the union is over the vertices of $\xi$. Consider now the set $F$ of long-range edges in $B$ that are at long-range distance $L-1$ from $F'$. If $\xi$ is a non-backtracking long-range path from $F'$ to $y$, let $\xi_F$ denote the unique edge of $F$ crossed by $\xi$ and $\xi_{F+1}$ the subsequent edge. Set $w_F(\xi) := w_{\xi_{F}^{+},R,L}(\xi_{F+1} \cdots \xi_{\abs{\xi} - L+1})$, where $\xi_{F}^{+}$ is the endpoint of $\xi_F$ closest to $y$. Here we truncate the path at both ends: the first steps to have a $\cF_{r,l}$-measurable weight but also the last steps, as the trajectories considered afterwards end at $y$. In the end the configuration we wish to have is illustrated in Figure \ref{fig:nice_paths}.

\begin{figure}
	\centering
	\includegraphics[scale=0.6]{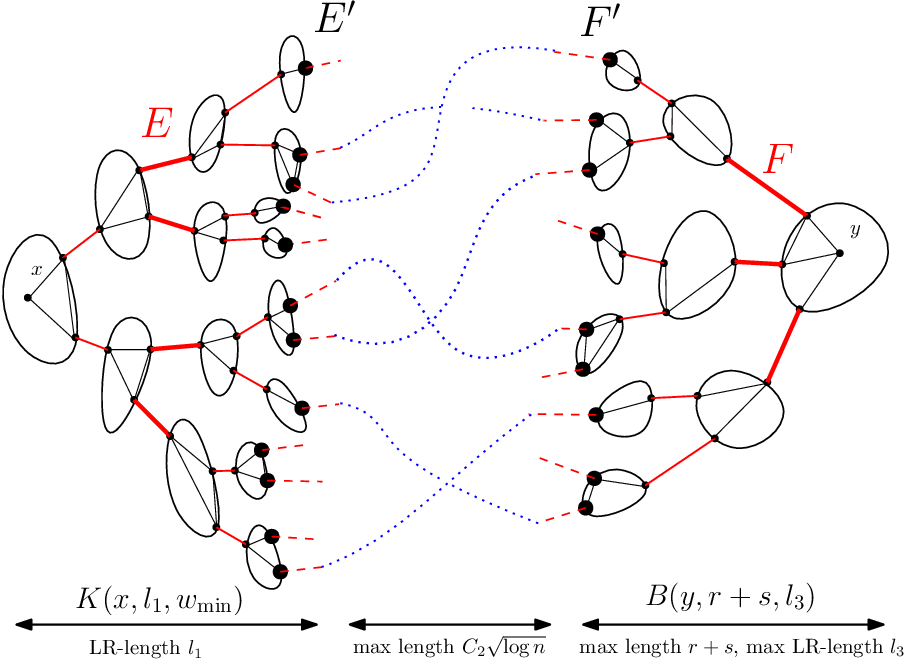}
	\caption{Illustration of nice paths.}
	\label{fig:nice_paths}
\end{figure}

\begin{definition}\label{def:nice_paths}
	Given $r \in \bN +1/2$, $L \leq r \leq M$ and $l_3 \in [\mathscr{d} s - C_5 \sqrt{\log n}, \mathscr{d} s + C_5 \sqrt{\log n}]$ let $\frN^{t}_{r,l_3}(x,y)$ be the set of length $t$ paths $\frp$ between $x$ and $y$ such that
	\begin{enumerate}[label=(\roman*)]
		\item $\frp \in \Gamma(R,L,M)$
		\item $\frp$ can be decomposed as the concatenation $\frp = \frp_1 \frp_2 \frp_3$ of three paths such that: $\frp_1$ is a path from $x$ to $E'$ entirely contained in $K(x, l_1, w_{\min})$, whose endpoint is the only vertex of $E'$ it contains (which implies that it ends with a short-range step),
		\item $\frp_2$ is a path between $E'$ and $F'$ which starts and ends with a long-range step such that
		\begin{equation*}
			C_3 \sqrt{\log n} \leq \abs{\frp_2} \leq C_2 \sqrt{\log n}
		\end{equation*}
		for some $C_2, C_3 > 0$ and the endpoint of $\frp_2$ is the only vertex of $B(y,r+s,l_3)$ it contains. 		
		\item $\frp_3$ is a path of length $r+s$ from $F'$ to $y$ entirely contained in $B(y,r+s,l_3)$, which starts and ends with a small-range step and does not contain any regeneration edge in its first $r$ steps,
		\item $w_E(\xi(\frp_1)) w_F(\xi(\frp_3)) \leq w_{\max}$.
	\end{enumerate}
	In the sequel we consider 
	\begin{equation*}
		\scP^{t}_{r,l_3}(x,y) := \sum_{\frp \in \frN^{t}_{r,l_3}(x,y)} \scP(\frp).
	\end{equation*}
	The complete set of nice paths is obtained by taking the union over parameters $r,l_3$, namely 
	\begin{equation*}
		\frN^{t}(x,y) := \bigcup_{\substack{L \leq r \leq M \\ r \in \bN + 1/2}} \bigcup_{l_3 = \mathscr{d} s - C_5 \sqrt{s}}^{\mathscr{d} s + C_5 \sqrt{s}} \frN_{r,l_3}^t(x,y)
	\end{equation*}
	and the total probability of nice paths	by $\scP^{t}_{\frN}(x,y) := \sum_{\frp \in \frN^{t}(x,y)} \scP(\frp)$.
\end{definition} 

Conditions (i) and (ii) allow to relate the probability of following a nice path with the weight constraint (v): thanks to the tree structure of $K(x, l_1)$, each vertex in $E'$ has a unique ancestor edge in $E$. Given $e \in E$, let $E'(e)$ denote the set of vertices in $E'$ with ancestor $e$ and recall $\xi(e)$ is the unique non-backtracking long-range path from $x$ to $e$. Similarly for $f \in F$, let $F'(f)$ be the set of vertices of $F'$ from which the unique non-backtracking long-range path to $y$ goes through $f$ and $\xi(f)$ the unique non-backtracking long-range path from $f^{+}$ to $y$ (which thus does not include $f$). Then for fixed $r, l_3$ and fixed total long-range length $l$:
\begin{align*}
	\sum_{\substack{\frp \in \frN_{r,l_3}^{t}: \xi(\frp_1)_E = e \\ \xi(\frp_3)_F = f \\ \abs{\xi(\frp)} = l}} \scP(\frp) &\leq \sum_{\substack{\xi: \xi_E = e \\ \xi_F = f, \\ \abs{\xi} = l}} \bfP_{x} \sbra{\xi(X_0 \cdots X_{\tau_{l}})_{\leq l - L+1} = \xi, (X_0 \cdots X_{\tau_{l}}) \in \Gamma(R,L,M)} \\
	&\leq\sum_{\substack{\xi: \xi_E = e \\ \xi_F = f, \\ \abs{\xi} = l}} w_{x,R,L}((\xi)_{\leq l-L+1})
\end{align*}
by Lemma \ref{lem:probability_weights}. Here the second and third sum are over long-range paths of length $l$ between $x$ and $y$, with the first $l_1$ steps entirely contained in $K$ and going from $x$ to $E'$ through $\xi_E$, and the last $l_3$ steps in $B$, going from $F'$ to $y$ through $\xi_F$. We can decompose such a path in three steps: from $x$ to $E$, $E$ to $F$ and $F$ to $y$, which we write $\xi = (\xi_1 \cdots \xi_E)(\xi_{E+1} \cdots \xi_{F'}) (\xi_{F+1} \cdots \xi_l)$. Then 
\begin{align*}
	w_{x,R,L}((\xi)_{l-L+1}) &= w_{x,R,L}(\xi_1 \cdots \xi_E) w_{\xi_{E}^{+},R,L}(\xi_{E+1} \cdots \xi_{F}) w_{\xi_{F}^{+},R,L} (\xi_{F+1} \cdots \xi_{l-L+1}) \\
	&= w_{E}(\xi_1 \cdots \xi_E) w_{\xi_{E}^{+},R,L}(\xi_{E+1} \cdots \xi_{F}) w_{F}(\xi_{F+1} \cdots \xi_{l-L+1}).
\end{align*}
Observe now that for fixed $e$ and $f$ the first and third factors are fixed as well and determined by $\xi(\frp_1), \xi(\frp_3)$, so the sum is only over the steps $\xi_{E +1} \cdots \xi_{F-1}$. Since weights sum up to $1$, we can also sum over $l$ and bound
\begin{equation}\label{eq:nice_paths_bound}
	\sum_{\substack{\frp \in \frN_{r,l_3}^{t} \\ \xi(\frp_1) = \xi_1 \\ \xi(\frp_3) = \xi_3}} \scP(\frp) \leq w_{\max}
\end{equation}
thanks to the weight constraint (v). 
In words: the total probability of nice paths in $\frN_{r,l_3}$ with prescribed long-range edges in $E$ and $F$ is upper bounded by $w_{\max}$.

\subsection{Nice paths are typical}\label{subsec:nice_typical}

We now prove the first point of Proposition \ref{prop:nice_approx}. Let $s'= C \log \log n$ and $t = \log n / h + C_0 (\e) \sqrt{\log n}$: we need to prove that the last $t$ steps of a length $s' + t$ trajectory are nice with quenched probability at least $1-\e$. There are several properties to check. Since there are finitely many of them, once a property is shown to hold with probability $1-o_{\bP}(1)$ or $1 -\e$ we can automatically assume it is satisfied when checking the remaining properties. Write $t' := s' + t$ to simplify notation. 

We already proved in Lemma \ref{lem:typical_paths} that a failure of the requirement $(X_{s'} \cdots X_{t'}) \in \Gamma(R,L,M)$ occurs with probability $o_{\bP}(1)$. For fixed $r, l_3$ the very definition of the backward neighbourhood $B(y,r+s,l_3)$ implies that it necessarily contains the last $r+s$ steps of the trajectory provided this trajectory has long-range length $l_3$. Furthermore, condition (iii) implies the long-range edge crossed at time $t'-(r+s)$ to enter $B$ has only been crossed once, while  condition (iv) requires that the rest of the trajectory never exits $B$: therefore $t'-(r+s)$ must be a regeneration time. Condition (iv) also requires that there is no other regeneration time between $t'-(r+s)$ and $t'-(r+s) + M$, which is at least $t'-s$ as $r \leq M$, thus $t'-(r+s)$ is the last regeneration time before $t'-s$. Summing over $r$ and $l_3$, we are thus conditioning on the last regeneration time occurring in the interval $[t'-s-M,t'-s]$, the existence of which is ensured by the fact that $(X_{s'} \cdots X_{t'}) \in \Gamma(R,L,M)$. Assume now this property holds and let $T_{F'}$ denote the last regeneration before time $t'-s$. The remaining obstructions to following a nice path are: 
\begin{enumerate} 
	\item The first steps of $(X_{s'} \cdots X_{t'})$ are not contained in $K=K(X_{s'},l_1, w_{\min})$, which occurs if $\BLR(X_{s'}, L)$ is not quasi-tree-like. This occurs with probability $o_{\bP}(1)$ by Lemma \ref{lem:QTlike}.
	\item From $X_{s'}$, the chain leaves $K$ before it reaches long-range distance $l_1$, which occurs if: 
	\begin{itemize}
		\item the loop-erased trace exits $K$ through the $L$ long-range neighbourhood of an edge $e$ which satisfied $\hat{w}(e) < w_{\min}$: since cumulative weights along a path are non-increasing this implies that $- \log w(\xi(X_{s'} \cdots X_{t'-s})) > - \log w_{\min} = h (t-s) + C_h \sqrt{t}$ which occurs with probability less than $\e$ by Proposition \ref{prop:concentration_G}.
  		\item the chain crosses an edge which violated the quasi-tree structure. Recall the process considered in Lemma \ref{lem:forward_martingale}. As mentioned above we can suppose $\BLR(X_{s'},L)$ is quasi-tree-like. Then by Proposition \ref{prop:concentration_G}, the total probability of paths with long-range length $l_1$ and weights above $\e / 2$ is $o_{\bP}(1)$. Thus up to this $o_{\bP}(1)$ error, the quantity $W_{\tau}$ considered in the lemma exactly counts the probability to exit $K$ at an edge which violated the tree structure. Thanks to the choice of $w_{\min}$ \eqref{eq:nice_parameters} the lemma establishes that with high probability, $W_{\tau} \leq W_{\lfloor L/2 \rfloor} + \e$ for any value of $X_{s'}$. However $W_{\lfloor L/2 \rfloor} = 0$ as $\BLR(X_{s'},L)$ is quasi-tree-like.
	\end{itemize}
	\item The loop-erased trace of the trajectory between $X_{T_{F'}}$ and $X_{t'}$ has length $l_3 \notin [\mathscr{d} s - C_5 \sqrt{\log n}, \mathscr{d} s + C_5 \sqrt{\log n}]$ or there is another long-range path of length $l_3$ that joins these two vertices and that can be crossed in less than $s+M$ steps. This non-uniqueness is ruled out by Lemma \ref{lem:QTlike} which shows that the ball $B_{\scP}(X_{t'-s-M}, s+M)$ is quasi-tree-like with probability $1-o_{\bP}(1)$, as $M = o(s)$. For the long-range distance requirement, Proposition \ref{prop:concentration_G} shows that $\abs{\abs{\xi(X_{t'-s} \ldots X_{t'})} - \mathscr{d} s} \leq C_{\LR} \sqrt{s}$ with probability at least $1- \e$. Then note that the long-range distance traveled in the intervals $[T_{F'},t']$ and $[t'-s,t']$ differ by at most $M$, hence $\abs{\abs{\xi(X_{T_{F'}} \cdots X_{t'})} - \mathscr{d} s} \leq C_{\LR} \sqrt{s} + M \leq C_5 \sqrt{s}$ for a large enough constant $C_5$, using that $M=o(\sqrt{s})$.
	\item The intermediate path $\frp_2$ does not have length $\Theta(\sqrt{t})$. Observe that the long-range length is sub-additive. Since by definition a nice trajectory decomposes as the concatenation $(X_{s'} \cdots X_{t'}) = \frp_1 \frp_2 \frp_3$ the sub-additivity implies 
	\begin{equation*}
		\abs{\xi(X_{s'} \cdots X_{s'+t})} \leq \abs{\xi(\frp_1)} + \abs{\xi(\frp_2)} + \abs{\xi(\frp_3)}.
	\end{equation*}
	The path $\xi(\frp_1)$ has length $l_1$, whereas $\xi(\frp_3)$ has variable length but from the bounds on $l_3$ in Definition \ref{def:nice_paths} and \eqref{eq:nice_parameters} we infer that their combined length is
	\begin{equation*}
		\abs{\xi(\frp_1)} + \abs{\xi(\frp_3)} \leq \mathscr{d} t - (C_4 - C_5) \sqrt{t},
	\end{equation*}
	while the intermediate path obviously has length $\abs{\xi(\frp_2)} \leq \abs{\frp_2}$. Choose $C_4 \geq C_5 + 2 C_{\LR}$. Hence if $\abs{\frp_2} < C_3 \sqrt{t}$ with $C_3 := C_{\LR}$, then $\abs{\xi(X_{s'} \cdots X_{s'+t})} < \mathscr{d} t - C_{\LR} \sqrt{t}$, which occurs with probability at most $o_{\bP}(1)$ by Proposition \ref{prop:concentration_G}. 
	To prove the upper bound on $\abs{\frp_2}$, observe that $s' + \abs{\frp_1}$ coincides with the hitting time of a vertex at long-range distance $l_1$ from $X_{s'}$ (it is not however the first such vertex, since $\frp_1$ ends at a vertex $E'$). We can thus in particular lower bound $s' + \abs{\frp_1}$ by the first hitting time $\tau_{l_1}$ of a vertex at long-range distance $l_1$ from $X_{s'}$. From Proposition \ref{prop:concentration_G} we can deduce the existence of $C_2> 0$ such that $\tau_{l_1} - s' \geq t -s - C_2 \sqrt{t} $ with probability at least $1 -\e$. Since $\abs{\frp_3} \geq s$, we deduce that
	\begin{equation*}
		\abs{\frp_2} = t - \abs{\frp_1} - \abs{\frp_3} \leq C_2 \sqrt{t}.
	\end{equation*}
	\item $w_{E}(\xi(X_{s'} \cdots X_{\tau_{l_1}})) w_{F} (\xi(X_{T_{F'}} \cdots X_{t'})) > w_{\max}$: let $\xi := \xi(X_{s'} \cdots X_{t'})$. When deriving \eqref{eq:nice_paths_bound} we used that
	\begin{equation*}
		w_{x,R,L}(\xi) = w_{E}(\xi_1 \cdots \xi_{E}) w_{\xi_{E}^{+},R,L}(\xi_{E+1} \cdots \xi_{F}) w_{F}(\xi_{F+1} \cdots \xi_{l-L+1})
	\end{equation*}
	where $\xi_l$ is the last edge of $\xi$. Now by the non-backtracking property of nice paths $\xi(X_{s'} \cdots X_{\tau_{l_1}})$ and $\xi(X_{T_{F'}} \cdots X_{t'})$ contain $\xi_1 \cdots \xi_{E}$ and $\xi_{F+1} \cdots \xi_{l}$ respectively so the goal is to show that $w_{E}(\xi_1 \cdots \xi_{E}) w_{F}(\xi_{F+1} \cdots \xi_l) > w_{\max}$ occurs with probability at most $\e$. As we argued for the previous point, Proposition \ref{prop:concentration_G} implies $\tau_{l_1} - s' \geq t_1 := t-s-C_2 \sqrt{t}$ with probability at least $1 - \e$. Since cumulative weights are non-increasing along a path, we deduce from Proposition \ref{prop:concentration_G} that 
	\begin{equation*}
		w_{E}(\xi_1 \cdots \xi_E) \leq w_{x,R,L}(\xi(X_{s'} \cdots X_{s'+t_1})) \leq e^{- t_1 h + C_h \sqrt{t_1}}
	\end{equation*}
	with probability at least $1 - \e$.
	Similarly we know that $T_{F'} \leq t'-s$, hence
	\begin{equation*}
		w_{E}(\xi_1 \cdots \xi_F) \geq w(\xi(X_{s'} \cdots X_{s'+t-s})) \geq e^{-(t-s) h - C_h \sqrt{t-s}}
	\end{equation*}
	with probability at least $1 - \e$. From these two bounds, we deduce
	\begin{equation*}
		w_{\xi_{E}^{+},R,L}(\xi_{E+1} \cdots \xi_{F}) \geq e^{- C_{h} \sqrt{t-s} - C_{h} \sqrt{t_1} - C_2 h \sqrt{t}} \geq e^{-C' \sqrt{t}}
	\end{equation*}
	for some $C' = C'(\e) > 0$. Thus if $w_{E}(\xi_1 \cdots \xi_{E}) w_{F}(\xi_{F+1} \cdots \xi_l) > w_{\max}$ we obtain that
	\begin{equation*}
		w_{x,R,L}(\xi) \geq w_{\max} e^{-C' \sqrt{t}}
	\end{equation*}
	which has probability at most $\e$ if the constant $C_1$ in $w_{\max}$ \eqref{eq:nice_parameters} is taken sufficiently large.
\end{enumerate}

\subsection{Concentration of nice paths: proof}

Recall that $\cF_{r,l}$ is the $\sigma$-algebra generated by the long-range edges revealed during the two first stages. For fixed $r$ and $l$, we prove concentration of $\scP_{\frN_{r,l}}^{t}(x,y)$ conditional on $\cF_{r,l}$ using Theorem \ref{thm:tensor_concentration}. The concentration will be strong enough for a union bound, which will yield Point \ref{enum:nice_upper_bound} of Proposition \ref{prop:nice_approx}. Conditional on $\cF_{r,l}$ all the randomness of $\scP_{\frN_{r,l}}^{t}(x,y)$ comes from the intermediate path $\frp_2$. It is contained in a random environment that arises from a uniform bijection $\sigma'$ between the two subsets $V'_1 \subset V_1$, $V'_2 \subset V_2$ that remain to be matched after the two first stages. As was observed in the previous section, the two first stages revealed a number of long-range edges which is $O(n^{1- \e'})$, so $\sigma'$ is a uniform bijection between sets of $n' = n - o(n)$ elements, which can be identified with a uniform permutation of $n'$ elements. 

The following lemma proves that the conditions of concentration are fulfilled. We use Corollary \ref{coroll:concentration} which is more practical to use than Theorem \ref{thm:tensor_concentration}.

\begin{lemma}\label{lem:conditions_concentration}
	For all $r \leq M, l_3 \geq 0$, conditional on $\cF_{r,l_3}$, $\scP_{\frN_{r,l_3}}^{t}(x,y)$ can be realized as a multilinear function $\phi$ on $S_{n'}$ of degree at most
	\begin{equation}\label{eq:degree_d}
		d := C_2 \sqrt{\log n}.
	\end{equation}
	With the notations of Corollary \ref{coroll:concentration},
	\begin{equation*}
		\begin{gathered}
		\alpha_{\phi} = O \left( d^{3} (4 \Delta)^{d} w_{\max} \log n \right),\quad A_{\phi} = O(d^{2} \Delta^{d} n^{2/10} w_{\max}) \\
		\text{and} \quad A_{\nabla \phi} = O \left( d^{2}  (2 \Delta)^{d} w_{\max} \right).
		\end{gathered}
	\end{equation*}
\end{lemma}

\begin{proof}
	A pair of multi-indices $(\bfi, \bfj)$ of size $k$ identifies with a sequence of $k$ potential long-range edges $(i_1, j_1), \ldots (i_k, j_k)$. Note that all properties required of nice paths are either $\cF_{r,l}$-measurable or can be determined from the path (this would not be the case if we had chosen for instance to take the weight of $w(\xi(\frp_2))$ into account). Therefore it is possible to define the coefficient $a_{\bfi, \bfj}$ as the total probability of nice paths which meet exactly these long-range edges in addition to those that are contained in $K$ or $B$ (which do not count as random here). Nice paths do not have to cross the long-range edges but may only pass through one endpoint, due to the fact that the probability $p = p(x, \sigma(x))$ of crossing a long-range edge depends in general on $\sigma$.

	The upper bound on the degree follows easily from this definition. By Point (iii) in Definition \ref{def:nice_paths}, the random part $\frp_2$ of nice paths has length bounded by $d = C_2 \sqrt{\log n}$ which thus also upper bounds the number of long-range edges met by nice paths in the intermediate section and consequently the degree.

	To prove concentration recall the notation of Proposition \ref{prop:concentration_constants}. Consider a set $S$ of long-range edges of size $d$ or less. It is connected to at most $d$ vertices of $E'$ and $F'$, which in turn correspond to at most $d$ edges in $E$ and $F$. Therefore the set of nice paths that meet the edges of $S$ exactly can cross at most $d$ edges of $E$ and $d$ edges of $F$, hence by \eqref{eq:nice_paths_bound} we deduce that the maximal coefficient of $\phi$ is bounded by $d^{2} w_{\max}$. By multilinearity, the maximal coefficient is non-increasing with respect to partial differentiation so
	\begin{equation*}
		M(\partial_{ij} \phi) \leq M(\phi) \leq d^{2} w_{\max}
	\end{equation*}
	for any potential long-range edge $(i,j)$. On the other hand, the number of monomials in $\phi$ can be upper bounded by the maximal number of paths between $E'$ and $F'$ of length $d$. Since $F'$ has cardinality at most $n^{1 / 5}$ by \eqref{eq:stages}, we have
	\begin{equation*}
		N(\phi) \leq \abs{F'} \Delta^{d} = O(\Delta^{d} n^{1/5}).
	\end{equation*}
	Restricting to a potential long-range edge $(i,j)$, the latter can be part of at most $O((2 \Delta)^{d})$ paths of length $d$ between $E'$ and $F'$ and thus as many monomials. Therefore 
	\begin{equation*}
		N(\partial_{ij}\phi) \leq O((2 \Delta)^{d}).
	\end{equation*} 
	We deduce that $A_{\phi} = O(d^{2} \Delta^{d} n^{1/5} w_{\max})$ and $A_{\nabla \phi} = O \left( d^{2}  (2 \Delta)^{d} w_{\max} \right)$. Note that these are $o(1)$ and that the map $x \mapsto x \log(1/x)$ is increasing on $(0,1/e)$, so we can bound the term $A_{\nabla \phi} \log(A_{\phi}/A_{\nabla_{\phi}})^{+}$ in Corollary \ref{coroll:concentration} with the upper bounds on $A_{\phi}$ and $A_{\nabla \phi}$ to obtain that
	\begin{equation*}
		\alpha_{\phi} = O \left( d^{3} (4 \Delta)^{d} w_{\max} \log n \right).
	\end{equation*}
\end{proof}

The following Lemma will be sufficient to deduce Proposition \ref{prop:nice_approx}. The proof is postponed to the next section. Recall that the measure $\bfQ^{(L)}_u$ was defined in \eqref{eq:bfQu_G}. To ease notation we will drop the superscript $(L)$ in the sequel. Recall also that $T_1^{(G,L)}, T_1^{(\cG,L)}$ denote regeneration times with horizon $L$ in $G^{\ast}$ and $\cG$ respectively, $\mu$ is the invariant measure of the regeneration chain in the quasi-tree and $\nu$ was considered in Proposition \ref{prop:mixing_annealed_QT}. Below we write $\bfQ_{\nu + c} = \sum_{u} (\nu(u) + c(u)) \bfQ_{u}$.

\begin{lemma}\label{lem:mixing_annealed_G}
	There exists a measure $c$ on $V$ such that $\sum_{v \in V} c(v) = o_{\bP}(1)$ and
	\begin{equation*}
		\bE \cond{ \scP_{\frN_{r,l_3}}^{t}(x,y) }{\cF_{r,l_3}} =  \frac{(1-o_{\bP}(1))}{\bE_{\bQ_{\mu}} \sbra{T_{1}^{(\cG,\infty)}}} \bfQ_{\nu + c} \sbra{ \begin{array}{c} X_{r+s} = y, \abs{\xi(X_0 \cdots X_{r+s})} = l_3 \\ r < T_1^{(G,L)} \leq M \end{array}}.
	\end{equation*}
\end{lemma}

\begin{proof}[Proof of Proposition \ref{prop:nice_approx}]
	The first point was proved in Section \ref{subsec:nice_typical}. 
	Suppose first $x, y \in V$ and $r,l_3$ are fixed. Let $\phi := \scP_{\frN_{r,l_3}}^{t}(x,y)$ as in Lemma \ref{lem:conditions_concentration} and $z:= \frac{\e}{2} \bE \sbra{\phi} + \frac{\e}{2 C_6 M n \sqrt{\log n}}$ for some $C_6 > 0$. 
	Note that $w_{\max} \leq e^{-C \sqrt{\log n}} / n$ for some constant $C$ which tends to $+ \infty$ as the constant $C_0$ in the definition of $t$ grows \eqref{eq:nice_parameters}, while other factors of $\alpha_{\phi}, A_{\nabla \phi}$ are all of order at most $e^{C' \sqrt{\log n}}$. Thus for any choice of $C = C(\e) > 0$, Lemma \ref{lem:conditions_concentration} shows that $\alpha_{\phi}$, $A_{\nabla \phi}$ can both be bounded by $e^{-C \sqrt{\log n}} / n$, provided the constant $C_0$ is sufficiently large, while $d^{2} (d-1) A_{\phi}/ n = o(n^{-3/5})$. In particular we can choose $C$ so that $A_{\nabla \phi} \leq z$. Since $\bE \sbra{\phi} \leq 2 z / \e$ and $z \geq \e / (2 C_6 M n \sqrt{\log n})$, applying Corollary \ref{coroll:concentration} yields 
	\begin{equation}\label{eq:concentration_nice}
		\bP \sbra{\abs{\phi - \bE \sbra{\phi}} \geq z} \leq 2 \exp \left( \frac{- C' \e^{2}}{\alpha_{\phi} M n \sqrt{\log n} } \right)
	\end{equation}
	for some $C' = C'(\e) > 0$. Up to increasing again the value of $C$, we can ensure that $\alpha_{\phi} M n \sqrt{\log n} \leq (\log n)^{-2}$, hence
	\begin{equation*}
		\bP \sbra{\abs{\phi - \bE \sbra{\phi}} > \frac{\e}{2} \bE \sbra{\phi} + \frac{\e}{2 C_6 M n \sqrt{\log n}} } \leq \exp \left(- C' \e^{2} (\log n)^{2} \right).
	\end{equation*}
	This is sufficient to take a union bound over $x,y \in V$ and $r \leq M, l_3 \in [\mathscr{d} s - C_5 \sqrt{\log n}, \mathscr{d} s + C_5 \sqrt{\log n}]$. Thus summing over $r,l_3$ we obtain that with high probability, for all $x,y \in V$,
	\begin{equation*}
		\abs{\scP_{\frN}^{t}(x,y) - \sum_{r,l_3} \bE \cond{\scP_{\frN_{r,l_3}}^{t}(x,y)}{\cF_{r,l_3}}} \leq \frac{\e}{2} \sum_{r,l_3} \bE \cond{\scP_{\frN_{r,l_3}}^{t}(x,y)}{\cF_{r,l_3}} + \frac{\e}{n}.
	\end{equation*}
	Lemma \ref{lem:mixing_annealed_G} gives an estimate of the conditional expectation which shows:
	\begin{align}
		\scP_{\frN}^{t}(x,y) &\leq \frac{1 + \e / 2}{\bE_{\bQ_{\mu}} \sbra{T_{1}^{(\cG,\infty)}}} \sum_{r = 0}^{M} \bfQ_{\nu + c} \sbra{X_{r+s} = y, r < T_1^{(G,L)} \leq M} + \frac{\e}{n} \nonumber \\
		&= (1 +\e /2) A \hat{\pi}(y) + \frac{(1 +\e /2 )A}{\bfE_{\bfQ_{\nu}} \sbra{T_{1}^{(G,L)}}} \sum_{r = 0}^{M} \bfQ_{c} \sbra{X_{r+s} = y, r < T_1^{(G,L)} \leq M} +\frac{\e}{n} \label{eq:upper_bound_nice}
	\end{align}
	with $A:= \frac{\bfE_{\bfQ_{\nu}} \sbra{T_{1}^{(G, L)} \II_{T_{1}^{(G, L)} \leq M}}}{\bE_{\bQ_{\mu}} \sbra{T_{1}^{(\cG, \infty)}}}$. Summing over $y \in V$ and $r \in [0, M]$ in the second term yields $(1 + \e /2) A \sum_{v \in V} c(v) \leq 2 A \e$ with high probability. On the other hand, from Lemma \ref{lem:mixing_annealed_G} we also have the lower bound
	\begin{equation*}
		\scP_{\frN}^{t}(x,y) \geq (1 - \e / 2)(1-o_{\bP}(1)) A \hat{\pi}(y) - \e / n,
	\end{equation*}
	which by summing over $y \in V$ shows that $A \leq 1 + 10 \e$ with high probability. Plugging this in the upper bound above yields the second statement of the proposition. 	
\end{proof}

\subsection{Expectation of nice paths}

\begin{proof}[Proof of Lemma \ref{lem:mixing_annealed_G}]
	By point (i) of Proposition \ref{prop:concentration_G} we can include the case where $\frN_{r,l_3}$ is empty in the $o_{\bP}(1)$ term. Otherwise let $\tau_{E'}$ be the hitting time of $E'$. By definition, a nice path requires that $\tau_{E'} \leq t-s$ and the first part $\frp_1$ of the path is the trajectory until $\tau_{E'}$. The second part of a nice path is the trajectory until hitting $F'$. Therefore by strong Markov's property, one can bound
	\begin{align}
		&\sum_{\frp \in \frN_{r,l_3}(x,y)} \scP(\frp) \leq \sum_{\substack{t_1 + t_2 =  t -(r+s) \\ C_3 \sqrt{\log n} \leq t_2 \leq C_2 \sqrt{\log n}}} \sum_{u \in E'} \sum_{v \in F'} \left( \bfP_{x} \sbra{\tau_{E'} = t_1, X_{t_1} = u} \right. \label{eq:lem_upperbd}\\
		& \times \bfP \cond{\exists k \geq 0: T_{k}^{(G,L)} = t_2, X_{t_2} = v}{X_{1/2} = u} \nonumber \\
		&\left. \times \bfP \cond{X_{r + s} = y, r < T_1 \leq M}{X_{1/2} = v, \tau_{\eta(v) > \tau_L}} \right) .\nonumber
	\end{align}
	In each term of this sum, the first and third factor are $\cF_{r,l_3}$-measurable, so only the second factor gets averaged when taking conditional expectation. We claim that this expectation satisfies 
	\begin{equation*}
		\sum_{v \in F'} \abs{ \bP \cond{\exists k \geq 0: T^{(G,L)}_{k} = t_2, X_{t_2} = v}{X_{1/2} = u, \cF_{r,l_3}} - \frac{\nu(v)}{\bE_{\bQ_{\mu}} \sbra{T_{1}^{(\cG,\infty)} }}} = o_{\bP}(1).
	\end{equation*}
	Note that $\nu / \bE_{\bQ_{\mu}} \sbra{T_{1}^{(\cG, \infty)}}$ is independent of $u$ and $t_1$, while the first factor in the bound \eqref{eq:lem_upperbd} can be summed over $t_1$ and $u$ to at most $1$. 
	We can also recognize the measure $\bfQ_v$ \eqref{eq:bfQu_G} in the third factor. Therefore provided the claim holds, one obtains that for some $c = (c_v)_{v}$ satisfying $\sum_{v \in V} c(v)= o_{\bP}(1)$, 
	\begin{align*}
		&\bE \cond{\sum_{\frp \in \frN_{r,l_3}} \scP(\frp)}{\cF_{r,l_3}} \leq \sum_{v \in F'} \left( \frac{\nu(v)}{\bE_{\bQ_{\mu}} \sbra{T_{1}^{(\cG,\infty)}}} + c(v) \right) \bfQ_{v} \sbra{X_{r + s} = y, r < T_1 \leq M} \\
		&\qquad \leq \sum_{v \in V} \frac{\nu(v) + c(v)}{\bE_{\bQ_{\mu}} \sbra{T_{1}^{(\cG,\infty)}}} \bfQ_{v} \sbra{X_{r + s} = y, d_{\LR}(v,y) = l_3, r < T_1 \leq M},
	\end{align*}
	where in the second line we implicitly changed the definition of $c$, using that $\bE_{\bQ_{\mu}} \sbra{T_{1}^{(\cG, \infty)}} = O(1)$. This proves an upper bound. To prove a matching lower bound, note that the inequality \eqref{eq:lem_upperbd} is not sharp if the trajectory is not in $\Gamma(R,L,M)$, which occurs with probability $o_{\bP}(1)$ by Lemma \ref{lem:typical_paths}, so the upper bound is also a lower bound up to a $o_{\bP}(1)$ error.
	
	Let us prove the claim. Let $u \in E', v \in F'$. The idea is to couple the chains $X$ and $\cX$ on a quasi-tree to relate their regeneration times, in order to use Proposition \ref{prop:mixing_annealed_QT}. However the conditioning by $\cF_{r,l_3}$ already revealed some-long range edges, which requires in turn a similar conditioning in the quasi-tree. We argue that the only conditioning required is on $\BLR(O,L)$, the ball of radius $L$ in a quasi-tree. Let $u_0$ be the ancestor of $E$ at long-range distance $L$ from $u$. By definition of $K(x,l_1,w_{\min})$, the long-range neighbourhood $\BLR(u_0,L)$ contains no long-range cycle and thus is a possible realization of $\BLR(O, L)$ around the root of a quasi-tree. Consider a quasi-tree $\cG$ which has this ball as the neighbourhood of its root and is completed with the standard procedure (so without consideration of the long-range edges revealed in $K \cup B$). Using the coupling of Section \ref{subsec:coupling}, the chain $X$ started at $u$ can thus be coupled with the chain $\cX$ on $\cG$, started at the vertex of $\BLR(O,L)$ that identifies with $u$. This coupling fails after $X$ enters $B$ or if it re-enters $K$ by another path than the one it used. Let $\tcoup$ denote this new coupling time. 
	
	 From \eqref{eq:stages}, the probability of sampling an element of either $K$ or $B$ under the uniform measure on $V$ is $O(1/n^{\e'})$ for some $\e' > 0$. Consequently, using the same comparison with a binomial we used in the proof of Lemma \ref{lem:coupling_quasiT}, the coupling fails by time $O(\sqrt{\log n})$ with probability $O(\log n / n^{2 \e'}) = o(1)$. Note that considering regeneration times requires the knowledge of the $L$ steps ahead but this is $O(\log \log n)$. Consequently it remains true that for $t_2 + L = O(\sqrt{\log n})$,
	 \begin{multline*}
		\sum_{v \in F'} \left| \bP \cond{\exists k \geq 0: T_{k}^{(G, L)} = t_2, X_{t_2} = v}{X_{1/2} = u, \cF_{r,l_3}} \right. \\ \left. - \, \bP \cond{\exists k \geq 0: T_{k}^{(\cG, L)} = t_2, \iota(\cX_{t_2}) = v}{\cX_{1/2} = x, \BLR(O,L)} \right| \leq \bP \cond{\tcoup > t}{\cF_{r,l_3}} \\ = o_{\bP}(1),
	\end{multline*}
	where $x \in \BLR(O,L)$ is a vertex at long-range distance $L$ with type $u$.
	Obviously, $\{T_{k}^{(G, \infty)}, k \geq 1 \} \subset \{T_{k}^{(G, L)}, k \geq 1 \}$. This inclusion may be strict however, if the chain backtracks over a long-range distance $L$.  Lemma \ref{lem:typical_paths_QT} shows this occurs before time $O(\sqrt{\log n})$ with probability $o(1)$. Consequently, with high probability $T_{k}^{(\cG,L)} = T_{k}^{(\cG,\infty)}$ for all regeneration times that occur before $t_2$, so we can exchange these random times in the equation above. 
	Now since $t_2 = \Theta(\sqrt{\log n})$, Proposition \ref{prop:mixing_annealed_QT} proves that
	\begin{equation*}
		\sum_{v \in V} \left| \bP \cond{\exists k \geq 0: T_{k}^{(\cG, \infty)} = t_2, \iota(\cX_{t_2}) = v}{\cX_{1/2} = x, \BLR(O,L)} - \frac{\nu(v)}{\bE_{\bQ_{\mu}} \sbra{T_{1}^{(\cG,\infty)}}} \right| = o(1).
	\end{equation*}
	Using triangle inequality to combine the two previous bounds yields the claim. 
\end{proof}
\section{Concentration for low-degree functions on the symmetric group: proof of Theorem \ref{thm:tensor_concentration}}\label{section:concentration}

We start proving Proposition \ref{prop:concentration_constants} in Section \ref{subsection:control} then prove the main Theorem in Section \ref{subsection:proof_main}.

\subsection[Control of the smoothness parameters]{Control of the smoothness parameters: proof of Proposition \ref{prop:concentration_constants}}\label{subsection:control}

	Let $\phi$ be a multilinear polynomial in the indeterminates $(X_{ij})_{i,j = 1}^{n}$ of degree $d$. The operator $D$ has the effect of replacing one indeterminate in each monomial by $1$ and taking the average. When doing so, some coefficients coming from different monomials of $\phi$ can be regrouped together. To avoid this, we can introduce new indeterminates instead of evaluating them at $1$. This is formalized by the following construction: for all $m \geq 1$ write $X^{(m)} = (X_{ij}^{(m)})_{i,j=1}^{n}$. Let $\tD$ be the linear operator on $\cR = \bR \sbra{X^{(1)}, X^{(2)}, \ldots }$ defined on monomials as 
	\begin{equation*}
		\tD : X^{(1)}_{i_1 j_1} \cdots X_{i_k j_k}^{(1)} X_{i_{k+1} j_{k+1}}^{(m_{k+1})} \ldots X^{(m_{d})}_{i_d j_d} \mapsto \frac{1}{kn} \sum_{l=1}^{k} X_{i_1 j_1}^{(1)} \cdots X_{i_{l-1} j_{l-1}}^{(1)} X_{i_l j_l}^{(m_{d}+1)} X_{i_{l+1} j_{l+1}}^{(m_{l+1})} \cdots X^{(m_{d})}_{i_d j_d},
	\end{equation*}
	where $m_1 = \ldots m_k = 1 < m_{k+1} \leq \cdots \leq m_{d}$. Let $T: \cR \rightarrow \bR \sbra{X_{ij}} $ be the operator that identifies $X^{(1)}$ with the $X_{ij}$ and evaluates all other indeterminates $X^{(2)}, X^{(3)}, \ldots$ at $1$.
	Then by construction for all $l \geq 1$,
	\begin{equation}\label{eq:tDphi}
		D^{l} \phi = T \tD^{l} \phi.
	\end{equation}
	Extend the definitions of $M, N$ to $\cR$ as follows: if $\psi \in \cR$, let $M(\psi)$ be the maximal coefficient in absolute value of $\psi$ and $N(\psi)$ the number of non-zero monomials. Note these are not $M(T \psi), N(T \psi)$: the point is precisely to distinguish monomials that would otherwise be regrouped together when applying $T$. If $\psi = \sum_{k=1}^{d} \psi_k \in \cR$ decomposes as the sum of functions that are homogeneous in $X^{(1)}$ of degree $k$, then for all $\sigma \in \frS_n$
	\begin{equation}\label{eq:TMN}
		\abs{T \psi(\sigma)} \leq \sum_{k=1}^{d} M(\psi_k) N(\psi_k)
	\end{equation}
	On the other hand if $\psi$ is homogeneous of degree $d \geq 1$, it is easy to see that
	\begin{equation}\label{eq:MN_inequality}
		M (\tD \psi) \leq \frac{1}{dn} M (\psi), \quad \text{and} \quad N (\tD \psi) \leq 2 dn \, N(\psi).
	\end{equation}
	Combining \eqref{eq:tDphi}, \eqref{eq:TMN} and \eqref{eq:MN_inequality}, we obtain that for all $\sigma \in \frS_n$ and $l \leq d$
	\begin{align*}
		\abs{D^{l} \phi(\sigma)} &\leq \sum_{k=1}^{d} M(\tD^{l} \phi_{k}) N(\tD^{l} \phi_{k}) \\
		&\leq 2^{l} \sum_{k=1}^{d} M(\phi_{k}) N(\phi_{k}) \\
		&\leq 2^{l} M(\phi) N(\phi)
	\end{align*}
	using that $M(\phi) = \max_{k \leq d} M(\phi_k)$ and $N(\phi) = \sum_{k=1}^{d} N(\phi_k)$.
 
	The inequality for $U \phi$ is proved similarly. Clearly $M(U \phi) \leq (dn)^{-1} M(\phi)$ if $\phi$ is homogeneous of degree $d$. On the other hand, monomials of $U \phi$ are obtained by picking a monomial of $\phi$, two indeterminates $X_{ij}, X_{kl}$ in this monomial and replace them by $X_{il}, X_{kj}$. Consequently every monomial of $\phi$ gives rise to at most $d (d - 1)$ monomials in $U \phi$. Thus we deduce that 
	\begin{equation*}
		M(U \phi) N(U \phi) \leq \frac{d-1}{n} M(\phi) N(\phi)
	\end{equation*}
	Then for a general function decomposing $\phi = \sum_{k=1}^{d} \phi_k$ as a sum of homogeneous functions implies that for all $\sigma \in \frS_n$
	\begin{align*}
		\abs{U \phi(\sigma)} &\leq \sum_{k=1}^{d} M(U \phi_k) N (U \phi_k) \\
		&\leq \sum_{k=1}^{d} \frac{k-1}{n} M(\phi_k) N(\phi_k) \\
		&\leq \frac{d-1}{n} M(\phi) N(\phi).
	\end{align*}

\subsection{Induction with the method of exchangeable pairs}\label{subsection:proof_main}

The proof of Theorem \ref{thm:tensor_concentration} follows the original argument of Chatterjee for the $d = 1$ case, based on Stein's method of exchangeable pairs. The basis of the method is quite standard: eventually the goal is to obtain a differential inequality on the moment generating function (mgf) that can be integrated and combined with Chernoff's inequality. In his work \cite{chatterjee2005thesis,chatterjee2007concentration,chatterjee2007stein} Chatterjee provided a general technique to obtain such a differential inequality thanks to the method of exchangeable pairs. An exchangeable pair is a pair of random variables $(X,X')$ invariant under permutation, so it has the same distribution as $(X', X)$. We rephrase Chatterjee's method as follows:

\begin{proposition}[{\cite[Proof of Thm. 1.5]{chatterjee2007stein}}]\label{prop:exchangeable}
	Let $\cX$ be a separable metric space, $(X,X')$ be an exchangeable pair of $\cX$-valued random variables. Let $f : \cX \rightarrow \bR$ such that $\bE \sbra{f} := \bE \sbra{f(X)} = 0$. Suppose $F: \cX \times \cX \rightarrow \bR$ is a square-integrable antisymmetric function which satisfies
	\begin{equation*}
		f(X) := \bE \sbra{F(X,X') \ | \ X} 
	\end{equation*}
	and $\bE \abs{e^{\theta f(x)} F(X,X')} < \infty$ for all $\theta \in \bR$. Then for all $\theta \in \bR$
	\begin{equation}\label{eq:mgf_bound_delta}
        \abs{M_{\phi}'(\theta)} \leq \abs{\theta} \bE \, \sbra{\Delta(X) e^{\theta X}}
    \end{equation}
	where we write $M_{\phi}(\theta) := \bE \sbra{e^{\theta \phi(X)}}$ and
	\begin{equation*}
		\Delta(X) := \frac{1}{2} \bE \sbra{ \abs{F(X,X') \left( f(X) - f(X') \right) } \ | \ X}.
	\end{equation*}
\end{proposition}

The method of exchangeable pairs consists thus in three main steps: 1. construct an exchangeable pair, 2. find a function $F$, 3. bound $\Delta(X)$ almost surely. The third point is by far the most challenging, especially as an almost sure bound is required. Bounding $\Delta(X)$ by a constant generally gives a suboptimal concentration bound, as the maximal value of $\Delta(X)$ can be very far from its typical value $\bE \sbra{\Delta(X)} = \Var(f(X))$. In that regard, it is natural to try improving the bound by leverage the concentration of $\Delta(X)$. This is similar to proving concentration bounds for self-bounding functions or weakly self-bounding functions \cite{boucheron2009selfbounding} and the idea behind the following lemma. Assuming a known upper bound on $\Delta$, it shows how to bound the moment generating function of $f$ from that of the upper bound $\Delta$, in particular if it involves the function $f$ itself. It is inspired by \cite[Thm. 3.3]{chatterjee2005thesis} and \cite[Lemma A.1]{mackey2014matrix}, where the method of exchangeable pairs is extended to obtain matrix concentration inequalities.

\begin{lemma}\label{lem:bootstrap_mgf}
    Consider the same setting as in the previous proposition. Suppose that there exist constants $A,B,C \geq 0 $ and a function $g$ such that
    \begin{equation*}
        \Delta(X) \leq A g(X) + B f(X) + C \qquad \text{a.s.}
    \end{equation*}
    and write $m_{f}(\theta) := \log \bE \sbra{ e^{\theta f(X)}}$, $m_{g}(\theta) := \log \bE \sbra{ e^{\theta g(X)}}$.
    For all $\eta >0$ and $\theta \geq 0$
    \begin{equation*}
        m_{f}(\theta) \leq \frac{\left(C + A m_{g}(\eta) / \eta \right) \theta^2 }{2 \left(1 -  B \theta - A \theta^2 / \eta \right)} \qquad \text{and} \qquad \abs{m_f(-\theta)} \leq \left(C + A m_{g}(\eta) / \eta \right) \theta^2 / 2,
    \end{equation*}
    provided the denominators in the right hand side are positive.
\end{lemma}

The proof relies on the following
\begin{proposition}[{\cite[Thm. 4.13]{boucheron2013concentration}}]
	Let $Y$ be a non-negative random variable. Define the entropy of $Y$ as
	\begin{equation*}
		\Ent(Y) := \bE \sbra{Y \log Y} - \bE \sbra{Y} \log \bE \sbra{Y}.
	\end{equation*}
	The entropy satisfies the variational relation
	\begin{equation}\label{eq:duality_entropy}
		\Ent(Y) = \sup_{W} \bE \sbra{(W - \log \bE \sbra{e^{W}}) Y}
	\end{equation}
	where the supremum is over all random variables with finite exponential moment.
\end{proposition}
    
    \begin{proof}[Proof of Lemma \ref{lem:bootstrap_mgf}]
    Recall $M_f(\theta) := \bE \sbra{e^{\theta f(X)}}$. Using the assumption together with Proposition \ref{prop:exchangeable} shows for all $\theta \in \bR$, $\eta >0$
	\begin{equation}\label{eq:mgf_bootstrap_proof}
		\abs{M_{f}'(\theta)} \leq A \abs{\theta} \eta^{-1} \bE \sbra{\eta g(X) e^{\theta f(X)}} + B \abs{\theta} M_f'(\theta) + C \abs{\theta} M_f(\theta).
	\end{equation}
	Then \eqref{eq:duality_entropy} allows to bound 
    \begin{align*}
        \bE \sbra{\eta g(X) e^{\theta f(X)}} &\leq \log \bE \sbra{e^{\eta g(X)}} \bE \sbra{e^{\theta f(X)}} + \mathrm{Ent} \left(e^{\theta f(X)} \right) \\
        &= m_{g}(\eta) M_{f}(\theta) + \mathrm{Ent} \left( e^{\theta f(X)} \right).
    \end{align*}
	However observe that 
	\begin{equation*}
		\frac{\Ent (e^{\theta f(X)})}{\bE \sbra{e^{\theta f(X)}}} = \theta m_f'(\theta) - m_f(\theta),
	\end{equation*}
	which is at the basis of the well known Herbst argument to prove concentration inequalities (see \cite{boucheron2013concentration}). Thus dividing by $M_f(\theta)$ in \eqref{eq:mgf_bootstrap_proof} yields
	\begin{equation*}
		\abs{m_{f}'(\theta)} \leq A \abs{\theta} \eta^{-1} m_g(\eta) + \left( B \abs{\theta} + A \theta^2 \eta^{-1} \right) m_f'(\theta) + \left( C - A \right) \abs{\theta} m_f(\theta).
	\end{equation*}
	Since $e^{m_f}$ is a convex function, which at $\theta = 0$ takes value $1$ and derivative $\bE \sbra{f} = 0$, $m_f'(\theta)$ has the sign of $\theta$ and $m_f(\theta) \geq 0$ for all $\theta \in \bR$. In particular the term $- A \abs{\theta} m_f(\theta)$ can be neglected, and rearranging the terms we obtain that for all $\theta \geq 0$,
    \begin{align*}
        m_{f}'(\theta) &\leq \frac{\left[C + A m_{g}(\eta) / \eta \right] \theta}{1 - B \theta - A \theta^2 / \eta}
    \end{align*}
    provided the denominator is positive. As $m_f(0)= 0$ the upper tail follows from integrating
    \begin{equation*}
        \int_{0}^{\theta} \frac{u \, du}{1- b u - c u^2 } \leq \int_{0}^{\theta} \frac{u \, du}{1- b \theta - c \theta^2}\ \leq \frac{\theta^2}{2(1-b\theta - c \theta^2)}
    \end{equation*}
    for any $b,c > 0$ such that $b \theta + c \theta^2 < 1$. For the lower tail, consider $\theta \leq 0$. Then $m'_f(\theta) \leq 0$ can be neglected, and repeating the same arguments yields
	\begin{align*}
		\abs{m'_f(\theta)} &\leq \left(C + A m_{g}(\eta) / \eta \right) \abs{\theta} 
	\end{align*}
	and 
	\begin{equation*}
		\abs{m_f(\theta)} \leq \left(C + A m_{g}(\eta) / \eta \right) \theta^2 / 2.
	\end{equation*}
\end{proof}

Now comes our main contribution for getting concentration inequalities when $f$ expresses as a polynomial with non-negative coefficients. By bounding only the $F(X,X')$ factor in $\Delta(X)$ by a constant while averaging the remaining difference, it is possible to get un upper bound involving essentially $f$ and its derivatives. If the degree is $d=1$, these derivatives would be constant so the previous lemma applies with $A = 0$, giving eventually Proposition 1.1 of \cite{chatterjee2007stein}. However if $d > 1$, one gets an additional term of degree $d-1$, but since this is also a polynomial of degree $d-1$ very much related to $f$, there is good hope that it satisfies the same concentration as $f$. All in all the idea is thus simply to combine the previous lemma with an induction on $d$. 

\subsection{Upper bound on the m.g.f.}

We now apply the strategy described in the previous section to the case of polynomial functions on the symmetric group. We need first to construct the exchangeable pair and find the function $F$. As explained in \cite{chatterjee2005thesis, chatterjee2007concentration}, defining an exchangeable pair $(X,X')$ on $\cX \times \cX$ is equivalent to considering a reversible Markov kernel $P$ on $\cX \times \cX$ defined by $Ph(x) = \bE \cond{h(X')}{X = x}$ for all function $h$. A very nice outcome of this is that it also provides a generic formula for $F$, which often allows to completely avoid the explicit computation of $F$ if good estimates are known on the mixing time of the Markov chain. 
Indeed, given a function $f$ of zero mean, the antisymmetric function $F$ can be obtained as $F(X,X') = g(X) - g(X')$ where $g$ satisfies the Poisson equation 
\begin{equation*}
	g - Pg = f.
\end{equation*}
This can generally be constructed as $g = \sum_{k \geq 0} P^{k} f$. If $\cX$ is a finite group and $P$ is the kernel of an ergodic random walk on $\cX$ which puts constant mass on conjugacy classes, then the previous infinite sum converges and  \cite{chatterjee2007concentration}[Theorem 1.2] provides a concentration result in terms of what is essentially the mixing time of the random walk. 

On the symmetric group, one obvious candidate of a Markov kernel is that of random transpositions. The corresponding exchangeable pair is $(\sigma, \sigma \tau)$ where $\tau = (I J)$ is a random transposition with $I, J$ uniform and independent in $[n]$ ($\tau$ can be the identity). The random transposition chain puts constant mass on the identity and on transpositions, which form conjugacy classes, and its mixing time has been completely determined by the work of Diaconis and Shahshahani \cite{diaconis1981generating}. Adapting arguments of \cite{chatterjee2005thesis}[Chapt. 4], we arrive at the following lemma, proved in Section \ref{subsubsection:pf_as_bound_F}.

\begin{lemma}\label{lem:as_bound_F}
	Let $f: \frS_n \rightarrow \bR$ have zero mean. There exists a function $F: \frS_{n} \times \frS_{n} \rightarrow \bR$ such that $\bE \cond{F(\sigma, \sigma \tau)}{\sigma} = f(\sigma)$. Furthermore, if $C \geq 0$ is a constant such that $\abs{ f(\sigma) - f(\sigma \tau)} \leq C$ for all $\sigma \in \frS_n$ and transposition $\tau$ then
	\begin{equation}\label{eq:as_bound_F}
		\abs{F(\sigma, \sigma \tau)} \leq \frac{C n}{2} \left( \log \left( \frac{24 \norm{f}_{\infty} n}{C} \right) + \frac{(2/n)(2-e^{-2/n})}{1-e^{-2/n}} \right).
	\end{equation}
	If $f$ has degree $1$, $F(\sigma, \sigma') = (n/2) (f(\sigma) - f(\sigma'))$, so one has actually $\abs{F(\sigma, \sigma \tau)} \leq Cn /2$.
\end{lemma}

From the previous, upper bounding $\Delta(\sigma)$ essentially comes down to bounding quantities involving $\phi$ only. These are gathered in the following lemmas, proved in Section \ref{subsubsection:pf_tensor}. 

\begin{lemma}\label{lem:as_bound_phi}
	Let $\phi \in \frF$. A.s.
	\begin{equation*}
		\abs{\phi(\sigma) - \phi(\sigma \tau)} \leq \left\{ \begin{array}{l l}
			2 \norm{\nabla \phi}_{\infty} & \text{if $d =1$} \\
			6 \norm{\nabla \phi}_{\infty} & \text{if $d \geq 2$} \end{array} \right. 
	\end{equation*}
\end{lemma}

\begin{lemma}\label{lem:conditional_F}
	Let $\phi \in \frF_{=d}$. Then 
	\begin{enumerate}[label=(\roman*)]
		\item \label{enum:F} \begin{equation*}
			\frac{n}{2d} \bE \cond{ \phi(\sigma) - \phi(\sigma \tau)}{\sigma} = \left( 1 - \frac{d-1}{2n} \right) \phi(\sigma) - D\phi(\sigma) - U \phi(\sigma).
		\end{equation*}
		\item \label{enum:expec_bound}  \begin{equation*}
			\frac{n}{2d} \bE \cond{ \abs{\phi(\sigma) - \phi (\sigma \tau)}}{\sigma} \leq  \left( 1 - \frac{d-1}{2n} \right) \phi(\sigma) + D\phi(\sigma) + U \phi(\sigma).
		\end{equation*}
		\item \label{enum:Dphi} $D\phi \in \frF_{=(d-1)}$ satisfies 
		\begin{equation*}
			\bE \sbra{D \phi} = \left( 1 - \frac{d-1}{n} \right) \bE \sbra{\phi}.
		\end{equation*}
	\end{enumerate}
\end{lemma}

\begin{remark}\label{rk:representation_theory}
	The degree $1$ case is made much simpler as in this case the function $\phi - \bE \sbra{\phi}$ is actually an eigenfunction of the random transposition kernel. In general, decomposing the function into a basis of eigenfunctions or using representation theory can provide a neat expression of the function $F$ in Lemma \ref{lem:as_bound_F}, but is not clear how to relate the projections onto eigenspaces to the hypotheses made on $\phi$, in particular the non-negativity of the coefficients. This seems however to be the good strategy if one wants to get rid of the log factor, and could provide further improvements in the proof of Theorem \ref{thm:tensor_concentration}, allowing perhaps to get rid of the consideration of the operators $D$ and $U$. For instance, it is always possible to replace the function $\phi$ by another representative $\psi \in \frF$, which yields the same function on $\frS_n$ but has the property that $D \phi = 0$. The issue is of course that we lose the non-negativity of the coefficients, which seems essential to get a self-bounding property like in the lemma. 
	Note that writing $\phi(\sigma) = \tr(A S^{\otimes d})$ can already be seen as the use of a specific representation of the symmetric group, the $d$-fold tensor product of the standard representation (by permutation matrices).
\end{remark}

\begin{proof}[Proof of Theorem \ref{thm:tensor_concentration}]
	Let $\phi \in \frF_{d}$ satisfy the assumptions of Theorem \ref{thm:tensor_concentration}, $f := \phi - \bE \sbra{\phi}$ and $m_{f}(\theta) := \log \bE \sbra{e^{\theta f(\sigma)}}$. If $d = 1$ let $\alpha_1 = 2 C'_D$, otherwise consider
	\begin{equation}\label{eq:concentration_alpha}
		\alpha_d := 6 d C'_D \left( \log \left( \frac{4 C_D n}{C'_D} \right)^{+} + \frac{(2/n)(2-e^{-2/n})}{1-e^{-2/n}} \right).
	\end{equation}
	so the parameters $\beta_{\phi}, \gamma_{\phi}$ of the theorem are simply equal to
	\begin{equation}\label{eq:concentration_betagamma}
		\beta_d := \frac{3}{2} \alpha_d, \qquad \gamma_d := \frac{3}{5} \alpha_d (2 \bE \sbra{\phi} + C_U).
	\end{equation}
	Thus Theorem \ref{thm:tensor_concentration} will be a direct consequence of the following claim: for all $\theta \in [0, 1 / \beta_{d})$
	\begin{equation}\label{eq:mgf_bound}
		\quad m_f(\theta) \leq \frac{\gamma_d \theta^{2}}{2(1 - \beta_d \ \theta)}.
	\end{equation}
	and for all $\theta \in (-1/\beta_d, 0]$,
	\begin{equation}\label{eq:mgf_lower_bound}
		\quad m_f(\theta) \leq \frac{\gamma_d \theta^{2}}{2}.
	\end{equation}
	Indeed, taking $\theta := t / (\gamma_d + \beta_d t)$, Chernoff's inequality and \eqref{eq:mgf_bound} give for all $t \geq 0$
	\begin{equation*}
		\bP \sbra{\phi(\sigma) - \bE \sbra{\phi(\sigma)} \geq t} \leq e^{- \theta t + m_f (\theta)} \leq e^{\frac{-t^2}{2 (\gamma_d + \beta_d t)}}.
	\end{equation*}
	Taking $\theta = - t / \gamma_d$ and \eqref{eq:mgf_lower_bound} yields the lower tail.

	The rest of the proof is devoted to establishing the claim. The case $d=1$ is proved in \cite{chatterjee2005thesis} but is also recovered from the following arguments.
	
	Let $\Delta(\sigma) := 1/2 \ \bE \cond{\abs{F(\sigma, \sigma \tau)} \abs{f(\sigma) - f(\sigma \tau)}}{\sigma}$, where $F$ is the function of Lemma \ref{lem:as_bound_F}. Assumption \eqref{eq:hyp_norms_DU} and Lemma \ref{lem:as_bound_phi} give an upper bound on the constant $C$ appearing in \eqref{eq:as_bound_F}, which with the definition of $\alpha_d$ \eqref{eq:concentration_alpha} shows 
	\begin{equation*}
		\abs{F(\sigma,\sigma \tau)} \leq \frac{n}{2d} \alpha_d.
	\end{equation*}
	On the other hand, decompose $\phi =  \sum_{l = 0}^{d} \phi_l$ as a sum of homogeneous functions with non-negative coefficients. Without loss of generality we can suppose that $\phi_0 = 0$. Point \ref{enum:expec_bound} in Lemma \ref{lem:conditional_F} gives the bound in (conditional) expectation
	\begin{align*}
		\frac{n}{2d} \bE \cond{\abs{\phi(\sigma) - \phi(\sigma \tau)}}{\sigma} &\leq \sum_{l=1}^{d} \frac{n}{2 l}  \bE \cond{\abs{\phi_{l}(\sigma) - \phi_{l}(\sigma \tau)}}{\sigma} \\
		&\leq \sum_{l=1}^{d} \phi_{l}(\sigma) + D \phi_l (\sigma) + U \phi_l (\sigma)  \\
		&= \phi(\sigma) + D \phi(\sigma) + U \phi(\sigma).
	\end{align*}
	By assumption \eqref{eq:hyp_norms_DU}, $U \phi(\sigma) \leq C_U$ and Point \ref{enum:Dphi} of Lemma \ref{lem:conditional_F} shows $\bE \sbra{D \phi} \leq \bE \sbra{\phi}$. Hence letting $g := D \phi - \bE \sbra{D \phi}$, one has
	\begin{equation*}
		\Delta(\sigma) \leq \frac{\alpha_d}{2} \left( f(\sigma) + g(\sigma) + 2 \bE \sbra{\phi} + C_U \right),
	\end{equation*}
	which is exactly the kind of upper bound needed to apply Lemma \ref{lem:bootstrap_mgf}. The latter shows for all $\theta \geq 0, \eta > 0$,
	\begin{equation*}
		m_{f}(\theta) \leq \frac{\frac{\alpha_d}{2}(2 \bE \sbra{\phi} + C_U  + \eta^{-1} m_{g}(\eta))\theta^2}{2 \left(1 - \frac{\alpha_d}{2} \theta - \frac{\alpha_d}{2} \eta^{-1} \theta^2 \right)}
	\end{equation*}
	provided the denominator is defined.
	In the case $d=1$, $g = 0$ and $C_U = 0$ and with the value of $\alpha_1 = 2 C'_D$ the previous inequality becomes 
	\begin{equation*}
		m_{f}(\theta) \leq \frac{C'_D \bE \sbra{\phi} \theta^2}{1 - C'_D \theta}.
	\end{equation*}
	This gives the claim for $d=1$, with actually better constants (those of Proposition 1.1 in \cite{chatterjee2007stein}) than those given by \eqref{eq:concentration_betagamma}. 
	
	Suppose now $d > 1$ and the claim holds for $d-1$. Then by assumption the function $g$ satisfies the same properties as $f$ and can thus be applied the induction hypothesis: 
	\begin{equation*}
		m_g(\eta) \leq \frac{\gamma_{d-1} \eta^2}{2(1 - \beta_{d-1} \eta)}
	\end{equation*}
	with $\beta_{d-1}, \gamma_{d-1}$ as in \eqref{eq:concentration_betagamma}.
	Taking $\eta := (2 \beta_{d-1})^{-1}$ allows to bound 
    \begin{equation*}
        \eta^{-1} m_{g}(\eta) \leq \frac{\gamma_{d-1} \eta}{2(1 - \beta_{d-1} \eta)} = \frac{\gamma_{d-1}}{2 \beta_{d-1}} = \frac{2 \bE \sbra{\phi} + C_U}{5}.
    \end{equation*}
	If $\theta < \sqrt{\eta / \alpha_d}$ we can bound $\eta^{-1} \theta \leq \sqrt{2 \beta_{d-1} / \alpha_{d}} \leq \sqrt{2 \beta_d / \alpha_d} = \sqrt{3}$ to arrive at 
	\begin{equation*}
		m_{f}(\theta) \leq \frac{\frac{3 \alpha_d}{5} (2 \bE \sbra{\phi} + C_U) \theta^2}{2 \left( 1 - \frac{\alpha_d}{2} (1+\sqrt{3})  \theta \right)} \leq \frac{\gamma_d \theta^2}{2 \left( 1- \beta_d \theta \right)},
	\end{equation*}
	which proves the upper bound of the claim for $d$. Similar reasoning applies for the lower bound.
\end{proof}

\subsubsection[Almost sure bound on $F$]{Almost sure bound on $F$: proof of Lemma \ref{lem:as_bound_F}}\label{subsubsection:pf_as_bound_F}

Let $\mu$ be the probability measure on $\frS_n$ which puts mass $1/n$ on the identity and $2/n^2$ on every transposition. The random transposition Markov chain is the random walk on $\frS_n$ defined by i.i.d. increments of law $\mu$. Since these are symmetric the uniform distribution $\mathrm{unif}$ is stationary. The Markov chain is ergodic and its mixing properties have been investigated in \cite{diaconis1981generating}. In particular, it was proved that for all $k \geq 0$,
\begin{equation}\label{eq:TV_RT}
	\TV{\mu^{\ast k} - \mathrm{unif}} \leq 6 n e^{-2k/n}.
\end{equation}
Let $P$ denote the transition matrix of random transpositions and $f$ be a function on $\frS_n$ with zero mean under the uniform measure. Then the function $F$ given by 
\begin{equation*}
	F(\sigma, \sigma') := \sum_{k \geq 0} \left(P^{k}f(\sigma) - P^{k}f(\sigma')\right)
\end{equation*}
is well defined by the total variation convergence above and satisfies $\bE \cond{F(\sigma, \sigma')}{\sigma} = f(\sigma)$. We refer to \cite{chatterjee2005thesis,chatterjee2007concentration} for details. 

Lemma \ref{lem:as_bound_F} is obtained by bounding $F$ in two ways. On the one hand \eqref{eq:TV_RT} implies
\begin{equation*}
	\abs{P^{k}f(\sigma)} = \abs{P^{k}f(\sigma) - \bE \sbra{f}} \leq 12 \norm{f}_{\infty} n e^{-2k/n}
\end{equation*}
and thus
\begin{equation*}
	\abs{P^{k}f(\sigma) - P^{k} f(\sigma \tau)} \leq 24 \norm{f}_{\infty} n e^{-2 k /n}.
\end{equation*}
The second bound is based on the fact that $\mu$ puts constant mass on conjugacy classes. As observed by Chatterjee in \cite{chatterjee2005thesis, chatterjee2007concentration}, this implies that
\begin{align*}
	\abs{P^{k} f(\sigma) - P^{k} f(\sigma \tau)} \leq \max_{\sigma', \tau'} \abs{f(\sigma') - f(\sigma' \tau')} \leq C
\end{align*}
for all $k \geq 0$, $\sigma \in \frS_n$ and transposition $\tau$. Combine the two bounds as
\begin{align*}
	\abs{F(\sigma, \sigma \tau)} &\leq \sum_{k \geq 0} \min \left( C, 24 \norm{f}_{\infty} n e^{-2 k /n} \right) \\
	&\leq C \sum_{k \geq 0} \min \left( 1, 24 \ C^{-1} \norm{f}_{\infty} n e^{-2 k /n} \right) \\
	&\leq C \left( \frac{n}{2} \log \left( 24 C^{-1} \norm{f}_{\infty} n \right) + 1 + \frac{1}{1-e^{-2/n}} .\right)
\end{align*}

\subsubsection{Tensor representation}\label{subsubsection:pf_tensor}

We now prove Lemmas \ref{lem:as_bound_phi} and \ref{lem:conditional_F}.

\begin{lemma}\label{lem:phi_sigma_tau}
	Let $\phi \in \frF$ and $\tau = (I J)$. For all $\sigma \in \frS_n$,
	\begin{equation}\label{eq:phi_sigma_tau}
		\begin{split}\phi(\sigma \tau) = \ &\phi(\sigma) + \partial_{I \sigma(J)} \phi(\sigma) + \partial_{J \sigma(I)} \phi(\sigma) - \partial_{I \sigma(I)} \phi(\sigma) - \partial_{J \sigma(J)} \phi(\sigma) \\ 
		&+ 2 \partial_{I \sigma(I)} \partial_{J \sigma(J)} \phi(\sigma) + 2 \partial_{I \sigma(J)} \partial_{J \sigma(I)} \phi(\sigma).
		\end{split}
	\end{equation}
\end{lemma}

\begin{proof}
		
	Let $\phi \in \frF$. Decomposing $\phi$ into homogeneous components, it suffices to consider the case of a homogeneous function. Suppose therefore that $\phi$ is homogeneous of degree $d \geq 1$. It can be realized as $\phi(\sigma) = \tr(A S^{\otimes d})$ for some $A \in M_{n}(\bR_+)^{\otimes d}$. We start with a simple computation relating derivatives of $\phi$ with the tensor $A$. 

	Let $E_{ij}$ denote the matrix which has entry $(i,j)$ equal to $1$ and all other entries equal to $0$. For all $M \in M_{n}(\bR)$, $i, j \in [n]$, 
		\begin{equation}\label{eq:derivatives}
			\partial_{ij} \phi(M) = \sum_{k=1}^{d} \tr(A \ M^{\otimes (k-1)} \otimes E_{ij} \otimes M^{\otimes (d-k)}).
		\end{equation}
		Indeed, for all $t \in \bR$, expanding the tensor product yields
		\begin{equation*}
			(M + t E_{ij})^{\otimes d} = M^{\otimes d} + t \sum_{k=1}^{d} M^{\otimes (k-1)} \otimes E_{ij} \otimes M^{\otimes (d-k)} + O(t^2)
		\end{equation*}
		Dividing by $t \neq 0$ and taking the limit $t \rightarrow 0$ gives the result. 

		For the sequel, we make use of the multilinearity of $\phi$. Given multi-indices $\bfi, \bfj \in [n]^{d}$, write $A_{\bfi, \bfj} = A_{\substack{i_1 \cdots i_d \\ j_1 \cdots j_d}}$. The multilinearity implies that we can suppose $A_{\bfi, \bfj} = 0$ whenever $\bfi$ or $\bfj$ has two identical coordinates. Consequently the computation of $\tr(AM)$ does not depend either on the entries $M_{\bfi \bfj}$ when $\bfi$ or $\bfj$ has some identical coordinates. More precisely, the kernel of the linear map $M \mapsto \tr(AM)$ contains the subspace $H$ of tensors whose only non-zero entries are such multi-indices. Therefore when computing $\tr(AM)$, one can freely replace $M$ with any of its representative modulo $H$, which allows in particular to get rid of potential dependency properties between entries of $S$.
		The permutation matrix of $\tau$ is $T = I + M_{IJ}$ with 
	\begin{equation*}
		M_{IJ} :=  E_{IJ} + E_{JI} - E_{II} - E_{JJ}.
	\end{equation*}
	Note that if we write products of permutation from right to left, so $\sigma \tau$ applies $\tau$ first and then $\sigma$, the permutation matrix of $\sigma \tau$ is $T S$, hence $\phi(\sigma \tau) = \tr(A T^{\otimes d} S^{\otimes d})$. By expanding the tensor product,
	\begin{align*}
		T^{\otimes d} &= I + \sum_{k=1}^{d} I ^{\otimes (k-1)} \otimes M_{IJ} \otimes I^{\otimes (d-k)} \\
		&+ \sum_{k_1 < k_2 \in [d]} I^{\otimes (k_1 - 1)} \otimes {M_{IJ}} \otimes I^{\otimes (k_2-k_1)} \otimes M_{IJ} \otimes I^{\otimes (d-k_2)} \mod H.
	\end{align*}
	Expanding the expression of $M_{IJ}$ in the second sum yields terms like $E_{IJ} \otimes E_{II}$ which are also in $H$ and can be discarded. 	
	Now by \eqref{eq:derivatives}, 
	\begin{equation*}
		\sum_{k=1}^{d} \tr(A (I^{\otimes (k-1)} \otimes E_{I J} \otimes I^{\otimes (d-k)}) S^{\otimes d}) = \partial_{I \sigma(J)} \phi(\sigma)
	\end{equation*}
	and a similar observation can be made for order $2$ derivatives. This proves Lemma \ref{lem:phi_sigma_tau}.
	\end{proof}

Lemma \ref{lem:as_bound_phi} can be deduced easily, provided one can control second order derivatives. This requires no additional assumption, for taking partial derivates can only give a smaller function, as proved by the following lemma.

\begin{lemma}\label{lem:further_derivatives}
	Let $\phi \in \frF$. For all $k \geq 1$ and $\bfi, \bfj \in [n]^{k}$,
	\begin{equation}\label{eq:bound_derivatives}
		\norm{\partial^{k}_{\bfi \bfj} \phi}_{\infty} \leq \norm{\phi}_{\infty}.
	\end{equation}
\end{lemma}

\begin{proof}
	The general case follows from the $k=1$ case by an easy induction. Let $i,j \in [n]$. Note that by multilinearity, the partial derivative $\partial_{ij} \phi$ cannot contain any indeterminate $X_{ik}$ or $X_{kj}, k \in [n]$, so $\partial_{ij} \phi(\sigma) = \partial_{ij} \phi((\sigma(i) j ) \sigma)$. Hence the maximum is always realized for a permutation $\sigma$ such that $\sigma(i) = j$, but then for such permutations $\phi$ actually coincides with the partial derivative $\partial_{ij} \phi$. Consequently
	\begin{align*}
		\max_{\sigma \in \frS_n} \abs{\partial_{ij} \phi(\sigma)} &= \max_{\sigma: \sigma(i) =j} \abs{\partial_{ij} \phi(\sigma)} \\
		&\leq \max_{\sigma: \sigma(i) =j} \abs{\phi(\sigma)} \\
		&\leq \norm{\phi}_{\infty}.
	\end{align*}
\end{proof}

\begin{proof}[Proof of Lemma \ref{lem:as_bound_phi}]
	If $\phi$ is assumed to have non-negative coefficients, \eqref{eq:phi_sigma_tau} shows that
\begin{align*}
	\abs {\phi(\sigma) - \phi(\sigma \tau)} &\leq \abs{ \partial_{I \sigma(J)} \phi(\sigma) + \partial_{J \sigma(I)} \phi(\sigma) - \partial_{I \sigma(I)} \phi(\sigma) - \partial_{J \sigma(J)} \phi(\sigma) } \\
	&+ 2 \partial_{I \sigma(I)} \partial_{J \sigma(J)} \phi(\sigma) + 2 \partial_{I \sigma(J)} \partial_{J \sigma(I)} \phi(\sigma) \\
	&\leq 2 \norm{\nabla \phi}_{\infty} + 4 \max_{i,j,k,l \in [n]} \abs{\partial^{2}_{ij,kl} \phi(S)},
\end{align*}
which establishes the result thanks to the previous lemma.
\end{proof}

\begin{proof}[Proof of Lemma \ref{lem:conditional_F}]
 	Restricting to a homogeneous function $\phi \in \frF_{=d}$, \eqref{eq:phi_sigma_tau} gives by averaging over $I,J$
\begin{align*}
	\frac{n}{2d} \bE \sbra{\phi(\sigma) - \phi(\sigma \tau) \ | \ \sigma} &= \frac{1}{2 d n} \sum_{i,j \in [n]} \left( \partial_{i \sigma(i)} \phi(\sigma) + \partial_{j \sigma(j)} \phi(\sigma) - \partial_{i \sigma(j)} \phi(\sigma) - \partial_{j \sigma(i)} \phi(\sigma) \right. \\
	 &\left. \qquad - 2 \partial_{i \sigma(i)} \partial_{j \sigma(j)} \phi(\sigma) - 2 \partial_{i \sigma(j)} \partial_{j \sigma(i)} \phi(\sigma) \right) \\
	 \begin{split} &= \left( 1- \frac{d-1}{2n} \right)\phi(\sigma) -\frac{1}{dn} \sum_{i,j \in [n]} \partial_{i \sigma(j)} \phi(\sigma) \\ &\qquad - \frac{1}{dn} \sum_{i,j \in [n]} \partial_{i \sigma(j)} \partial_{j \sigma(i)} \phi(\sigma), \end{split}
\end{align*}
which gives Point \ref{enum:F} of the Lemma. The second equality arises from the relation \eqref{eq:phi_homogeneous}.

The bound in absolute value \ref{enum:expec_bound} is obtained similarly, using first triangle inequality in \eqref{eq:phi_sigma_tau}. Finally Point \ref{enum:Dphi} is proved easily.
\end{proof}

\subsection*{Acknowledgments}
    This work was achieved while the author was at INRIA, D\'{e}partement d'Informatique of École Normale Supérieure, PSL Research University and I2M, Aix-Marseille Université. Part of this work was also supported by the KTH Royal Institute of Technology. The author is deeply grateful to Charles Bordenave and Laurent Massouli\'{e} for many insightful discussions and encouragement. Finally, the author would like to thank the anonymous referee for their thorough reading, which helped correct several mistakes and significantly improved the manuscript.

\bibliographystyle{plainurl}
\bibliography{bibliography.bib}

\end{document}